\documentclass[12pt, reqno]{amsart}
\usepackage{graphicx, amssymb, amsthm, amsfonts, latexsym, epsfig, amsmath, amscd, amsbsy}
\usepackage[usenames]{color}
\usepackage{tikz, caption, subcaption}
\usepackage{enumerate, enumitem}
\setlist[itemize, 2]{label=$\circ$}
\usepackage{epstopdf}

\usetikzlibrary{decorations.pathmorphing}
\usetikzlibrary{calc}

\usetikzlibrary{arrows, arrows.meta, patterns, calc}
\pgfdeclarelayer{bg}
\pgfsetlayers{bg,main}

\pgfdeclarelayer{bg}
\pgfsetlayers{bg,main}

\numberwithin{equation}{section}
\theoremstyle{plain}
\newtheorem{theorem}{Theorem}[section]
\newtheorem{lem}[theorem]{Lemma}
\newtheorem{cor}[theorem]{Corollary}
\newtheorem{prop}[theorem]{Proposition}
\theoremstyle{definition}

\newtheorem{rem}[theorem]{Remark}

\def\eps{\varepsilon}

\renewcommand{\phi}{\varphi}

\def\Int{\mathop\mathrm{Int}}
\def\Cl{\mathop\mathrm{Cl}}
\def\dist{\mathop\mathrm{dist}}

\def\Conv{\mathop\mathrm{Conv}}

\usepackage{hyperref}
\hypersetup{
    colorlinks=true,
    citecolor=violet,
    pdftitle={Zero Entropy and Basin of Attraction for Lozi Maps},
    pdfauthor={Kristijan Kilassa Kvaternik}
    }

\begin{document}
	
	\title{Basins of attraction and escape in the Lozi map}
	\author{Kristijan Kilassa Kvaternik}
	\address[K. Kilassa Kvaternik]{
		University of Zagreb,
		Faculty of Electrical Engineering and Computing,
		Department of Applied Mathematics,
		Unska 3,
		10\,000 Zagreb,
		Croatia -- and -- Jagiellonian University,
		Faculty of Mathematics and Computer Science,
		ul.\ prof.\ Stanis{\l}awa {\L}ojasiewicza 6,
		30-348 Kraków,
		Poland}
	\urladdr{\url{https://www.fer.unizg.hr/en/kristijan.kilassa_kvaternik}}
	\email{kristijan.kilassakvaternik@fer.unizg.hr}
	\thanks{This work was supported in part by the Croatian Science Foundation grants MOBODL-2023-08-4960, IP-2022-10-9820 GLODS, UIP-2025-02-4309 ToGeCoP, and in part by the Horizon grant 101183111-DSYREKI-HORIZON-MSCA-2023-SE-01.}
		
	\date{\today}
	
	\subjclass[2020]{37B35, 37D10, 37B20, 37E30}
	\keywords{Lozi map, piecewise-affine dynamical systems, basin of attraction, omega-limit set, attractor, invariant manifolds}
	
	\begin{abstract}
		
		For the Lozi map $L_{a,b}$, we consider parameter pairs for which the fixed point $X$ in the first quadrant has no homoclinic points and the period-two orbit $\{P,P'\}$ is attracting. For such parameters, let $\ell$ denote the set of accumulation points of the unstable manifold $W_X^u$ that do not belong to $W_X^u$. We completely classify the forward asymptotic behavior of points in the phase space. The forward orbit of every point in the plane either converges to $X$, to the other fixed point $Y$ in the third quadrant, or to $\ell$, or it escapes to infinity. The global phase space is organized by the stable manifolds of the fixed points: $W_Y^s$ separates the basin of $\ell$ from the region of escaping orbits, while $W_X^s$ is the exceptional set of points whose orbits converge to $X$. In particular, if $\mathcal{A}_1$ denotes the component of $\mathbb{R}^2 \setminus W_Y^s$ containing $X$, then $\mathcal{A}_1 \setminus W_X^s$ is precisely the basin of attraction of $\ell$.
		
	\end{abstract}
	
	\maketitle	
	
	\baselineskip=18pt
	
	\section{Introduction}\label{sec:intro}
	
	The Lozi map family is one of the canonical models in planar dynamics. Defined as a two-parameter family of homeomorphisms
	\begin{equation*}
	L_{a,b}\colon \mathbb{R}^2 \rightarrow \mathbb{R}^2, \quad L_{a,b}(x,y) = (1 + y - a|x|,bx),
	\end{equation*}
its piecewise-affine nature makes it an appealing setting for studying a variety of dynamical phenomena, while also providing valuable insights into other families, such as border collision normal forms and the H\'{e}non family. Since its introduction in 1978 \cite{lozi1978attracteur}, numerous techniques have been employed to characterize its dynamics; nevertheless, many questions concerning its dynamical properties remain open to this day.

	In this paper, we completely characterize the forward asymptotic behavior of orbits in a parameter set that arose from the study of the zero entropy locus. In \cite{misiurewicz2024zero}, Misiurewicz and \v{S}timac studied parameters $(a,b) \in \mathfrak{R}$; a formal definition of $\mathfrak{R}$ is given at the beginning of Section \ref{sec:basin_attraction}. On this parameter set, there is an attracting period-two cycle $\{P,P'\}$ and there are no homoclinic points for the fixed point $X$ in the first quadrant; in other words, there are no points in the plane whose backward and forward orbits both tend to $X$, apart from $X$ itself (see Figure \ref{fig:Lozi_stable_unstable_Z_V}). On a distinguished subset of $\mathfrak{R}$, the authors showed that the branches of $W_X^u$, the unstable manifold of $X$, accumulate on the set $\ell = \{P,P'\}$.
	
	In a recent paper \cite{kilassa2026accumulation}, the author of this work extended previous results to the whole of $\mathfrak{R}$, where the accumulation set $\ell$ of $W_X^u$ can, in general, have a more intricate structure. By studying the non-wandering set of the second iterate of the Lozi map, the author showed that $\Omega(L_{a,b}^2) \subseteq \ell \cup \{X,Y\}$ for every pair $(a,b) \in \mathfrak{R}$, where $Y$ is the other fixed point of $L_{a,b}$ in the third quadrant. Moreover, the connected components of $\ell$ were characterized as intersections of nested images of a specially constructed polygon $\mathcal{D}$, which is an eventually trapping region for $\ell$. This construction is revisited in Subsection \ref{subsec:accum_set_ell}.
	
	In view of these findings, $\ell$ behaves as an attractor for $L_{a,b}$. The main objective of this work is therefore to investigate the dynamical role of $\ell$ as an attractor by providing a complete description of its basin of attraction. We show that the boundary of this basin is formed by $W_Y^s$, the stable manifold of $Y$, and we use the geometry of this invariant manifold to characterize the global dynamics in the surrounding regions of the plane.
	
	More precisely, a careful geometric analysis shows that $W_Y^s$ separates the plane (Corollary \ref{cor:WsY_separates}). The two connected components $\mathcal{A}_1$, $\mathcal{A}_2$ of the complement $\mathbb{R}^2 \setminus W_Y^s$ are both forward and backward invariant, and they capture all the qualitative behavior of the forward dynamics of $L_{a,b}$. First, every forward orbit in $\mathcal{A}_2$ escapes to infinity; see Figure \ref{figure:thm_A2_infinity} for a visual depiction of the escaping mechanism. Second, we show that a forward orbit can accumulate at a fixed point of $L_{a,b}$ only if that orbit belongs to the stable manifold of that fixed point. Since $W_X^s$ is contained in $\mathcal{A}_1$, we ultimately show in Theorem \ref{thm:A_1_basin_ell} that
	\begin{equation*}
	\mathop{\mathrm{Basin}}(\ell) = \mathcal{A}_1 \setminus W_X^s.
	\end{equation*}
In other words, all points in the remaining set $\mathcal{A}_1 \setminus W_X^{s}$ converge to $\ell$ under forward iteration of $L_{a,b}$.
	
	Thus, we completely classify the asymptotic behavior of every forward orbit in the phase space: each orbit either converges to one of the fixed points $X$ or $Y$, converges to the set $\ell$, or it escapes to infinity. In addition, we provide an explicit geometric decomposition of the phase space into regions corresponding to each of these asymptotic behaviors. 
	
	The interplay between the fixed points $X$ and $Y$ provides a useful way to interpret this global picture. The unstable manifold of $X$ is responsible for the complexity of the system, as it generates the accumulation set $\ell$, which captures the non-wandering dynamics that may induce positive topological entropy, as shown in \cite{kilassa2026accumulation}. On the other hand, the stable manifold of $Y$ organizes the phase space by forming the basin boundary and separating the escaping dynamics from the more intricate dynamics associated with $X$. Informally speaking, $X$ dictates the complexity of the dynamics, whereas $Y$ sets the limits of that complexity. 
	
	There is also a close analogy with the dynamics arising in the Misiurewicz parameter set, which corresponds to a different parameter regime from the one considered here. For the corresponding strange Lozi attractors, the basin of attraction was determined in \cite{baptista2009basin}, while the basin of the strange attractors for H\'{e}non maps was studied in \cite{cao2000nonwandering}. In both settings, the same interplay between stable and unstable invariant manifolds governs the basin geometry. In the Lozi case, the attractor is generated by the unstable manifold of the fixed point in the first quadrant, while the stable manifold of the other fixed point in the third quadrant forms the boundary of its basin of attraction; an analogous role is played by the corresponding invariant manifolds in the H\'{e}non setting. In the parameter regime considered here, however, the geometry of $W_Y^s$ is more intricate than in \cite{baptista2009basin}, requiring a more detailed technical analysis to determine the basin; see Remark \ref{rem:first_tangency_point}.
	
	Still, the internal geometry of $\ell$ remains unresolved. As noted above, the equality $\ell = \{P,P'\}$ holds on a subset of the parameter region $\mathfrak{R}$, making it plausible that this equality might hold throughout $\mathfrak{R}$. However, in a private communication, Y.\ Ishii and D.\ Sands brought to the author's attention examples of parameters suggesting that, in general, $\ell$ could have a more complex structure. It is also interesting to note that, even in such a case, $\ell$ would still not be comparable to the strange attractors arising from the Misiurewicz set; we refer the reader to Section \ref{sec:concluding_rems} for a more detailed discussion. Therefore, further study of the accumulation set $\ell$ could reveal a class of attractors distinct from the classical Lozi attractors, providing new insight into the dynamics of the Lozi family.
	
	This paper is organized as follows. In Section \ref{sec:prelim}, we provide an overview of preliminary notions and results, together with the notation used throughout the paper. Section \ref{sec:basin_attraction} comprises three parts. In Subsection \ref{subsec:accum_set_ell}, we review the construction of the eventually trapping region. We then analyze the geometry of the stable manifold $W_Y^s$ and show that it separates the plane in Subsection \ref{subsec:fixed_point_Y}. Subsection \ref{subsec:approaching_infinity} is dedicated to the analysis of the forward asymptotic behavior of orbits, culminating in the proof of the main theorem, which determines the basin of attraction of $\ell$. The concluding remarks are provided in Section \ref{sec:concluding_rems}.
	
	The author thanks S.\ \v{S}timac for reading the preliminary version of the text and providing constructive advice, as well as J.\ Boro\'{n}ski for helpful discussions on certain proofs. The author also thanks Y.\ Ishii and D.\ Sands for sharing their findings on the potential nontrivial structure of $\ell$.
	
	This paper is based on Chapter 4 of the author's PhD dissertation \cite{kilassa2022tangential}, where preliminary versions of these results first appeared.

	\section{Preliminaries}\label{sec:prelim}
	
	\subsection{Topological dynamics}\label{subsec:dynsys_general_notions}
	
	Let $\mathcal{X}$ be a metric space with a metric $d$ and $f \colon \mathcal{X} \rightarrow \mathcal{X}$ a continuous function. For $n \in \mathbb{N}$, we denote $f^n = f \circ f \circ \ldots \circ f$ ($n$ times). In particular, we let $f^0$ be the identity map. In addition, if $f$ is invertible, we also put $f^{-n} = f^{-1} \circ f^{-1} \circ \ldots \circ f^{-1}$ ($n$ times). We say that a set $\mathcal{Y} \subseteq \mathcal{X}$ is $f$-\emph{invariant} if $f(\mathcal{Y}) \subseteq \mathcal{Y}$.
	
	For a point $x \in \mathcal{X}$, the set $\mathcal{O}^{+}(x,f)=\{f^n(x) \colon n \in \mathbb{N}_0\}$ is called the \emph{forward orbit} of $x$. If $f$ is invertible, we define the \emph{backward orbit} $\mathcal{O}^{-}(x,f) = \{f^{-n}(x) \colon n \in \mathbb{N}_0\}$ and the \emph{full orbit} $\mathcal{O}(x,f) = \mathcal{O}^{+}(x,f) \cup \mathcal{O}^{-}(x,f)$.
	
	We say that a point $x \in \mathcal{X}$ is a \emph{non-wandering point} if for every neighborhood $\mathcal{U}$ of $x$, there exists a positive integer $n$ such that $f^{n}(\mathcal{U})\cap \mathcal{U} \neq \emptyset$. The \emph{non-wandering set} $\Omega(f)$ is the set of all non-wandering points. If $f$ is a homeomorphism, then $\Omega(f)$ is both $f$-invariant and $f^{-1}$-invariant, as well as $\Omega(f)=\Omega(f^{-1})$; see, e.g., Proposition 1.1 in \cite{shub1987global}.
	
	Furthermore, for a point $x \in \mathcal{X}$, the $\omega$-\emph{limit set} of $x$ is defined by
	\begin{equation*}
	\begin{split}
	\omega(x,f) & = \bigcap_{n \in \mathbb{N}} \Cl\{f^k(x) \colon k \geqslant n\} \\ 
				& = \{ y \in \mathcal{X} \colon \exists \text{ a sequence } (n_k)_{k \in \mathbb{N}}\subseteq\mathbb{N} \text{ such that } n_k \rightarrow \infty \text{ and } f^{n_k}(x) \xrightarrow{k \rightarrow \infty} y \},
	\end{split}
	\end{equation*}
	where $\Cl \mathcal{Y}$ is the topological closure of a set $\mathcal{Y} \subseteq \mathcal{X}$. We will use the standard fact that for every $x \in \mathcal{X}$, the set $\omega(x,f)$ is contained in $\Omega(f)$; see, e.g., Proposition 1.3 in \cite{shub1987global}.
	
	For a point $x \in \mathcal{X}$ and a non-empty subset $\mathcal{A} \subseteq \mathcal{X}$, the distance of $x$ to $\mathcal{A}$ is defined as $d(x,\mathcal{A})=\inf_{a \in \mathcal{A}} d(x,a)$. The \emph{basin of attraction} of $\mathcal{A}$ is the set of all $x \in \mathcal{X}$ such that $d(f^n(x),\mathcal{A}) \xrightarrow{n \rightarrow \infty} 0$.
	
	In addition, the \emph{Hausdorff distance} $d_H$ of two non-empty subsets $\mathcal{A},\mathcal{B}\subseteq\mathcal{X}$ is defined as
	\begin{equation*}
	\begin{split}
	d_H(\mathcal{A},\mathcal{B}) & = \max \left\{ \sup_{a \in \mathcal{A}} d(a,\mathcal{B}),\ \sup_{b \in \mathcal{B}} d(b, \mathcal{A}) \right\} \\
								 & = \inf \{ \varepsilon > 0 \colon \mathcal{B} \subseteq \mathcal{A}_{\varepsilon} \text{ and } \mathcal{A} \subseteq \mathcal{B}_{\varepsilon} \},
	\end{split}
	\end{equation*}
	where $\mathcal{A}_{\varepsilon} := \bigcup_{a \in \mathcal{A}}\{x \in \mathcal{X} \colon d(x,a) < \varepsilon\}$, and $\mathcal{B}_\varepsilon$ is defined analogously.
	
	On the collection of all non-empty compact subsets of $\mathcal{X}$, the Hausdorff distance $d_H$ is a metric.
	
	We conclude this subsection by recalling the notion of internal chain transitivity that will be invoked later, before the proof of our main theorem.
	
	A non-empty subset $\mathcal{A} \subseteq \mathcal{X}$ is \emph{internally chain transitive} if for any $a,b \in \mathcal{A}$ and any $\varepsilon > 0$, there exists $n \in \mathbb{N}$ and a finite sequence of points $x_1,x_2,\ldots,x_n \in \mathcal{A}$ such that
	\begin{equation*}
	x_1 = a, x_n = b, \quad \text{and} \quad d(f(x_i),x_{i + 1}) < \varepsilon,\ 1 \leqslant i \leqslant n - 1. 
	\end{equation*}
This sequence of points is called an $\varepsilon$-\emph{chain} from $a$ to $b$, and it is visualized in Figure \ref{figure:epsilon_chain}.

	\begin{figure}[!ht]
	\begin{center}
	\begin{tikzpicture}[auto, scale=1]
			\tikzstyle{nodec}=[draw,circle,fill=black,minimum size=2pt,
			inner sep=0pt, label distance=2mm]
			\tikzstyle{nodeh}=[draw,circle,fill=white,minimum size=4pt,
			inner sep=0pt]
			\tikzstyle{dot}=[circle,draw=none,fill=none,minimum size=0pt,inner sep=2pt, outer sep=-1pt]

			\def\eps{1.5}
			
			\coordinate (x1) at (-7.5,3);
			\coordinate (x2) at (-2.6,3.92);
			\coordinate (x3) at (0,1);
			\node (x4) at (3,3) {$\ldots$};
			\coordinate (x5) at (5.5,2);
			
			\node[nodec, label={[below]$x_1=a$}] at (x1) {};
			\node[nodec, label={[above]$x_2$}] at (x2) {};
			\node[nodec, label={[below]$x_3$}] at (x3) {};
			\node[nodec, label={[right]$x_n=b$}] at (x5) {};
			
			\foreach \i in {1,2,3,5}
				\draw[thick, dashed] (x\i) circle (\eps);
				
			\coordinate (fx1) at ($(x2) + (195:{0.5*\eps})$);
			\coordinate (fx2) at ($(x3) + (135:{0.65*\eps})$);
			\coordinate (fxnm1) at ($(x5) + (70:{0.33*\eps})$);
			
			\node[nodec, label={[below]$f(x_1)$}] at (fx1) {};
			\node[nodec, label={[below]$f(x_2)$}] at (fx2) {};
			\node[nodec, label={[left]$f(x_{n-1})$}] at (fxnm1) {};
			
			\draw[-{Stealth[scale=1]}, shorten <=3pt, shorten >=3pt, thick, bend left=30] (x1) to (fx1);
			\draw[-{Stealth[scale=1]}, shorten <=3pt, shorten >=3pt, thick, bend right=30] (x2) to (fx2);
			\draw[-{Stealth[scale=1]}, shorten <=3pt, thick, bend left=30] (x3) to (x4.west);
			\draw[-{Stealth[scale=1]}, shorten >=3pt, thick, bend left=60] (x4.east) to (fxnm1);
			
			\draw[dotted, thick] (x1) to node[dot, swap]{$\varepsilon$} ($(x1) + (130:\eps)$);
			\draw[dotted, thick] (x2) to node[dot, swap]{$\varepsilon$} ($(x2) + (20:\eps)$);
			\draw[dotted, thick] (x3) to node[dot]{$\varepsilon$} ($(x3) + (-10:\eps)$);
			\draw[dotted, thick] (x5) to node[dot, swap]{$\varepsilon$} ($(x5) + (235:\eps)$);

			\end{tikzpicture}
	\end{center}
	\caption{An $\varepsilon$-chain from $a$ to $b$.}
	\label{figure:epsilon_chain}
	\end{figure}
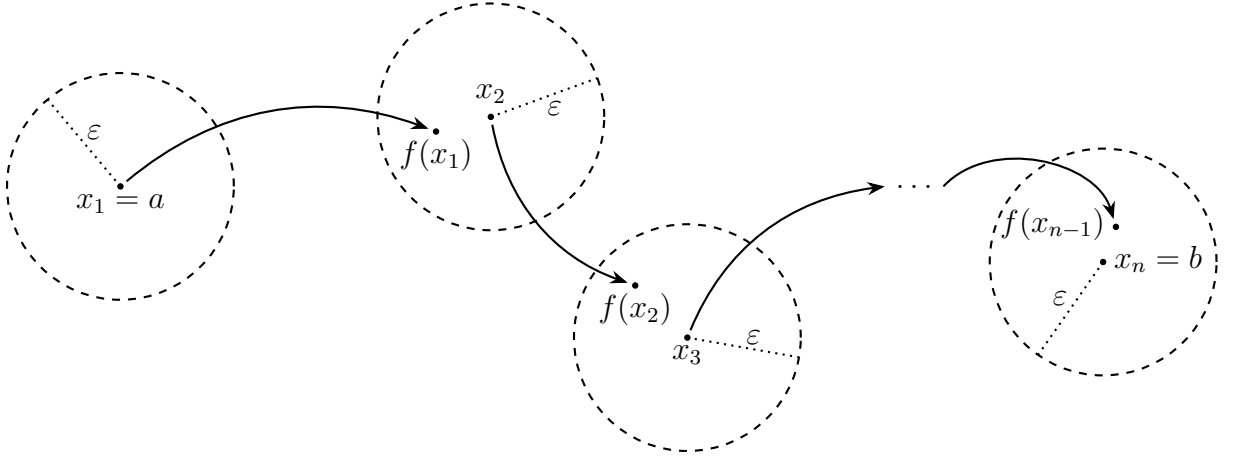

	In \cite{hirsch2001chain}, the authors show that the $\omega$-limit set of a point with a precompact forward orbit is internally chain transitive. Recall that a subset $\mathcal{Y} \subseteq \mathcal{X}$ is \emph{precompact} if $\Cl\mathcal{Y}$ is compact. In particular, a subset of the Euclidean plane $\mathbb{R}^2$ is precompact if and only if it is bounded.
	
	\begin{lem}[{\cite{hirsch2001chain}, Lemma 2.1}] \label{lem:omega_ICT}
	If $x \in \mathcal{X}$ is such that $\mathcal{O}^{+}(x,f)$ is precompact, then $\omega(x,f)$ is internally chain transitive.
	\end{lem}

	\subsection{Lozi map}\label{subsec:lozi_maps}
	
	We consider the orientation-reversing Lozi map
	\begin{equation*}
	L_{a,b} \colon \mathbb{R}^2 \rightarrow \mathbb{R}^2,\quad L_{a,b}(x,y) = (1 + y - a|x|,bx),
	\end{equation*}
for parameter values $(a,b)$ such that $0 < b < 1$ and $a + b > 1$. Here, we will review some notions and concepts already introduced in \cite{kilassa2026accumulation} and \cite{kilassa2026tangential}. For the reader's convenience, Figure \ref{fig:lozi_quadrants_images} summarizes the action of $L_{a,b}$ and the inverse $L_{a,b}^{-1}$ on the four quadrants of the coordinate system. This elementary observation will be used throughout the paper. 
	
	\begin{figure}[!ht]
	\begin{center}
	\begin{subfigure}{.4\textwidth}
	\begin{flushleft}
	\begin{tikzpicture}[auto, scale=.75]
			\tikzstyle{nodec}=[draw,circle,fill=black,minimum size=2pt,
			inner sep=0pt, label distance=2mm]
			\tikzstyle{nodeq}=[draw,circle,fill=black,minimum size=5pt,
			inner sep=0pt, label distance=2mm]
			\tikzstyle{nodeh}=[draw,circle,fill=white,minimum size=4pt,
			inner sep=0pt]
			\tikzstyle{dot}=[circle,draw=none,fill=none,minimum size=0pt,inner sep=2pt, outer sep=-1pt]
			
			\draw[->] (-5.25,0)--(5.25,0) node [below]{$x$};
			\draw[->] (0,-4.25)--(0,4.25) node [left]{$y$};
			\node[label={[xshift=-0.2cm, yshift=-0.7cm]$0$}] at (0,0) {};

			\coordinate (b1) at (5,0);
			\coordinate (b2) at (0,5);
			\coordinate (b3) at (-5,0);
			\coordinate (b4) at (0,-5);	
			
			\coordinate (a1) at (5,4);
			\coordinate (a2) at (5,-4);
			
			\coordinate (xend1) at (-5,4);
			\coordinate (xend2) at (-5,-4);
			\coordinate (LO) at (2.5,0);
			
			\draw [ultra thick, dashed] (xend2)--(LO)--(xend1) node[above right]{$x=1-\frac{a}{b}|y|$};
			
			\draw [ultra thick, dashed] (b3)--(b1);
			
			\node [fill=white] (L1) at (3,2) {$L_{a,b}(\mathcal{Q}_1)$};
			\node [fill=white] (L2) at (3,-2) {$L_{a,b}(\mathcal{Q}_2)$};
			\node [fill=white] (L3) at (-3,-1.5) {$L_{a,b}(\mathcal{Q}_3)$};
			\node [fill=white] (L4) at (-3,1.5) {$L_{a,b}(\mathcal{Q}_4)$};
			
			\begin{pgfonlayer}{bg}
			\fill [pattern=dots] (b1.center)--(LO.center)--(xend1.center)--(a1.center)--cycle;
			\fill [pattern=vertical lines] (b1.center)--(LO.center)--(xend2.center)--(a2.center)--cycle;
			\fill [pattern=horizontal lines] (xend1.center)--(LO.center)--(b3.center)--cycle;
			\fill [pattern=crosshatch] (b3.center)--(xend2.center)--(LO.center)--cycle;
			\end{pgfonlayer}
	\end{tikzpicture}
	\end{flushleft}
	\end{subfigure}\hspace{2cm}%
	\begin{subfigure}{.4\textwidth}
	\begin{tikzpicture}[auto, scale=.75]
			\tikzstyle{nodec}=[draw,circle,fill=black,minimum size=2pt,
			inner sep=0pt, label distance=2mm]
			\tikzstyle{nodeq}=[draw,circle,fill=black,minimum size=5pt,
			inner sep=0pt, label distance=2mm]
			\tikzstyle{nodeh}=[draw,circle,fill=white,minimum size=4pt,
			inner sep=0pt]
			\tikzstyle{dot}=[circle,draw=none,fill=none,minimum size=0pt,inner sep=2pt, outer sep=-1pt]
			
			\draw[->] (-5.25,0)--(5.25,0) node [below]{$x$};
			\draw[->] (0,-4.25)--(0,4.25) node [left]{$y$};
			\node[label={[xshift=-0.2cm, yshift=-0.7cm]$0$}] at (0,0) {};

			\coordinate (b1) at (5,0);
			\coordinate (b2) at (0,4);
			\coordinate (b3) at (-5,0);
			\coordinate (b4) at (0,-4);
			
			\coordinate (c1) at (5,-4);
			\coordinate (c2) at (-5,-4);
			
			\coordinate (yend1) at (5,4);
			\coordinate (yend2) at (-5,4);
			\coordinate (lm1O) at (0,-3);
			
			\draw [ultra thick, dashed] (yend2)--(lm1O)--(yend1) node[above left]{$y=a|x|-1$};
			
			\draw [ultra thick, dashed] (b4)--(b2);
			
			\node [fill=white] (Li1) at (2,3) {$L_{a,b}^{-1}(\mathcal{Q}_1)$};
			\node [fill=white] (Li2) at (3,-2) {$L_{a,b}^{-1}(\mathcal{Q}_2)$};
			\node [fill=white] (Li3) at (-3,-2) {$L_{a,b}^{-1}(\mathcal{Q}_3)$};
			\node [fill=white] (Li4) at (-2,3) {$L_{a,b}^{-1}(\mathcal{Q}_4)$};
			
			\begin{pgfonlayer}{bg}
			\fill [pattern=dots] (b2.center)--(yend1.center)--(lm1O.center)--cycle;
			\fill [pattern=vertical lines] (b4.center)--(lm1O.center)--(yend1.center)--(c1.center)--cycle;
			\fill [pattern=horizontal lines] (yend2.center)--(lm1O.center)--(b2.center)--cycle;
			\fill [pattern=crosshatch] (yend2.center)--(lm1O.center)--(b4.center)--(c2.center)--cycle;
			\end{pgfonlayer}
			\end{tikzpicture}
	\end{subfigure}
	\end{center}
	\caption{Images of the quadrants ($\mathcal{Q}_i$ is the $i$-th quadrant, $i=1,2,3,4$) under the Lozi map $L_{a,b}$ (left) and its inverse $L_{a,b}^{-1}$ (right), for positive values of parameters $a$ and $b$.}
	\label{fig:lozi_quadrants_images}
	\end{figure}
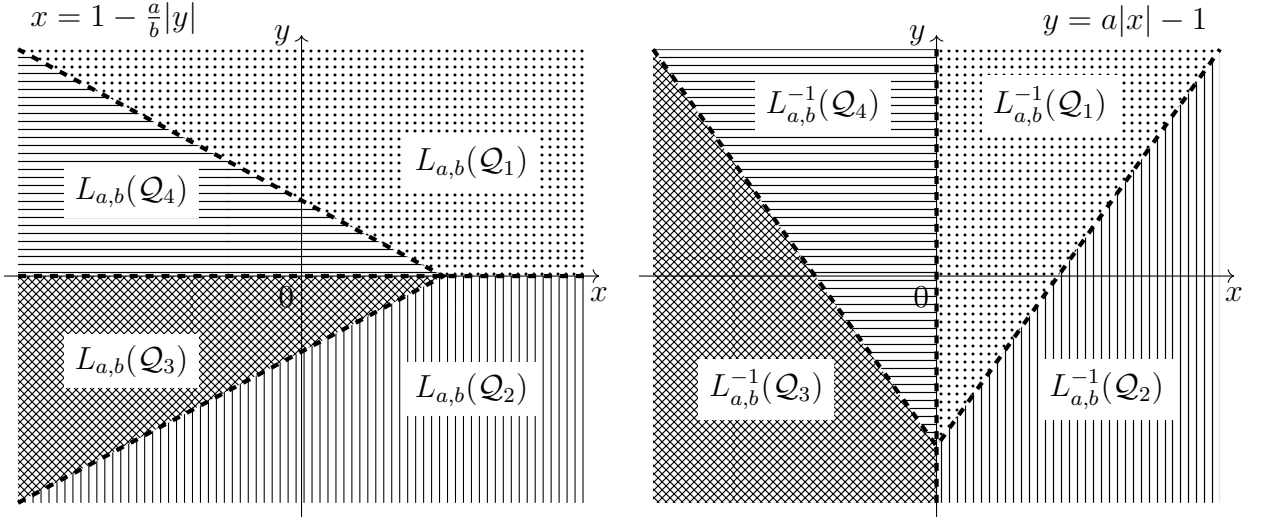	 

	For parameter pairs $(a,b)$ such that $0<b<1$ and $a+b>1$, the Lozi map $L_{a,b}$ has two hyperbolic saddle fixed points, $X=\bigl(\frac{1}{1+a-b},\:\frac{b}{1+a-b}\bigr)$ in the first and $Y=\bigl(\frac{1}{1-a-b},\:\frac{b}{1-a-b}\bigr)$ in the third quadrant. The eigenvalues of the differential of $L_{a,b}$ at $X$ are
	\begin{equation*}
	\lambda_X^u = \tfrac{1}{2}\left(-a-\sqrt{a^2+4b}\right),\quad \lambda_X^s = \tfrac{1}{2}\left(-a+\sqrt{a^2+4b}\right),
	\end{equation*}
while those at $Y$ are given by
	\begin{equation*}
	\lambda_Y^u = \tfrac{1}{2}\left(a+\sqrt{a^2+4b}\right),\quad \lambda_Y^s = \tfrac{1}{2}\left(a-\sqrt{a^2+4b}\right).
	\end{equation*}
Observe that $\lambda_X^u<-1$, $0<\lambda_X^s<1$ and $\lambda_Y^u>1$, $-1<\lambda_Y^s<0$. Moreover, for every eigenvalue $\lambda$, the corresponding eigenvector is given by $\binom{\lambda}{b}$.

	Furthermore, when $0<b<1$ and $1-b<a<1+b$, there are also two period-two points,
	\begin{equation*}
	P=\left(\frac{1+a-b}{a^2+(1-b)^2},\,\frac{b(1-a-b)}{a^2+(1-b)^2}\right),\ P'=\left(\frac{1-a-b}{a^2+(1-b)^2},\,\frac{b(1+a-b)}{a^2+(1-b)^2}\right).
	\end{equation*}
Points $P$ and $P'$ lie in the fourth and second quadrant, respectively, and they are attracting. 

	For every point $A \in \mathbb{R}^2$ and every $k \in \mathbb{Z} \setminus \{0\}$, we put $A^k=L_{a,b}^k(A)$; in particular, $A^0=A$.

	Recall that the \emph{unstable} and \emph{stable manifold} of $X$ are respectively
	\begin{equation*}
	W_X^u = \{A \in \mathbb{R}^2 \colon A^{-n} \overset{n \rightarrow \infty}{\longrightarrow} X\}, \quad W_X^s = \{A \in \mathbb{R}^2 \colon A^n \overset{n \rightarrow \infty}{\longrightarrow} X\}. 
	\end{equation*}
We know that both $W_X^u$ and $W_X^s$ are invariant under $L_{a,b}$ and $L_{a,b}^{-1}$, and contain $X$. The \emph{unstable} and \emph{stable manifold} of $Y$ are defined analogously: 
	\begin{equation*}
	W_Y^u = \{A \in \mathbb{R}^2 \colon A^{-n} \overset{n \rightarrow \infty}{\longrightarrow} Y\}, \quad W_Y^s = \{A \in \mathbb{R}^2 \colon A^n \overset{n \rightarrow \infty}{\longrightarrow} Y\}. 
	\end{equation*}
Since the Lozi map is piecewise affine rather than smooth, $W_X^u$, $W_X^s$, $W_Y^u$ and $W_Y^s$ are polygonal lines rather than smooth curves; see Figures \ref{fig:Lozi_stable_unstable_Z_V} and \ref{figure:stable_Y}. We nevertheless refer to them as stable and unstable manifolds, following \cite{misiurewicz1980strange}.
	
	Moreover, consider the points at which the stable manifolds $W_X^s$ and $W_Y^s$ break: these points are endpoints of the maximal line segments contained in these polygonal lines. We follow the terminology from \cite{boronski2023densely} and call them \emph{V-points}.

	\begin{figure}
	\begin{center}
		\includegraphics[width=\linewidth]{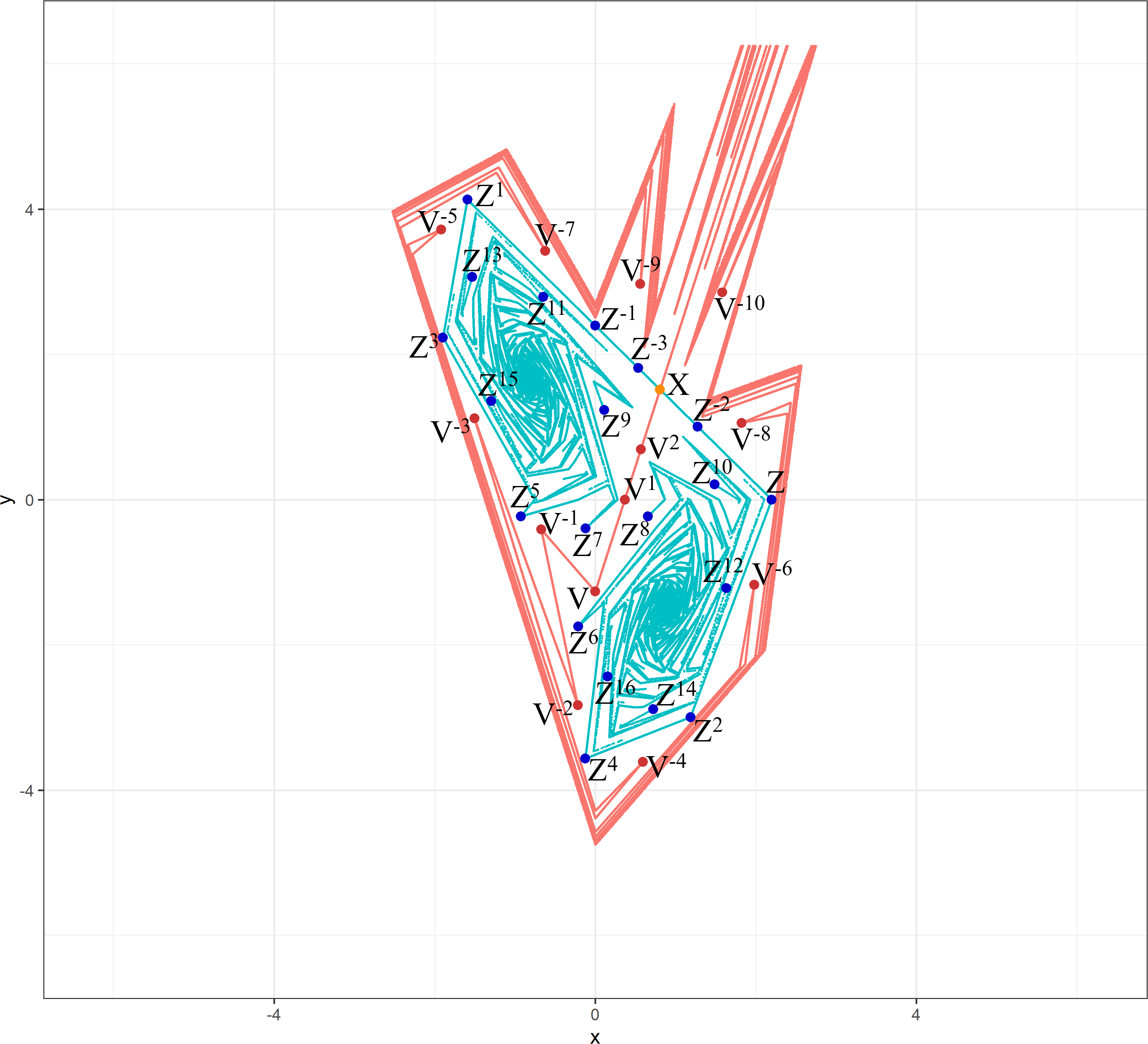}
	\end{center}
	\caption{The stable (red) and unstable (blue) manifold of $X$ for parameter values $a=1.18$, $b=0.94$, together with some iterates of $Z$ and $V$.}
	\label{fig:Lozi_stable_unstable_Z_V}
	\end{figure}
	
	Consider the unstable manifold $W_X^u$. We denote the half of that manifold that starts at $X$ and goes to the right by $W_X^{u+}$. That half intersects the positive $x$-axis for the first time at the point
	\begin{equation}
	Z=\left(\frac{2+a+\sqrt{a^2+4b}}{2(1+a-b)},\: 0\right)=\left(\frac{2}{2+a-\sqrt{a^2+4b}},\: 0\right).
	\label{eq:Z}
	\end{equation}
We denote by $W_X^{u-}$ the other half, starting at $X$ and going to the left. Note that
	\begin{equation*}
	W_X^{u+} = \{X\}\cup\bigcup_{n=-\infty}^{\infty}L_{a,b}^{2n}\bigl(\overline{ZZ^2}\bigr) ,\quad W_X^{u-} = \{X\}\cup\bigcup_{n=-\infty}^{\infty}L_{a,b}^{2n}\bigl(\overline{Z^{-1}Z^1}\bigr).
	\end{equation*}
Similarly, we denote by $W_X^{s-}$ the half of the stable manifold $W_X^s$ that starts at $X$ and extends downward. That branch intersects the negative $y$-axis for the first time at the point
	\begin{equation}
	V=\left(0,\: \frac{2b - a - \sqrt{a^2 + 4b}}{2 ( 1 + a - b)}\right)=\left(0,\: -\frac{2b}{-a+2b+\sqrt{a^2+4b}}\right).
	\label{eq:V}
	\end{equation}
We have
	\begin{equation*}
	W_X^{s-} = \{X\} \cup \bigcup_{n=-\infty}^{\infty}L_{a,b}^{-n}\bigl(\overline{VV^1}\bigr).
	\end{equation*}
The other half of $W_X^s$ is a half-line emanating from $X$ and going up in the first quadrant. We denote that branch by $W_X^{s+}$.

	Let us observe more closely the backward iterates of $V$. Since $V$ lies on the negative $y$-axis, its preimage $V^{-1}$ lies on the line $y=-ax-1$, in the second or third quadrant. More generally, the preimage of a point in the third quadrant lies either in the second or third quadrant. Thus, at first sight, it might seem possible that all backward iterates of $V$ remain in the third quadrant. However, it is shown in \cite{kilassa2026tangential} that this cannot occur.
	
	Let $s_0$ be the slope of the straight line segment $\overline{VV^1}$, $s_0=\frac{b}{\lambda^s_X}=\frac{1}{2}\left(a+\sqrt{a^2+4b}\right)$. For every $n \in \mathbb{N}$, let $s_n$ be the slope of the arc component of $L_{a,b}^{-n}\bigl(\overline{VV^1}\bigr)$ in the third quadrant (note that if $V^{-1}$, $V^{-2}$, \ldots, $V^{-n+1}$ all lie in the third quadrant, then $L_{a,b}^{-n}\bigl(\overline{VV^1}\bigr)$ is simply the line segment $\overline{V^{-n}V^{-n+1}}$). Since $L_{a,b}^{-1}$ acts as an affine map on points in the lower half-plane, one can obtain a recurrence for the sequence of slopes $(s_n)_{n \in \mathbb{N}_0}$. We will use the following result, which summarizes the behavior of these slopes.
	
	\begin{lem}[{\cite{kilassa2026tangential}, Lemma 3.2}]\label{lem:lozi_recurr_coef}
	The sequence $(s_n)_{n\in\mathbb{N}_0}$ given by the recurrence
	\begin{equation}\label{eg:recurr}
	s_{n+1}=b\cdot\frac{1}{s_n}-a,\quad n\geqslant0;\quad s_0=\tfrac{1}{2}\left(a+\sqrt{a^2+4b}\right),
	\end{equation} 
	has the following properties:
	\begin{enumerate}
	\item if $s_{n_1}>0$ for some $n_1\in\mathbb{N}_0$, then $s_n>0$ for all $n\leqslant n_1$,
	\item if $s_{n_2}<0$ for some $n_2\in\mathbb{N}_0$, then $s_n<0$ for all $n\geqslant n_2$,
	\item for all $n\in\mathbb{N}_0$, if $s_{2n+1}>0$, then $s_{2n+2}>0$,
	\item $(s_n)$ converges and $\displaystyle\lim_{n\rightarrow\infty}s_n=\tfrac{1}{2}\left(-a-\sqrt{a^2+4b}\right)$.  
	\end{enumerate}
	\end{lem}
	
	As it turns out, the limit of $(s_n)$ is the slope of the arc component of $W_Y^s$ in the third quadrant that passes through the fixed point $Y$. This result will be revisited in Section \ref{sec:basin_attraction}. In addition, Lemma \ref{lem:lozi_recurr_coef} is the key ingredient for the description of the shape formed by the stable manifold $W_X^s$ in the third quadrant.
	
	For points $A_1,A_2 \in W_X^s$, let $[A_1,A_2]^s\subset W_X^s$ denote the polygonal line lying on $W_X^s$ with $A_1$ and $A_2$ as endpoints. 
	
	\begin{lem}[Zigzag structure of $W_X^{s}$ in the third quadrant; {\cite{kilassa2026tangential}, Lemma 3.3}]\label{lem:zigzag}
	There exists a positive integer $n$ such that $V^{-n}$ lies in the second quadrant. The smallest such positive integer $n_0$ is odd, $[V,V^{-n_0+1}]^{s}$ is contained in the third quadrant and all V-points of $W_X^{s}$ on it are negative iterates of $V$.
	\end{lem}
	
	In other words, only finitely many consecutive backward iterates of $V$ lie in the third quadrant, and $W_X^s$ forms a zigzag polygonal line breaking at those iterates. One iterate $V^{-n_0}$ eventually lands in the second quadrant and all further backward iterates of the straight line segment $\overline{V^{-n_0}V^{-n_0+1}}$ will intersect both coordinate axes in the third quadrant.
	
	The polygonal line $[V,V^{-n_0}]^s$ is called the \emph{zigzag part} of $W_X^s$. For example, in Figure \ref{fig:Lozi_stable_unstable_Z_V}, one can see that $n_0=3$, that is, the zigzag part of $W_X^s$ is $[V,V^{-3}]^s$.

	\subsection{Notation}\label{subsec:notation}
	
	Since we are working in the Euclidean plane, we will use the following notation to denote geometrical objects and their topological characteristics.
	\begin{itemize}
	\item The first, second, third and fourth quadrant of the Cartesian coordinate system are denoted by $\mathcal{Q}_1$, $\mathcal{Q}_2$, $\mathcal{Q}_3$ and $\mathcal{Q}_4$, respectively. The origin is denoted by $O$.
	\item The positive and negative $x$-axis ($y$-axis) are denoted by $Ox^{+}$ and $Ox^{-}$ ($Oy^{+}$ and $Oy^{-}$), respectively. The $x$-axis and $y$-axis are denoted by $Ox$ and $Oy$, respectively.
	\item We use capital Latin letters such as $A,B,C,\ldots$ to mark points in the plane whenever no other notation is specified. The notation $L_{a,b}$ is reserved for the Lozi map. 
	\item In particular, $X$ and $Y$ denote the fixed points of $L_{a,b}$. Moreover, $Z$ and $V$ denote points on $W_X^u$ and $W_X^s$, respectively, as defined in Subsection \ref{subsec:lozi_maps}; see Figure \ref{fig:Lozi_stable_unstable_Z_V}.
	\item For a point $A\in\mathbb{R}^2$, $A_x$ and $A_y$ are the $x$-coordinate and $y$-coordinate of $A$, respectively.
	\item For $A,B\in\mathbb{R}^2$, the straight line segment with endpoints $A$ and $B$ is denoted by $\overline{AB}$. The straight line through $A$ and $B$ is denoted by $AB$.
	\item For $A\in\mathbb{R}^2$ and $\varepsilon>0$, $B_{\varepsilon}(A)$ represents the open ball in the plane centered at $A$ of radius $\varepsilon$.
	\item Lowercase Greek letters $\alpha,\beta,\gamma,$ etc.\ will stand for straight line segments or polygonal lines.
	\item Calligraphic letters $\mathcal{A},\mathcal{B},$ etc.\ will denote two-dimensional subsets of the plane, typically polygons.
	\item Let $\dist$ denote the Euclidean metric on $\mathbb{R}^2$.
	\item For $\mathcal{A}\subset\mathbb{R}^2$, we will use the following notation:
		\begin{itemize}
			\item $\Int\mathcal{A}$ is the interior of $\mathcal{A}$,
			\item $\Cl\mathcal{A}$ is the closure of $\mathcal{A}$,
			\item $\partial\mathcal{A}$ is the boundary of $\mathcal{A}$,
			\item $\Conv\mathcal{A}$ is the convex hull of $\mathcal{A}.$
		\end{itemize}
	\end{itemize}			
	
	Finally, we will use specific notation concerning the Lozi map and the stable and unstable manifolds $W_X^s$ and $W_X^u$.
	\begin{itemize}
		\item For every $k\in\mathbb{Z}$ and every point $A \in \mathbb{R}^2$, we put $A^k=L_{a,b}^{k}(A)$. In particular, $A^{0} = A$.
		\item For points $A,B\in W_X^u$, we put:
			\begin{itemize}
				\item $[A,B]^{u}\subset W_X^u$ is the polygonal line lying on $W_X^{u}$ with $A$ and $B$ as endpoints,
				\item in particular, if $[A,B]^{u}$ is a straight line segment, we denote it by $\overline{AB}^{u}$,
				\item $[A,B)^{u}:=[A,B]^{u}\setminus\{B\}$,
				\item $(A,B]^{u}:=[A,B]^{u}\setminus\{A\}$,
				\item $(A,B)^{u}:=[A,B]^{u}\setminus\{A,B\}$.
			\end{itemize}
		\item For $A,B\in W_X^s$, we define the sets $[A,B]^{s}$, $\overline{AB}^{s}$, $[A,B)^{s}$, $(A,B]^{s}$ and $(A,B)^{s}$ analogously.
		\item Additionally, $\ell$ will represent the set of accumulation points of $W_X^u$ that do not lie on $W_X^u$, that is, $\ell = \Cl W_X^u \setminus W_X^u$.			 
	\end{itemize}

	\section{Basin of attraction} \label{sec:basin_attraction}
	
	Let $\mathfrak{R}$ denote the set of parameter pairs $(a,b)$ such that $0 < b < 1$, $a + b > 1$, $W_X^u \cap W_X^s = \{X\}$ (there are no homoclinic points for the fixed point $X$), and the period-two orbit $\{P,P'\}$ is attracting. Within this parameter set, we will completely describe the basin of attraction of the accumulation set $\ell$ of the unstable manifold of $X$.
	
	In \cite{misiurewicz2024zero}, the authors studied the parameters in a subregion $\mathcal{R}$ satisfying the additional condition that $W_X^u$ intersects the coordinate axes only at $Z$ and $Z^{-1}$. They showed that $\ell = \{P,P'\}$ whenever $(a,b) \in \mathcal{R}$. The same conclusion holds when $(a,b) \in \mathfrak{R} \setminus \mathcal{R}$, and $W_X^u$ intersects the coordinate axes at finitely many points; see Theorem 3.1 in \cite{kilassa2026accumulation}.
	
	The case in which $W_X^u$ has infinitely many intersections with $Ox$ and $Oy$ is the most intricate, as it gives rise to a more complex accumulation set $\ell$. It is shown in \cite{kilassa2026accumulation} that in this case, the connected components of $\ell$ can be obtained geometrically as intersections of forward images of a specially constructed polygon $\mathcal{D}$ under $L_{a,b}^2$. We briefly review this construction in the following subsection. We then analyze the stable manifold of the fixed point $Y$, which plays a central role in determining the geometry of the region containing the basin of attraction of $\ell$.
	
	\subsection{Accumulation set $\ell$ and trapping region $\mathcal{D}$} \label{subsec:accum_set_ell}    
	
	We now consider parameter pairs $(a,b) \in \mathfrak{R}$ for which $W_X^u$ intersects $Ox$ and $Oy$ at infinitely many points. In \cite{kilassa2026accumulation}, the following lemma was proved.
	
	\begin{lem}[{\cite{kilassa2026accumulation}, Lemma 3.2}] \label{lem:point_S}
	Let $(a,b) \in \mathfrak{R}$ and assume that $W_X^u$ has infinitely many intersections with $Ox$ and $Oy$. Then there exists a point $S \in W_X^u \cap Ox$, $S \neq Z$, such that $\overline{SZ} \cap W_X^u = \{S,Z\}$.
	\end{lem}
	
	In other words, the point $S$ is such that the segment $\overline{SZ}$ does not contain any further intersections of $W_X^u$ with $Ox$. This result serves as the basis for constructing a trapping polygon $\mathcal{D}$ using a distinguished point $E$. We define this point according to the following two cases.
	
	\fbox{Case 1:}  $W_X^{u+}$ intersects $Ox$ at points other than $Z$. Then $S \in W_X^{u+}$, and we define $E:=S$; see Figure \ref{fig:polygon_D_case1}.
	
	\fbox{Case 2:} $W_X^{u+}$ does not intersect $Ox$ at points other than $Z$. We then have $S \in W_X^{u-}$. In this case, let $B_1,B_2,B_3$ be the first three intersections of $W_X^{u-}$ with $Ox$, in the order along $W_X^{u-}$ starting from $X$; see Figure \ref{fig:polygon_D_case2}. In the proof of Lemma \ref{lem:point_S}, it has been shown that $S = B_2$. In addition, $W_X^{u-} \setminus [X, B_3]^u$ lies in the region bounded by $[B_1,B_3]^u \cup \overline{B_1B_3}$.
	
	Let $E^{-1}$ be the point on $[S,B_3]^u$ that is closest to $Z^{-1}$ in the Euclidean metric. If there are several such points, we choose $E^{-1}$ such that $[S,E^{-1}]^u$ contains all the others. We now define $E$ as $E := L_{a,b}(E^{-1})$.
	
	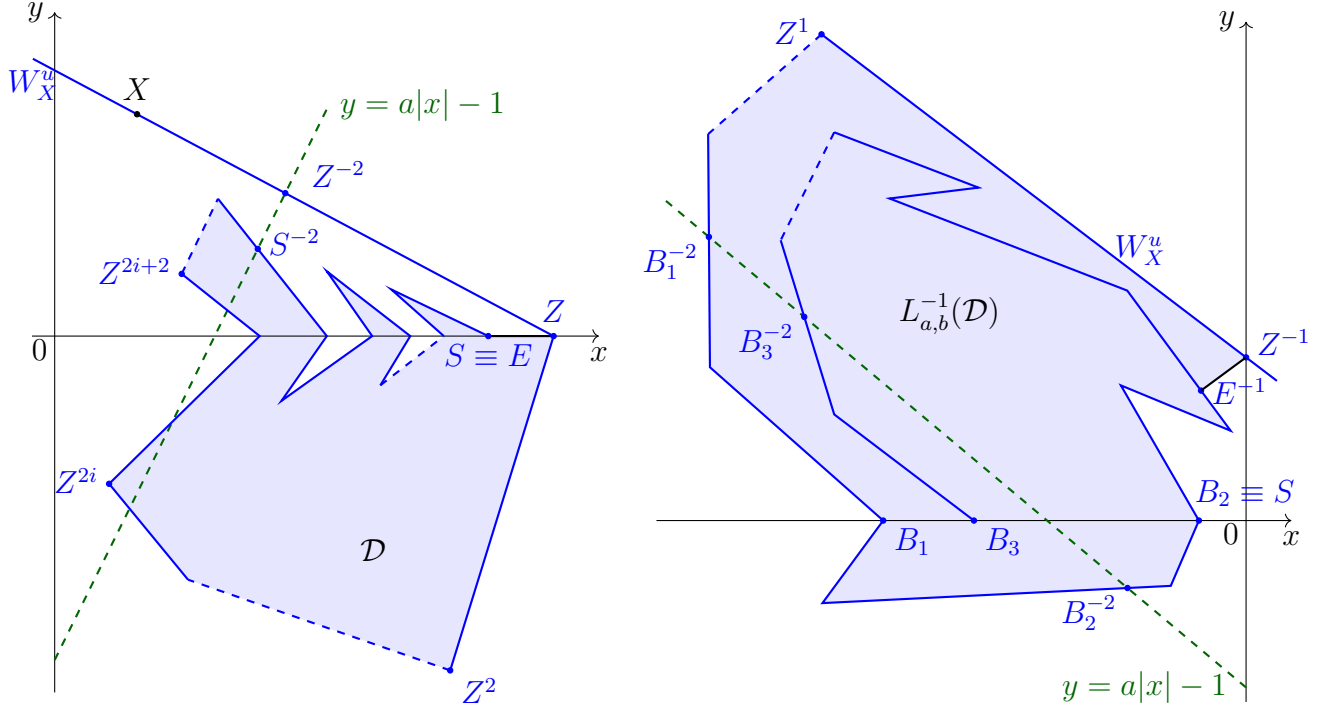
\begin{figure}[!ht]
	\begin{center}
	\begin{subfigure}[t]{0.475\textwidth}
	\begin{tikzpicture}[auto, scale=0.6, yscale=1.4285]
			\tikzstyle{nodec}=[draw,circle,fill=black,minimum size=2pt,
			inner sep=0pt, label distance=2mm]
			\tikzstyle{nodeh}=[draw,circle,fill=white,minimum size=4pt,
			inner sep=0pt]
			\tikzstyle{dot}=[circle,draw=none,fill=none,minimum size=0pt,inner sep=2pt, outer sep=-1pt]

			\draw[->] (-0.5,0)--(12,0) node [below]{$x$};
			\draw[->] (0,-5.5)--(0,5) node [left]{$y$};
			\node[label={[xshift=-0.2cm, yshift=-0.6cm]$0$}] at (0,0) {};
			
			\coordinate (e0) at (0,0);
			\coordinate (ex) at (1,0);
			\coordinate (ey) at (0,1);

			\coordinate (pr01) at (0,-5);
			\coordinate (pr02) at (6.04,3.56);
			
			\draw[green!40!black, thick, dashed] (pr01)--(pr02) node[right]{$y=a|x|-1$};
			
			\coordinate (a01) at (-0.48,4.28);
			\coordinate (a02) at (11,0);
			\coordinate (a03) at (8.72,-5.16);
			\coordinate (a04) at (2.94,-3.76);
			\coordinate (a05) at (1.2,-2.28);
			\coordinate (a06) at (4.52,0);
			\coordinate (a07) at (2.8,0.96);
			\coordinate (a08) at (3.6,2.12);
			\coordinate (a09) at (6,0);
			\coordinate (a10) at (5,-1);
			\coordinate (a11) at (7,0);
			\coordinate (a12) at (6,1);
			\coordinate (a13) at (7.84,0);
			\coordinate (a14) at (7.18,-0.76);
			\coordinate (a15) at (8.58,0);
			\coordinate (a16) at (7.42,0.72);
			\coordinate (a17) at (9.56,0);
			\coordinate (a18) at (8.18,-1.01);
			
			\draw[blue, thick] (a01) node[below]{$W_X^u$}--(a02)--(a03);
			\draw[blue, thick, dashed] (a03)--(a04);
			\draw[blue, thick] (a04)--(a05)--(a06)--(a07);
			\draw[blue, thick, dashed] (a07)--(a08);
			\draw[blue, thick] (a08)--(a09)--(a10)--(a11)--(a12)--(a13)--(a14);
			\draw[blue, thick, dashed] (a14)--(a15);
			\draw[blue, thick] (a15)--(a16)--(a17);
			
			\draw[thick] (a17)--(a02);
			
			\coordinate (x) at ($(a01)!0.2!(a02)$);
			
			\node[nodec, label={[above]$X$}] at (x) {};
			
			\node[nodec, blue, label={[above right, blue, xshift=2mm, yshift=-1mm]$Z^{-2}$}] at (intersection of a01--a02 and pr01--pr02) {};
			\node[nodec, blue, label={[above, blue]$Z$}] at (a02) {};
			\node[nodec, blue, label={[below right, blue]$Z^2$}] at (a03) {};
			\node[nodec, blue, label={[left, blue]$Z^{2i}$}] at (a05) {};
			\node[nodec, blue, label={[left, blue]$Z^{2i+2}$}] at (a07) {};
			\node[nodec, blue, label={[below, blue]$S\equiv E$}] at (a17) {};
			\node[nodec, blue, label={[right, blue, yshift=0.5mm]$S^{-2}$}] at (intersection of a08--a09 and pr01--pr02) {};
			
			\node[fill=blue!10!white] (dd) at (7.02,-3.3) {$\mathcal{D}$};
			
			\begin{pgfonlayer}{bg}
				\fill[blue!10!white] (a02.center)--(a03.center)--(a04.center)--(a05.center)--(a06.center)--(a07.center)--(a08.center)--(a09.center)--(a10.center)--(a11.center)--(a12.center)--(a13.center)--(a14.center)--(a15.center)--(a16.center)--(a17.center)--cycle;
			\end{pgfonlayer}
			\end{tikzpicture}
		\caption{$W_X^{u+}$ intersects the $x$-axis at points other than $Z$. \label{fig:polygon_D_case1}}
	\end{subfigure}
	\hfill
	\begin{subfigure}[t]{0.475\textwidth}
	\begin{tikzpicture}[auto, scale=0.6]
			\tikzstyle{nodec}=[draw,circle,fill=black,minimum size=2pt,
			inner sep=0pt, label distance=2mm]
			\tikzstyle{nodeh}=[draw,circle,fill=white,minimum size=4pt,
			inner sep=0pt]
			\tikzstyle{dot}=[circle,draw=none,fill=none,minimum size=0pt,inner sep=2pt, outer sep=-1pt]

			\draw[->] (-13,0)--(1,0) node [below]{$x$};
			\draw[->] (0,-4)--(0,11) node [left]{$y$};
			\node[label={[xshift=-0.2cm, yshift=-0.6cm]$0$}] at (0,0) {};
			
			\coordinate (e0) at (0,0);
			\coordinate (ex) at (1,0);
			\coordinate (ey) at (0,1);

			\coordinate (pr01) at (0,-3.68);
			\coordinate (pr02) at (-12.9,7.14);
			
			\draw[green!40!black, thick, dashed] (pr01)node[left, xshift=-1mm]{$y=a|x|-1$}--(pr02);
			
			\coordinate (a01) at (0.68,3.08);
			\coordinate (a02) at (-9.36,10.72);
			\coordinate (a03) at (-11.86,8.52);
			\coordinate (a04) at (-11.82,3.38);
			\coordinate (a05) at (-8,0);
			\coordinate (a06) at (-9.34,-1.82);
			\coordinate (a07) at (-1.66,-1.44);
			\coordinate (a08) at (-1.04,0);
			\coordinate (a09) at (-2.76,2.98);
			\coordinate (a10) at (-0.34,1.98);
			\coordinate (a11) at (-2.62,5.07);
			\coordinate (a12) at (-7.86,7.1);
			\coordinate (a13) at (-5.9,7.34);
			\coordinate (a14) at (-9.08,8.56);
			\coordinate (a15) at (-10.26,6.18);
			\coordinate (a16) at (-9.08,2.34);
			\coordinate (a17) at (-6,0);
			
			\draw[blue, thick] (a01)--(a02) node[pos=0.3, above]{$W_X^u$};
			\draw[blue, thick, dashed] (a02)--(a03);
			\draw[blue, thick] (a03)--(a04)--(a05)--(a06)--(a07)--(a08)--(a09)--(a10)--(a11)--(a12)--(a13)--(a14);
			\draw[blue, thick, dashed] (a14)--(a15);
			\draw[blue, thick] (a15)--(a16)--(a17);

			\coordinate (zm1) at (intersection of a01--a02 and e0--ey);
			
			\node[nodec, blue, label={[above right, blue, yshift=-1.5mm]$Z^{-1}$}] at (zm1) {};
			\node[nodec, blue, label={[left, blue]$Z^1$}] at (a02) {};
			\node[nodec, blue, label={[below right, blue]$B_1$}] at (a05) {};
			
			\node[nodec, blue, label={[above right, blue, xshift=-2mm]$B_2\equiv S$}] at (a08) {};
			
			\node[nodec, blue, label={[below right, blue]$B_3$}] at (a17) {};
			
			\node[nodec, blue, label={[below left, blue]$B_1^{-2}$}] at (intersection of a03--a04 and pr01--pr02) {};
			\node[nodec, blue, label={[below left, blue]$B_2^{-2}$}] at (intersection of a06--a07 and pr01--pr02) {};
			\node[nodec, blue, label={[below left, blue]$B_3^{-2}$}] at (intersection of a15--a16 and pr01--pr02) {};
			
			\coordinate (t) at ($(a10)!(zm1)!(a11)$);
			
			\node[nodec, blue, label={[right, blue, yshift=-0.5mm]$E^{-1}$}] at (t) {};
			
			\draw[thick] (zm1)--(t);
			
			\node[fill=blue!10!white] (ddm1) at (-6.54,4.56) {$L_{a,b}^{-1}(\mathcal{D})$};
			
			\begin{pgfonlayer}{bg}
				\fill[blue!10!white] (zm1.center)--(a02.center)--(a03.center)--(a04.center)--(a05.center)--(a06.center)--(a07.center)--(a08.center)--(a09.center)--(a10.center)--(t.center)--cycle;
			\end{pgfonlayer}
			\end{tikzpicture}
		\caption{$W_X^{u+}$ intersects the $x$-axis only at $Z$. \label{fig:polygon_D_case2}}
	\end{subfigure}
	\end{center}
	\caption{Construction of the point $E$ and the polygon $\mathcal{D}$. Reproduced from K.\ Kilassa Kvaternik, Accumulation sets and zero entropy dynamics in the Lozi map, \emph{Chaos} (2026), \url{https://doi.org/10.1063/5.0344038}, with the permission of AIP Publishing. \label{fig:polygon_D_construction}}
	\end{figure} 	
	
	In both cases, we define $\mathcal{D}$ to be the polygon enclosed by the boundary
	\begin{equation}
	\label{eq:polygon_D_boundary} 
	\partial\mathcal{D} = [Z,E]^u \cup \overline{ZE}.
	\end{equation}	  
	
	Note that $\Int\mathcal{D}$ contains the branch $W_X^{u+} \setminus [X,E]^{u}$. Applying $L_{a,b}$ and using the fact that $L_{a,b}(W_X^{u+}) = W_X^{u-}$, we obtain that $L_{a,b}(\mathcal{D})$ contains $W_X^{u-} \setminus [X,E^1]^{u}$. Therefore, the union $\mathcal{D} \cup L_{a,b}(\mathcal{D})$ contains all of $W_X^u$ except for $(Z,Z^1)^u$, which has an open neighborhood disjoint from the rest of $W_X^u$ (the reader can find the justification of the last claim in \cite{kilassa2026accumulation}, Corollary 3.6).
	 
	 It follows that the accumulation set $\ell$ of $W_X^u$ is contained in $\mathcal{D} \cup L_{a,b}(\mathcal{D})$. Therefore, $\ell$ can be represented as
	\begin{equation*}
	\ell = \ell_L \cup \ell_R,
	\end{equation*}
where $\ell_R = \ell \cap \mathcal{D}$ and $\ell_L = \ell \cap L_{a,b}(\mathcal{D})$ (which we refer to as the right and left part of the accumulation set).  
	
	In \cite{kilassa2026accumulation}, it was proved that $\mathcal{D}$ is $L_{a,b}^2$-invariant. Moreover, $\mathcal{D}$ is an eventually trapping region for $L_{a,b}^2$, that is, there exists $n \in \mathbb{N}$ such that $L_{a,b}^{2n}(\mathcal{D}) \subset \Int\mathcal{D}$. Furthermore, $\ell_R$ and $\ell_L$ can be represented as the nested intersections
	\begin{equation}
	\label{eq:ell_intersect_D}
	\ell_R = \bigcap_{k=0}^{\infty} L_{a,b}^{2k}(\mathcal{D}), \quad \ell_L = \bigcap_{k=0}^{\infty} L_{a,b}^{2k+1}(\mathcal{D}).
	\end{equation}
	
	Since $\mathcal{D}$ is compact, and $(L_{a,b}^{2k}(\mathcal{D}))_{k \in \mathbb{N}_0}$ is a nested sequence of compact sets with intersection $\ell_R$, we have that $\sup_{J \in L_{a,b}^{2k}(\mathcal{D})}\dist(J,\ell_R) \rightarrow 0$ as $k \rightarrow \infty$. Similarly, for $\ell_L$, we have $\sup_{J \in L_{a,b}^{2k+1}(\mathcal{D})}\dist(J,\ell_L) \rightarrow 0$ as $k \rightarrow \infty$. 
	
	Hence, for $A \in \mathcal{D}$, we have $A^{2k} \in L_{a,b}^{2k}(\mathcal{D})$ and $A^{2k+1} \in L_{a,b}^{2k+1}(\mathcal{D})$ for every $k \in \mathbb{N}_0$, so the sequence $(A^{2k})_{k \in \mathbb{N}_0}$ is attracted to $\ell_R$, and $(A^{2k+1})_{k \in \mathbb{N}_0}$ is attracted to $\ell_L$. Since both subsequences are attracted to $\ell$, we have that every point in $\mathcal{D}$ is attracted to $\ell$.
	
	\begin{cor} \label{cor:D_attracted_to_ell}
	For every point $A \in \mathcal{D}$, we have that $\dist(A^n,\ell) \rightarrow 0$ as $n \rightarrow \infty$.
	\end{cor}
	
	The preceding construction shows that $\mathcal{D}$ is contained in the basin of attraction of $\ell$. We next show that, conversely, every point whose $\omega$-limit set meets $\ell$ belongs to the basin of $\ell$. We will then use the separation of $X,Y$ from $\ell$ and the structure of the non-wandering set to identify which points lie outside the basin of $\ell$.
	
	\begin{lem} \label{lem:approaching_ell}
	Let $(a,b) \in \mathfrak{R}$. If $A \in \mathbb{R}^2$ is a point such that $\omega(A,L_{a,b}) \cap \ell \neq \emptyset$, then $\dist(A^n,\ell) \rightarrow 0$ as $n \rightarrow \infty$.
	\end{lem}
	
	\begin{proof}
	We distinguish between two cases, depending on the number of intersections of $W_X^u$ with $Ox$ and $Oy$.
	
	Assume first that $W_X^u$ intersects $Ox$ and $Oy$ at finitely many points. In this case, $\ell$ is the attracting periodic orbit $\{P,P'\}$. Let $\varepsilon > 0$ be such that $B_{\varepsilon}(P)$ (resp.\ $B_{\varepsilon}(P')$) is contained in the basin of attraction of $P$ (resp.\ $P'$). 
	
	Now, let $A$ be a point such that $\omega(A,L_{a,b}) \cap \{P,P'\} \neq \emptyset$. Then arbitrarily large forward iterates of $A$ enter $B_{\varepsilon}(P) \cup B_{\varepsilon}(P')$. Hence, there exists a forward iterate $A^{n_1}$ that lies in the basin of attraction of $\{P,P'\}$, so forward iterates of $A$ are attracted to $\{P,P'\}$.
	
	Now assume that $W_X^u$ intersects $Ox$ and $Oy$ at infinitely many points, and consider a point $A$ such that $\omega(A,L_{a,b}) \cap \ell \neq \emptyset$. Since $\ell=\ell_L \cup \ell_R$, we have $\omega(A,L_{a,b}) \cap \ell_R \neq \emptyset$ or $\omega(A,L_{a,b}) \cap \ell_L \neq \emptyset$.
	
	We first consider the case $\omega(A,L_{a,b}) \cap \ell_R \neq \emptyset$. By \eqref{eq:ell_intersect_D}, we have that $\ell_R \subset L_{a,b}^{2k}(\mathcal{D})$ for every $k \in \mathbb{N}_0$. On the other hand, since $\mathcal{D}$ is eventually trapping, there exists $k_0 \in \mathbb{N}$ such that $L_{a,b}^{2k_0}(\mathcal{D})\subset \Int \mathcal{D}$. Choosing $k = k_0$, we obtain $\ell_R \subset L_{a,b}^{2k_0}(\mathcal{D}) \subset \Int \mathcal{D}$, so $\omega(A,L_{a,b}) \cap \Int\mathcal{D} \neq \emptyset$. Thus, there exists a forward iterate $A^{n_2} \in \Int\mathcal{D} \subset \mathcal{D}$, so the claim of the lemma follows from Corollary \ref{cor:D_attracted_to_ell}.
	
	Now assume that $\omega(A,L_{a,b}) \cap \ell_L \neq \emptyset$. Applying $L_{a,b}$ and using $L_{a,b}(\ell_L)=\ell_R$, we obtain that $L_{a,b}(\omega(A,L_{a,b})) \cap \ell_R \neq \emptyset$. Since $L_{a,b}$ is a homeomorphism, the $\omega$-limit set is invariant under both $L_{a,b}$ and $L_{a,b}^{-1}$, that is, $L_{a,b}(\omega(A,L_{a,b})) = \omega(A,L_{a,b})$, so we see that this case reduces to the previous one, which completes the proof.
	\end{proof}
	
	Note that the fixed points $X$ and $Y$ are isolated from $\ell$; in particular, $X,Y \not\in \ell$.
	
	\begin{lem} \label{lem:X_Y_isolated_from_ell}
	There exists $\varepsilon > 0$ such that $B_{\varepsilon}(X) \cap \ell = B_{\varepsilon}(Y) \cap \ell = \emptyset$.
	\end{lem}
	
	\begin{proof}
	By the construction of $\mathcal{D}$, the boundaries $\partial\mathcal{D}$ and $\partial L_{a,b}(\mathcal{D})$ do not intersect and neither of these polygons is contained in the other. Therefore, $\mathcal{D} \cap L_{a,b}(\mathcal{D}) = \emptyset$. Now, if $X$ were contained in $\mathcal{D}$, then $L_{a,b}(X) = X \in L_{a,b}(\mathcal{D})$, so it would follow that $X \in \mathcal{D} \cap L_{a,b}(\mathcal{D}) \neq \emptyset$, which is a contradiction. Therefore, $X \notin \mathcal{D}$. The same argument, using $L_{a,b}(Y) = Y$, gives $Y \notin \mathcal{D}$.
	
	Since $\mathcal{D} \cup L_{a,b}(\mathcal{D})$ is a closed set in $\mathbb{R}^2$, there exist $\varepsilon_1, \varepsilon_2 >0$ such that $B_{\varepsilon_1}(X) \cap (\mathcal{D} \cup L_{a,b}(\mathcal{D})) = B_{\varepsilon_2}(Y) \cap (\mathcal{D} \cup L_{a,b}(\mathcal{D})) = \emptyset$. In particular, because $\ell \subset \mathcal{D} \cup L_{a,b}(\mathcal{D})$, we have $B_{\varepsilon_1}(X) \cap \ell = B_{\varepsilon_2}(Y) \cap \ell = \emptyset$. Taking $\varepsilon := \min\{\varepsilon_1, \varepsilon_2\}$ yields the desired $\varepsilon$ and completes the proof.
	\end{proof}
	
	\begin{cor} \label{cor:WuXWsY_empty}
	$W_X^u$ and $W_Y^s$ do not intersect.
	\end{cor}
	
	\begin{proof}
	Assume by contradiction that there exists a point $A \in W_X^u \cap W_Y^s$. Since $A \in W_Y^s$, we have that $A^n \rightarrow Y$ as $n \rightarrow \infty$. Moreover, $A \in W_X^u$, so the convergence $A^n \rightarrow Y$ implies that the point $Y$ is an accumulation point of $W_X^u$. Therefore, $Y \in \ell$, which contradicts Lemma \ref{lem:X_Y_isolated_from_ell}.
	\end{proof}
	
	For parameters $(a,b) \in \mathfrak{R} \setminus \mathcal{R}$, it was also shown that the non-wandering set of $L_{a,b}^2$ is contained in the union of $\ell$ and the fixed points $X$ and $Y$ (Theorem 1.1 in \cite{kilassa2026accumulation}). The same inclusion holds for parameters $(a,b) \in \mathcal{R}$. Indeed, in this case, the authors in \cite{misiurewicz2024zero} proved that $\ell = \{P,P'\}$, and that $\Omega(L_{a,b}^2)=\{X,Y,P,P'\}$. We state these results collectively as a single theorem for convenient reference in the remainder of the paper.
	
	\begin{theorem} \label{thm:L2_non_wandering}
	For all $(a,b) \in \mathfrak{R}$, we have that $\Omega(L_{a,b}^2) \subseteq \ell \cup \{X,Y\}$.
	\end{theorem}

	\subsection{Introducing the other fixed point} \label{subsec:fixed_point_Y}
	
	From now on, we consider parameters $(a,b) \in \mathfrak{R}$. The stable manifold $W_Y^s$ of $Y$ is an $L_{a,b}$-invariant and $L_{a,b}^{-1}$-invariant polygonal line. Its downward branch emanating from $Y$ intersects $Oy$ for the first time at the point
	\begin{equation*}
	T := \left(0,\: \frac{a+2b+\sqrt{a^2+4b}}{2(1-a-b)}\right).
	\end{equation*}
	Since $a+b>1$, we see that $T$ lies on $Oy^{-}$ and its forward image
	\begin{equation*}
	T^1=\left(\frac{2-a+\sqrt{a^2+4b}}{2(1-a-b)},\: 0\right)
	\end{equation*}
	lies on $Ox^{-}$; see Figure \ref{figure:stable_Y}. Moreover, $W_Y^s=\bigcup_{n=0}^{\infty}L_{a,b}^{-n}(\overline{YT})$.
	Let $W_Y^{s+}$ be the half of $W_Y^{s}$ starting at $Y$ and going down, passing through $T$. Also, let $W_Y^{s-}$ be the other half starting at $Y$ and passing through $T^1$. 
	
	Furthermore, for points $A_1,A_2\in W_Y^s$, we use the same notation for polygonal lines and line segments on $W_Y^s$ as for $W_X^s$ and $W_X^u$, except that we additionally emphasize the point $Y$. For example, $[A_1,A_2]^{s,Y}\subset W_Y^s$ denotes the polygonal line with $A_1$ and $A_2$ as endpoints. Notice that the above observations imply that $[T,T^1]^{s,Y}=\overline{TT^1}^{s,Y}$ is a straight line segment in $\mathcal{Q}_3$ containing $Y$.
	
	The main objective of this subsection is to locate $W_Y^s$ relative to $W_X^u$ and $W_X^s$. We proceed in two stages. First, we show that $W_Y^{s+}$ crosses $Ox^{+}$ at a point $C$ to the right of $Z$ (Lemmas \ref{lem:WsY_positive_x_01} and \ref{lem:WsY_positive_x_02}). Second, we show that the backward iterates of the corresponding segment of $W_Y^{s+}$ approach $W_X^{s+}$ (in the sense of Corollary \ref{cor:CR_0_preim_cvg_WXs+}) and propagate to infinity in $\mathcal{Q}_1$. These two facts imply that $W_Y^s$ separates the plane (Corollary \ref{cor:WsY_separates}).
	
	\begin{figure}[!ht]
	\includegraphics[width=\linewidth]{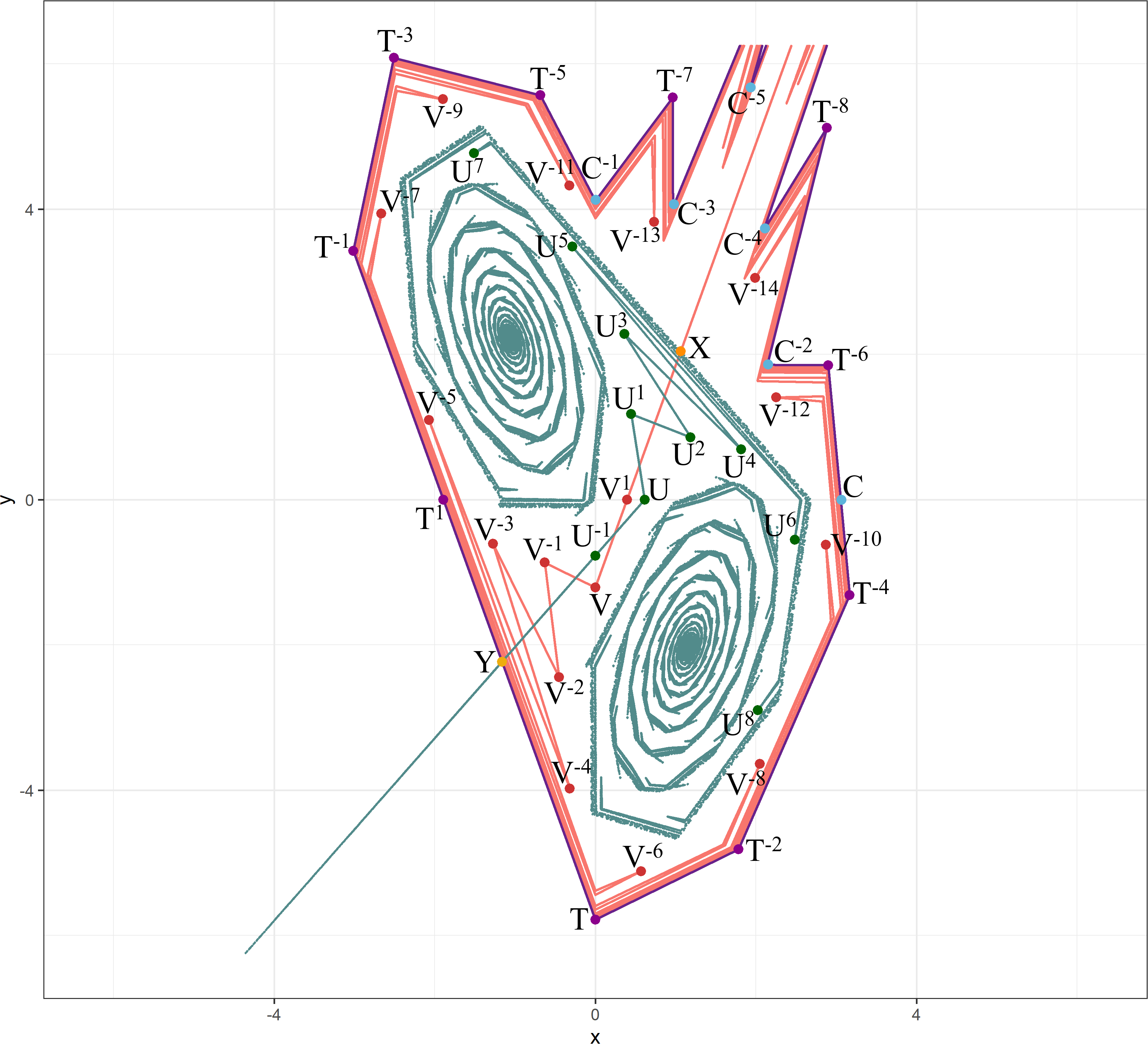}
	\caption{The stable (red) manifold of $X$ together with the stable (purple) and unstable (teal) manifold of $Y$ ($a=0.9$, $b=0.96$). Notice that $W_X^s$ accumulates on $W_Y^s$ (Lemma \ref{lem:WsXaccumWsY}). All V-points on $W_Y^s$ are iterates of points $T$ and $C$ (Lemmas \ref{lem:WsY_positive_x_01} and \ref{lem:WsY_positive_x_02}).}
	\label{figure:stable_Y}
	\end{figure}
	
	We now prove a concrete analogue of the Inclination Lemma for the Lozi map, showing that for an arc component $\alpha_0$ of $W_X^s$, the portion of $L_{a,b}^{-n}(\alpha_0)$ in $\mathcal{Q}_3$ gets arbitrarily close to $\overline{TT^{1}}^{s,Y}$ with respect to the Hausdorff distance as $n$ increases.   	 
	
	\begin{lem}[$W_X^s$ accumulates on $W_Y^s$]\label{lem:WsXaccumWsY}
	There exists a sequence $(\beta_n)_{n \in \mathbb{N}}$ of arc components of $W_X^s$ with the following property: for every $\varepsilon > 0$, there exists $n' \in \mathbb{N}$ such that $\beta_n$ and $\overline{TT^1}^{s,Y}$ are $\varepsilon$-close in the Hausdorff distance for all $n \geqslant n'$.
	\end{lem}
	
	\begin{proof}
	Let $n_0 \in \mathbb{N}$ be such that $[V,V^{-n_0}]^s$ is the zigzag part of $W_X^s$, as in Lemma \ref{lem:zigzag}. Moreover, for every $n \in \mathbb{N}_0$, let $\alpha_n := [V^{-n},V^{-n-1}]^s$, and let $\beta_n$ be the intersection of $\alpha_n$ with the third quadrant $\mathcal{Q}_3$.
	
	We have that $\alpha_{n+1} = L_{a,b}^{-1}(\alpha_n)$ for every $n \in \mathbb{N}_0$. Furthermore, $\alpha_j$ is a straight line segment for all $j \leqslant n_0$, and $\alpha_j$ intersects both coordinate axes for all $j > n_0$. In addition, $\beta_n$ is a straight line segment for every $n \in \mathbb{N}_0$, and the endpoints of $\beta_j$ lie on the coordinate axes for all $j > n_0$. Also notice that the slope of each $\beta_n$ is exactly $s_n$ from Lemma \ref{lem:lozi_recurr_coef}.
	
	Now observe the triangle $OA_1A_2$, where $O$ is the origin of the coordinate system and $A_1$, $A_2$ are the intersections of $\beta_{n_0+1}$ with $Ox$ and $Oy$ respectively. We claim that $Y$ lies outside that triangle. By contradiction, assume that $Y$ lies in the triangle. In that case, the line segment $\overline{TT^1}^{s,Y}$ intersects the zigzag part of $W_X^s$ or $\alpha_{n_0+1}$. Since $W_X^s$ and $W_Y^s$ do not intersect, this is a contradiction and we thus conclude that $Y$ is contained in $\mathcal{Q}_3$ outside the triangle $OA_1A_2$.
	
	On the other hand, consider the unstable manifold $W_Y^u$ of $Y$, more precisely, the maximal ray on $W_Y^u$ containing $Y$. That portion of $W_Y^u$ intersects $Oy$ at the point
	\begin{equation} \label{eq:U_minus1}
	U^{-1} := \left(0,\:\frac{a+2b-\sqrt{a^2+4b}}{2(1-a-b)}\right),
	\end{equation}
	and one easily verifies that $U^{-1}$ lies above $V$, that is, $U^{-1}_y > V_y$; see Figure \ref{figure:stable_Y}. Therefore, $W_Y^u$ intersects $\beta_{n_0+1}$. Since $W_Y^u$ is $L_{a,b}^{-1}$-invariant and $\beta_{n+1} \subseteq L_{a,b}^{-1}(\beta_n)$ for every $n \in \mathbb{N}_0$, it follows that $W_Y^u$ intersects every $\beta_n$. Moreover, the sequence of these intersection points is the backward orbit of a single point, and hence these intersection points converge to $Y \in \overline{TT^1}^{s,Y}$. In addition, the slopes of $\beta_n$ also converge to the slope of $\overline{TT^1}^{s,Y}$ (Lemma \ref{lem:lozi_recurr_coef}(4)). Since the endpoints of all $\beta_n$ lie on the coordinate axes, it follows that the endpoints of $\beta_n$ converge to $T$, $T^1$, respectively. Therefore, the sequence $(\beta_n)_n$ satisfies the statement of the lemma, which completes the proof.
	\end{proof}
	
	As a consequence, since $W_X^s$ and $W_Y^s$ are both $L_{a,b}^{-1}$-invariant, for every polygonal segment $[A,B]^{s,Y}\subset W_Y^s$, there is a sequence of polygonal segments on $W_X^s$ which converges to $[A,B]^{s,Y}$ with respect to the Hausdorff distance.
	
	Now observe the backward iterates of $T$. Since $T$ lies on $Oy^{-}$, its preimage $T^{-1}$ lies on $y=a|x|-1$ in the second quadrant $\mathcal{Q}_2$, so $T^{-2}$ lies in $\mathcal{Q}_4$ or $\mathcal{Q}_1$ below the line $y=ax-1$. In general, if an even backward iterate of $T$ lies in $\mathcal{Q}_4$, then its preimage lies in $\mathcal{Q}_2$ or $\mathcal{Q}_3$, and if an odd backward iterate of $T$ lies in $\mathcal{Q}_2$ below the line $x = 1 - \frac{a}{b}y$, its preimage lies in $\mathcal{Q}_4$ (we refer the reader to the diagrams in Figure \ref{fig:lozi_quadrants_images} as a visual aid). This might suggest that all even backward iterates of $T$ would lie in $\mathcal{Q}_4$ and all odd backward iterates in $\mathcal{Q}_2$, or that $T^{-n}$ would lie in $\mathcal{Q}_3$ for some $n\geqslant 2$. However, we will show in Lemmas \ref{lem:WsY_positive_x_01} and \ref{lem:WsY_positive_x_02} that neither alternative can occur. 
	
	Recall that $P$ and $P'$ denote the period-two points for $L_{a,b}$ in $\mathcal{Q}_4$ and $\mathcal{Q}_2$, respectively.
	
	\begin{lem} \label{lem:P_above_T}
	We have that $P_y > T_y$, that is, $P$ lies above $T$.
	\end{lem}
	
	\begin{proof}
	We know that
	\begin{equation*}
	P_y=\frac{b(1 - a - b)}{a^2 + (1 - b)^2},\ T_{y}=\frac{a + 2b + \sqrt{a^2 + 4b}}{2(1 - a - b)},
	\end{equation*}
	and since the first denominator is positive and the second one is negative, $P_y > T_y$ is equivalent to
	\begin{equation} \label{eq:P_above_T_01}
	(a + 2b + \sqrt{a^2 + 4b})(a^2 + (1 - b)^2) > 2b(a + b - 1)^2.
	\end{equation}
	For $a + b  > 1$ and $-1 < b < 1$, we have that
	\begin{equation*}
	a > 1 - b \implies a^2 > 1 - 2b + b^2,
	\end{equation*}
	and therefore,
	\begin{equation*}
	a + 2b + \sqrt{a^2 + 4b} > 1 + b + \sqrt{a^2 + 4b} > 1 + b + \sqrt{1 + 2b + b^2} = 2 + 2b.
	\end{equation*}
	For $b < 1$, we have $2 + 2b > 4b$, so we obtain
	\begin{equation} \label{eq:P_above_T_02}
	a + 2b + \sqrt{a^2 + 4b} > 4b.
	\end{equation}
	Furthermore, by $(a - (1 - b))^2 > 0$, we have
	\begin{equation} \label{eq:P_above_T_03}
	a^2 + (1 - b)^2 > \tfrac{1}{2}(a - b + 1)^2.
	\end{equation}
	We claim that $(a - b + 1)^2 > (a + b - 1)^2$. Indeed,
	\begin{equation*}
	\begin{split}
	(a - b + 1)^2 > (a + b - 1)^2 & \Leftrightarrow (a - b + 1)^2 - (a + b - 1)^2 > 0 \\
								 & \Leftrightarrow (a - b + 1 + a + b + 1)(a - b + 1 - a - b + 1) > 0 \\
								 & \Leftrightarrow 4a(1 - b) > 0, 
	\end{split}
	\end{equation*}
	and the last inequality holds due to $a > 0$ and $1 - b > 0$. Now, \eqref{eq:P_above_T_03} implies
	\begin{equation} \label{eq:P_above_T_04}
	a^2 + (1 - b)^2 > \tfrac{1}{2}(a + b - 1)^2.
	\end{equation}
	Multiplying \eqref{eq:P_above_T_02} and \eqref{eq:P_above_T_04} yields the desired inequality \eqref{eq:P_above_T_01}.
	\end{proof}	
	
	\begin{lem}\label{lem:WsY_positive_x_01}
	$W_Y^{s+}$ intersects $Ox^{+}$.
	\end{lem}
	
	\begin{proof}
	We will prove the lemma in two steps. Assume by contradiction that $W_Y^{s+}$ does not intersect $Ox^{+}$. Under this assumption, in the first step of the proof, we show that $W_Y^{s+}$ does not intersect $Oy^{-}$ at points other than $T$. As a consequence, $W_Y^s \setminus [T,T^1]^{s,Y}$ is contained in $\mathcal{Q}_2 \cup \mathcal{Q}_4$, from which we derive a contradiction in the second step of the proof.
	
	\noindent\emph{$1^{\circ}$ Under the assumption that $W_Y^{s+}$ does not intersect $Ox^{+}$, the only intersection of $W_Y^{s+}$ with $Oy^{-}$ is $T$}
	
	Assume by contradiction that $W_Y^{s+}$ intersects $Oy^{-}$ at points other than $T$. Then there exists $j \in \mathbb{N}$ such that $T^{-2j+2} \in \mathcal{Q}_4$ and $T^{-2j} \in \mathcal{Q}_3$; see Figure \ref{figure:prop_C_case_1}. We consider the smallest such $j$. Then $[T^{-1},T^{-2j+3}]^{s,Y} \subset \mathcal{Q}_2$ and $[T^{-2},T^{-2j+2}]^{s,Y} \subset \mathcal{Q}_4$.
	
	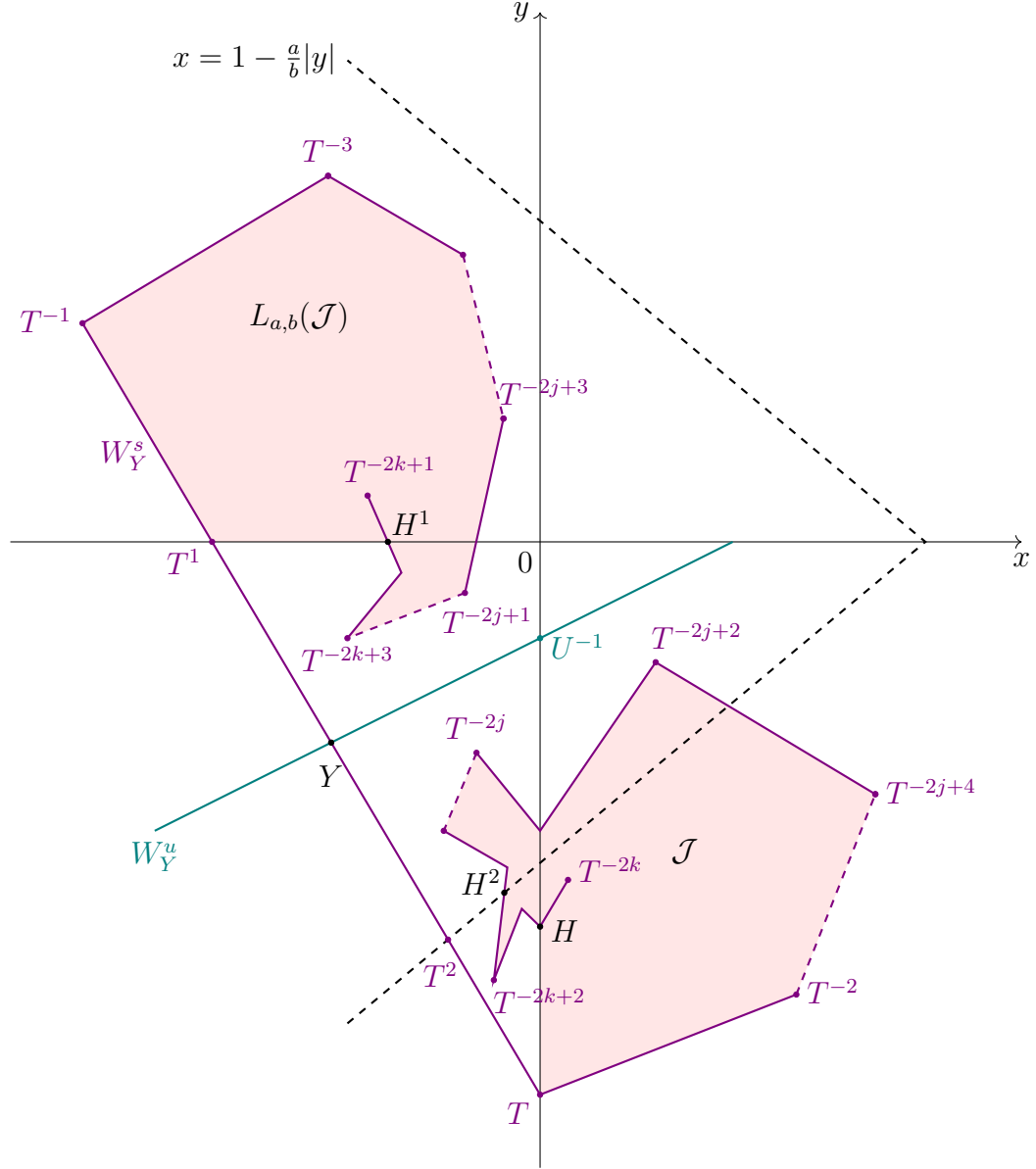
\begin{figure}[!ht]
	\begin{center}
	\begin{tikzpicture}[auto, scale=1.3]
			\tikzstyle{nodec}=[draw,circle,fill=black,minimum size=2pt,
			inner sep=0pt, label distance=2mm]
			\tikzstyle{nodeh}=[draw,circle,fill=white,minimum size=4pt,
			inner sep=0pt]
			\tikzstyle{dot}=[circle,draw=none,fill=none,minimum size=0pt,inner sep=2pt, outer sep=-1pt]

			\draw[->] (-5.5,0)--(5,0) node [below]{$x$};
			\draw[->] (0,-6.5)--(0,5.5) node [left]{$y$};
			\node[label={[xshift=-0.2cm, yshift=-0.7cm]$0$}] at (0,0) {};
			
			\coordinate (e0) at (0,0);
			\coordinate (ex) at (1,0);
			\coordinate (ey) at (0,1);
			
			\coordinate (xim1) at (-2,5);
			\coordinate (xim2) at (-2,-5);
			\coordinate (xim0) at (4,0);
			
			\draw[thick, dashed] (xim2)--(xim0)--(xim1) node[left]{$x=1-\frac{a}{b}|y|$};

			\coordinate (uY1) at (2,0);
			\coordinate (uY2) at (-4,-3);
			
			\coordinate (um1) at (intersection of e0--ey and uY1--uY2);
			
			\draw[teal, thick] (uY1)--(uY2) node[below]{$W_Y^{u}$};

			\coordinate (r2) at (0,-5.74);
			\coordinate (r3) at (2.66,-4.7);
			\coordinate (r4) at (3.48,-2.62); 
			\coordinate (r5) at (1.2,-1.25); 
			\coordinate (r6) at (0,-3);
			\coordinate (r7) at (-0.66,-2.19);
			\coordinate (r8) at (-1,-3);
			\coordinate (r9) at (-0.34,-3.38);
			\coordinate (r10) at (-0.48,-4.55); 
			\coordinate (r11) at (-0.19,-3.81);
			\coordinate (r12) at (0.29,-3.51); 
			\coordinate (r14) at (-4.75,2.27); 
			\coordinate (r15) at (-2.2,3.8); 
			\coordinate (r16) at (-0.8,2.98);
			\coordinate (r17) at (-0.38,1.28); 
			\coordinate (r18) at (-0.78,-0.53); 
			\coordinate (r19) at (-2,-1); 
			\coordinate (r20) at (-1.44,-0.32);
			\coordinate (r21) at (-2.23,0.41);

			\coordinate (y) at (intersection of uY1--uY2 and r2--r14);
			\coordinate (q) at (0,-3.995);
			\coordinate (q1) at (-1.58,0);
			\coordinate (q2) at (intersection of xim0--xim2 and r9--r10);
			\coordinate (r1) at (intersection of xim0--xim2 and r2--r14); 
			\coordinate (r13) at (intersection of e0--ex and r2--r14); 
			
			\coordinate (r22) at ($2.5*(q1)-1.5*(r20)$); 

			\draw [violet, thick] (r14) to node[dot, swap]{$W_Y^s$} (r13);
			\draw [violet, thick] (r13)--(r2)--(r3);
			\draw [violet, thick, dashed] (r3)--(r4);
			\draw [violet, thick] (r4)--(r5)--(r6)--(r7);
			\draw [violet, thick, dashed] (r7)--(r8);
			\draw [violet, thick] (r8)--(r9)--(r10)--(r11)--(q)--(r12);
			\draw [violet, thick] (r14)--(r15)--(r16);
			\draw [violet, thick, dashed] (r16)--(r17);
			\draw [violet, thick] (r17)--(r18);
			\draw [violet, thick, dashed] (r18)--(r19);
			\draw [violet, thick] (r19)--(r20)--(q1)--(r22);

			\node[nodec, color=violet] at (r15) {};
			\node[nodec, color=violet] at (r16) {};
			\node[nodec, color=violet] at (r17) {};
			\node[nodec, color=violet] at (r4) {};
			\node[nodec, color=violet] at (r8) {};
			
			\node[nodec, color=black, label={[below, yshift=-2mm, color=black]$Y$}] at (y) {};
			\node[nodec, color=violet, label={[below left, color=violet]$T$}] at (r2) {};
			\node[nodec, color=violet, label={[below left, color=violet]$T^1$}] at (r13) {};
			\node[nodec, color=violet, label={[below, yshift=-2mm, xshift=-1mm, color=violet]$T^{2}$}] at (r1) {};
			\node[nodec, color=violet, label={[left, color=violet]$T^{-1}$}] at (r14) {};
			\node[nodec, color=violet, label={[above, color=violet]$T^{-3}$}] at (r15) {};
			\node[nodec, color=violet, label={[right, color=violet]$T^{-2}$}] at (r3) {};
			\node[nodec, color=violet, label={[right, color=violet]$T^{-2j+4}$}] at (r4) {};
			\node[nodec, color=violet, label={[above right, xshift=-2mm, color=violet]$T^{-2j+2}$}] at (r5) {};
			\node[nodec, color=violet, label={[above, color=violet]$T^{-2j}$}] at (r7) {};
			\node[nodec, color=violet, label={[below right, xshift=-1.5mm, yshift=0mm, color=violet]$T^{-2k+2}$}] at (r10) {};
			\node[nodec, color=violet, label={[right, yshift=1mm, color=violet]$T^{-2k}$}] at (r12) {};
			\node[nodec, color=violet, label={[above right, yshift=0mm, xshift=-2mm, color=violet]$T^{-2j+3}$}] at (r17) {};
			\node[nodec, color=violet, label={[below, yshift=-1mm, xshift=3mm, color=violet]$T^{-2j+1}$}] at (r18) {};
			\node[nodec, color=violet, label={[below, color=violet]$T^{-2k+3}$}] at (r19) {};
			\node[nodec, color=violet, label={[above, xshift=3mm, color=violet]$T^{-2k+1}$}] at (r22) {};
			
			\node[nodec, color=teal, label={[right, xshift=0mm, yshift=-1.5mm, color=teal]$U^{-1}$}] at (um1) {};
			
			\node[nodec, color=black, label={[right, yshift=-1mm, color=black]$H$}] at (q) {};
			\node[nodec, color=black, label={[above right, xshift=-1mm, yshift=-1mm, color=black]$H^1$}] at (q1) {};
			\node[nodec, color=black, label={[left, xshift=1mm, yshift=1mm, color=black]$H^2$}] at (q2) {};
			
			\node[dot, label=$\mathcal{J}$] (jj) at (1.5,-3.5) {};
			\node[dot, label=$L_{a,b}(\mathcal{J})$] (ljj) at (-2.5,2) {};

			\begin{pgfonlayer}{bg}
				\fill[red!10!white] (r2.center)--(r3.center)--(r4.center)--(r5.center)--(r6.center)--(r7.center)--(r8.center)--(r9.center)--(r10.center)--(r11.center)--(q.center)--cycle;
				\fill[red!10!white] (r13.center)--(r14.center)--(r15.center)--(r16.center)--(r17.center)--(r18.center)--(r19.center)--(r20.center)--(q1.center)--cycle;
			\end{pgfonlayer}	
			
			\end{tikzpicture}
	\end{center}
	\caption{Diagram illustrating step $1^{\circ}$ in the proof of Lemma \ref{lem:WsY_positive_x_01}: $W_Y^s$ does not intersect $Oy^{-}$ at points other than $T$.}
	\label{figure:prop_C_case_1}
	\end{figure}
	
	The argument has three parts. First, we locate $T^{-2j}$ inside $OTT^1$ (Claim A). Second, using the expansion transverse to $YU^{-1}$, we find the next intersection $H$ of $W_Y^{s+}$ with $Oy^{-}$. Third, we construct a bounded region $\mathcal{J}$ with $L_{a,b}^{-2}(\mathcal{J}) \subseteq \mathcal{J}$, for which we show that it contains $P$, leading to a contradiction with the dynamics of $W_X^u$.
	
\noindent\textbf{Claim A.} $T^{-2j}$ lies in the triangle $OTT^1$. \\
\textbf{Proof.} By the construction of $j$, the segment $[T^{-2j+2},T^{-2j}]^{s,Y}$ intersects $Oy^{-}$, and $[T^{-2j+3},T^{-2j+1}]^{s,Y}$ intersects $Ox^{-}$. The segment $[T^{-2j+3},T^{-2j+1}]^{s,Y}$ lies in $L_{a,b}^{-1}(\mathcal{Q}_4)$ because it is the preimage of $[T^{-2j+4},T^{-2j+2}]^{s,Y} \subset \mathcal{Q}_4$. Therefore, $[T^{-2j+3},T^{-2j+1}]^{s,Y}$ lies in the portion of $\mathcal{Q}_2 \cup \mathcal{Q}_3$ above the line $y=-ax-1$.

On the other hand, the point $T^{1}$ lies below the line $y=-ax-1$, so $[T^{-2j+3},T^{-2j+1}]^{s,Y}$ intersects $Ox^{-}$ to the right of $T^1$. As a consequence, $[T^{-2j+2},T^{-2j}]^{s,Y}$ intersects $Oy^{-}$ above $T$, which implies that $T^{-2j}$ lies in the triangle $OTT^1$. This completes the proof of Claim A. \hfill $\blacksquare$

We have that $L_{a,b}^{-1}(OTT^{1}) \cap \mathcal{Q}_3 \subset OTT^{1}$. Now, assume that $T^{-2j},T^{-2j-1},T^{-2j-n}$ all lie in $\mathcal{Q}_3$ for some $n \in \mathbb{N}$. Then, by Claim A, these points all lie in $OTT^1$. In this case, $L_{a,b}^{-1}$ acts on these points as an affine map, so for every $i=0,1,\ldots,n$, we have
	\begin{equation*}
	\dist(T^{-2j-i}, YU^{-1}) = \left(\frac{1}{|\lambda_Y^s|}\right)^{i} \dist(T^{-2j}, YU^{-1}),
	\end{equation*}
where $U^{-1}$ is the first intersection of $W_Y^u$ with $Oy$ given by \eqref{eq:U_minus1}. 

Since $1 / |\lambda_Y^s| > 1$, distances of $T^{-2j-i}$ to the straight line $YU^{-1}$ increase geometrically, so there exists $k \in \mathbb{N}$ such that $T^{-2k+3},T^{-2k+2} \in \mathcal{Q}_3$, $T^{-2k+1} \in \mathcal{Q}_2$ and $T^{-2k} \in \mathcal{Q}_4$. In other words, $[T^{-2k+2},T^{-2k}]^{s,Y}$ is the next polygonal segment on $W_Y^{s+}$ to intersect $Oy^{-}$ after $[T^{-2j+2},T^{-2j}]^{s,Y}$.

	Let $H$ be the intersection of $[T^{-2k+2},T^{-2k}]^{s,Y}$ and $Oy^{-}$. Let $\mathcal{J}$ be the polygon enclosed by the boundary $\partial \mathcal{J} = \overline{TH} \cup [T,H]^{s,Y}$.
	
\noindent\textbf{Claim B.} $L_{a,b}^{-2}(\mathcal{J}) \subseteq \mathcal{J}$. \\
\textbf{Proof.} It suffices to show that $L_{a,b}^{-2}(\partial\mathcal{J}) \subset \mathcal{J}$. We first show that $L_{a,b}^{-1}([T,H]^{s,Y}) \subset \mathcal{J}$.

	The segment $[H,T^{-2k}]^{s,Y}$ is contained in $\mathcal{J}$. In order for $W_Y^{s+}$ to escape $\mathcal{J}$, $[T^{-2l},T^{-2l-2}]^{s,Y}$ has to intersect the straight line segment $\overline{HT}$ for some $l \geqslant k$. That would imply that $L_{a,b}^{-2}([T^{-2l},T^{-2l-2}]^{s,Y}) = [T^{-2l-2},T^{-2l-4}]^{s,Y}$, a polygonal line contained in $\mathcal{J}$, intersects $\overline{T^2H^2}$, which is contained outside $\mathcal{J}$. This is a contradiction, so we conclude that $W_Y^{s+}\setminus[Y,T^{-2k}]^{s,Y}$ is contained in $\mathcal{J}$. In particular, $L_{a,b}^{-2}([T,H]^{s,Y}) \subset \mathcal{J}$.
	
	It remains to show that $L_{a,b}^{-2}(\overline{TH}) \subset \mathcal{J}$. The endpoints $T^{-2},H^{-2}$ of $L_{a,b}^{-2}(\overline{TH})$ both lie in $\mathcal{J}$, so it suffices to show that $L_{a,b}^{-2}(\overline{TH})$ does not intersect $\partial\mathcal{J}$ apart from its endpoints.
	
	Since $\overline{T^2H^2}=L_{a,b}^2(\overline{TH})$ lies outside $\mathcal{J}$, we have that $L_{a,b}^2(\overline{TH}) \cap \overline{TH} = \emptyset$. Applying $L_{a,b}^{-2}$, we obtain that $\overline{TH} \cap L_{a,b}^{-2}(\overline{TH}) = \emptyset$. On the other hand, $\overline{TH}$ does not intersect $W_Y^s$ apart from the endpoints $T,H$, so the same holds for $L_{a,b}^{-2}(\overline{TH})$. Therefore, $L_{a,b}^{-2}(\overline{TH})$ does not intersect $\partial\mathcal{J}$ apart from its endpoints, which implies $L_{a,b}^{-2}(\overline{TH}) \subset \mathcal{J}$. As a consequence, $L_{a,b}^{-2}(\partial\mathcal{J}) \subset \mathcal{J}$, which completes the proof of Claim B. \hfill $\blacksquare$
	
	By Claim B, $L_{a,b}^{-2}$ restricts to a continuous self-map of $\mathcal{J}$. Since the polygon $\mathcal{J}$ is homeomorphic to the closed unit disk, the Brouwer Fixed Point Theorem (BFPT) and Claim B imply that $\mathcal{J}$ contains a fixed point of $L_{a,b}^{-2}$. If that fixed point were $Y$, then $[T^{-2j+2},H]^{s,Y}$ would intersect $\overline{T^1T}^{s,Y}$, which is not possible since $W_Y^s$ does not self-intersect. Moreover, $X$ and $P'$, the remaining two fixed points of $L_{a,b}^{-2}$, lie in the upper half-plane. We thus conclude that $\mathcal{J}$ contains $P$. The point $P$ does not lie on $\partial\mathcal{J}$ because it does not lie on $Oy$ or $W_Y^s$. Hence, $P \in \Int\mathcal{J}$, and there exists $\varepsilon > 0$ such that $B_{\varepsilon}(P) \subset \Int\mathcal{J}$.
	
	However, the attracting periodic orbit $\{P,P'\}$ is contained in the accumulation set $\ell$ of $W_X^u$, so $B_{\varepsilon}(P) \cap W_X^u \neq \emptyset$. Therefore, $W_X^u$ intersects $\Int\mathcal{J}$. Let $A \in W_X^u \cap \Int\mathcal{J}$. Since $L_{a,b}^{-2}(\mathcal{J}) \subseteq \mathcal{J}$, we have that $A^{-2n} \in \mathcal{J}$ for every $n \in \mathbb{N}_0$. But $A \in W_X^u$, so $A^{-2n} \rightarrow X$ as $n \rightarrow \infty$. Since $\mathcal{J}$ is closed, $X \in \mathcal{J}$, which contradicts $X \notin \mathcal{J}$. Therefore, the only intersection of $W_Y^{s+}$ with $Oy$ is the point $T$. This completes the first step of the proof.
	
\noindent\emph{$2^{\circ}$ Considering even and odd backward orbits of $T$; end of proof} 

	We have shown that under the assumption that $W_Y^{s+}$ does not intersect $Ox^{+}$, the only intersection of $W_Y^{s+}$ and $Oy^{-}$ is $T$. Therefore, $W_Y^{s+}\setminus[Y,T]^{s,Y} \subset \mathcal{Q}_4$. In particular, considering the backward iterates of $T$, we have that $\mathcal{O}^{-}_{\text{even}}(T) := \{T^{-2k}\colon k \in \mathbb{N}\}$ is contained in $\mathcal{Q}_4$, and $\mathcal{O}^{-}_{\text{odd}}(T) := \{T^{-2k+1} \colon k \in \mathbb{N}\}$ is contained in $\mathcal{Q}_2$. Moreover, the set $\mathcal{O}^{-}_{\text{odd}}(T)$ is bounded since it is contained in the bounded triangular region enclosed by $Oy$, $L_{a,b}^{-1}(Oy^{-})$ and $L_{a,b}(Ox^{+})$. Since $L_{a,b}^{-1}(\mathcal{O}^{-}_{\text{odd}}(T))=\mathcal{O}^{-}_{\text{even}}(T)$, and $L_{a,b}^{-1}$ acts on $\mathcal{O}^{-}_{\text{odd}}(T)$ as an affine map, $\mathcal{O}^{-}_{\text{even}}(T)$ is bounded as well.

	In particular, the sequence $(T^{-2k})_{k \in \mathbb{N}}$ is contained in a compact set $\Cl\mathcal{O}^{-}_{\text{even}}(T)$. Hence, there exists a convergent subsequence $(T^{-2k_m})_{m \in \mathbb{N}}$ and a point $I$ such that $T^{-2k_m} \rightarrow I$ as $m \rightarrow \infty$. By definition, $I \in \omega(T,L_{a,b}^{-2})$, and consequently, $I \in \Omega(L_{a,b}^{-2})$. Since $L_{a,b}^{2}$ is a homeomorphism, it follows that $\Omega(L_{a,b}^{-2}) = \Omega(L_{a,b}^{2})$, so $I \in \Omega(L_{a,b}^{2})$.
	
	Now, by Theorem \ref{thm:L2_non_wandering}, we have that $\Omega(L_{a,b}^2) \subseteq \ell \cup \{X,Y\}$. Since $I \in \Cl\mathcal{O}^{-}_{\text{even}}(T) \subset \mathcal{Q}_4$, we see that $I$ cannot be $X$ or $Y$. Hence, $I \in \ell$. Since $I \in \omega(T,L_{a,b}^{-2})$, it follows that $\omega(T,L_{a,b}^{-2}) \cap \ell \neq \emptyset$.
	
	We now proceed as in Lemma \ref{lem:approaching_ell}. If $W_X^u$ intersects $Ox$ and $Oy$ at finitely many points, then $\ell=\{P,P'\}$, and arbitrarily large even backward iterates of $T$ enter the basin of attraction of $\{P,P'\}$. If $W_X^u$ intersects $Ox$ and $Oy$ at infinitely many points, the same argument applies for $\Int\mathcal{D}$, where $\mathcal{D}$ is the polygon enclosed by the boundary given by \eqref{eq:polygon_D_boundary}. In either case, we conclude that sufficiently large even backward iterates $T^{-2k}$ of $T$ enter the basin of attraction of $\ell$. Since $\ell$ is $L_{a,b}$-invariant and $L_{a,b}^{-1}$-invariant, its basin is $L_{a,b}$-invariant. Since $L_{a,b}^{2k}(T^{-2k}) = T$, it follows that $T$ itself belongs to the basin of attraction of $\ell$. 
	
	Thus the assumption $W_Y^{s+} \cap Ox^{+} = \emptyset$ ultimately forces $T$ to be attracted to $\ell$, that is, $\dist(T^n,\ell) \rightarrow 0$ as $n \rightarrow \infty$. This is impossible, because $T \in W_Y^s$ and therefore $T^n \rightarrow Y$ as $n \rightarrow \infty$, while $Y \not\in \ell$ (Lemma \ref{lem:X_Y_isolated_from_ell}). Therefore, the assumption $W_Y^{s+} \cap Ox^{+} = \emptyset$ is wrong, which completes the proof of the lemma.
	\end{proof}
	
	Recall that $Z \in W_X^u$ given by \eqref{eq:Z} is the first intersection of $W_X^u$ with $Ox$, starting from $X$.
	
	\begin{lem}\label{lem:WsY_positive_x_02}
	$W_Y^{s+}$ intersects $Ox^{+}$ to the right of $Z$.
	\end{lem}
	
	\begin{proof}
	In Lemma \ref{lem:WsY_positive_x_01}, we have proved that $W_Y^{s+}$ intersects $Ox^{+}$. Therefore, there exists $j \in \mathbb{N}$ such that $[T^{-2j + 2},T^{-2j}]^{s,Y} \cap Ox^{+} \neq \emptyset$. For the smallest such $j$, odd backward iterates $T^{-1},T^{-3},\ldots,T^{-2j + 1}$ lie in $\mathcal{Q}_2$, even backward iterates $T^{-2},\ldots,T^{-2j + 2}$ lie in $\mathcal{Q}_4$, and $T^{-2j}$ lies in $\mathcal{Q}_1$. Indeed, if a backward iterate $T^{-i}$ lied in $\mathcal{Q}_3$ for some $i \in \{0,1,\ldots,2j-1\}$, we could construct a polygon $\mathcal{J}$ and obtain a contradiction as in the proof of Lemma \ref{lem:WsY_positive_x_01}, step $1^{\circ}$.
	
	In addition, for every $i \in \{0,1,\ldots,2j - 2\}$, we have that $[T^{-i},T^{-i - 2}]^{s,Y} = \overline{T^{-i}T^{-i - 2}}^{s,Y}$. Let $C$ denote the intersection of $[T^{-2j + 2},T^{-2j}]^{s,Y}=\overline{T^{-2j + 2}T^{-2j}}^{s,Y}$ with $Ox^{+}$. By contradiction, assume that $C$ lies to the left of $Z$.
	
	We define a polygon $\mathcal{K}$ bounded by portions of $W_Y^s$, $W_X^u$, and the coordinate axes. We will show first that the relevant portion of the preimage of $Oy$ lies in $\mathcal{K}$, then that $\mathcal{K}$ is $L_{a,b}^{-1}$-invariant, and finally that $P \in \mathcal{K}$. The last fact contradicts the geometry of $W_X^u$. 
	
	Let $\mathcal{K}$ be the polygon enclosed by the boundary
	\begin{equation*}
	\partial\mathcal{K} := [C^{-1},C]^{s,Y} \cup \overline{CZ} \cup \overline{ZZ^{-1}}^{u} \cup \overline{Z^{-1}C^{-1}},
	\end{equation*}

see Figure \ref{figure:prop_C_right}. Moreover, let $\alpha := \overline{T^{-1}O^{-1}} \cup \overline{O^{-1}Z^{-2}} \subset L_{a,b}^{-1}(Oy)$.

	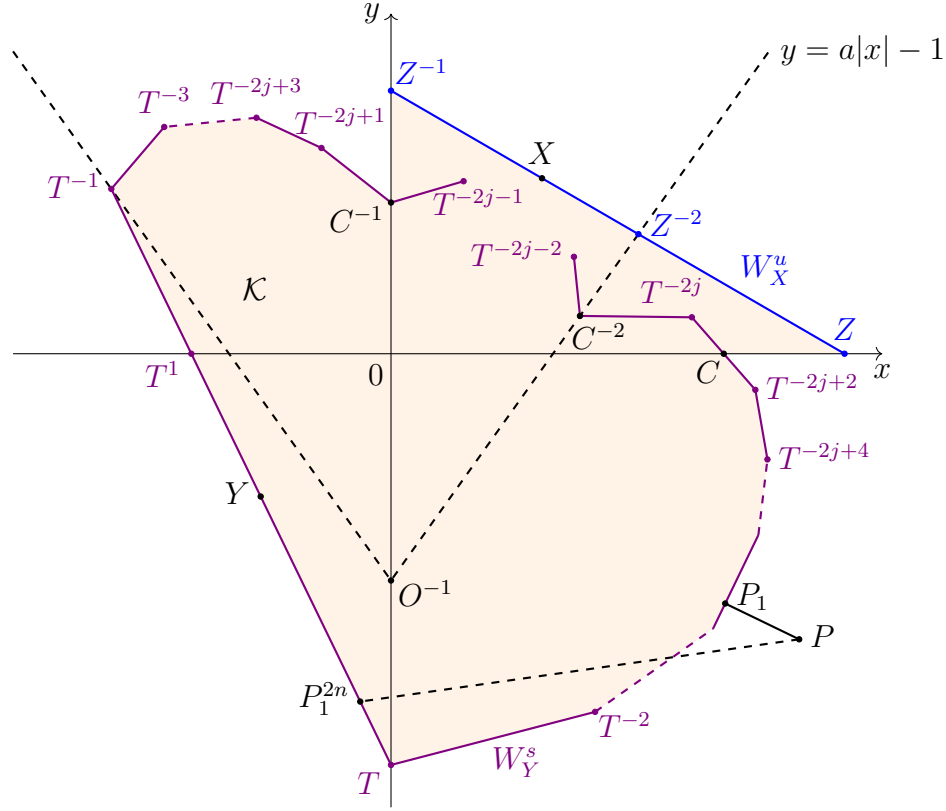
\begin{figure}[!ht]
	\begin{center}
	\begin{tikzpicture}[auto, scale=1]
			\tikzstyle{nodec}=[draw,circle,fill=black,minimum size=2pt,
			inner sep=0pt, label distance=2mm]
			\tikzstyle{nodeq}=[draw,circle,fill=black,minimum size=5pt,
			inner sep=0pt, label distance=2mm]
			\tikzstyle{nodeh}=[draw,circle,fill=white,minimum size=4pt,
			inner sep=0pt]
			\tikzstyle{dot}=[circle,draw=none,fill=none,minimum size=0pt,inner sep=2pt, outer sep=-1pt]

			\draw[->] (-5,0)--(6.5,0) node [below]{$x$};
			\draw[->] (0,-6)--(0,4.5) node [left]{$y$};
			\node[label={[xshift=-0.2cm, yshift=-0.7cm]$0$}] at (0,0) {};
			
			\coordinate (e0) at (0,0);
			\coordinate (ex) at (1,0);
			\coordinate (ey) at (0,1);
			
			\coordinate (yend1) at (5,4);
			\coordinate (yend2) at (-5,4);
			\coordinate (lm1O) at (0,-3);
			\node[nodec, color=black, label={[below right, color=black, xshift=-0.5mm, yshift=1mm]$O^{-1}$}] at (lm1O) {};
			
			\draw[thick, dashed] (yend2)--(lm1O)--(yend1) node[right]{$y=a|x|-1$};
			
			
			\coordinate (z) at (6,0);
			\coordinate (zm1) at (0,3.48);
			\coordinate (zm2) at (intersection of yend1--lm1O and z--zm1);
			\coordinate (x) at ($(z)!0.667!(zm1)$);
			
			\draw[blue, thick] (z) to node[dot, swap]{$W_X^u$} (zm2);
			\draw[blue, thick] (zm2)--(zm1);
			\node[nodec, color=blue, label={[above, color=blue]$Z$}] at (z) {};
			\node[nodec, color=blue, label={[above right, color=blue, xshift=-1mm, yshift=-1mm]$Z^{-1}$}] at (zm1) {};
			\node[nodec, color=blue, label={[right, color=blue, yshift=1mm]$Z^{-2}$}] at (zm2) {};
			\node[nodec, color=black, label={[above, color=black]$X$}] at (x) {};
			
			
			\coordinate (a1) at (0.96,2.282); 
			\coordinate (a2) at (0,2); 
			\coordinate (a3) at (-0.92,2.722); 
			\coordinate (a4) at (-1.78,3.122); 
			\coordinate (a5) at (-3,3); 
			\coordinate (a6) at ($(lm1O)!0.74!(yend2)$); 
			\coordinate (a7) at (0,-5.44); 
			\coordinate (a8) at (2.7,-4.738); 
			\coordinate (a9) at (4.26,-3.638);
			\coordinate (a10) at (4.86,-2.398);
			\coordinate (a11) at (4.98,-1.398); 
			\coordinate (a12) at (4.82,-0.478); 
			\coordinate (a13) at (3.98,0.482); 
			\coordinate (a14) at ($(lm1O)!0.5!(yend1)$); 
			\coordinate (a15) at (2.42,1.282); 
			
			\draw[violet, thick] (a1)--(a2)--(a3)--(a4);
			\draw[violet, thick, dashed] (a4)--(a5);
			\draw[violet, thick] (a5)--(a6)--(a7);
			\draw[violet, thick] (a7) to node[dot, swap]{$W_Y^s$} (a8);
			\draw[violet, thick, dashed] (a8)--(a9);
			\draw[violet, thick] (a9)--(a10);
			\draw[violet, thick, dashed] (a10)--(a11);
			\draw[violet, thick] (a11)--(a12)--(a13)--(a14)--(a15);
			
			\coordinate (t1) at (intersection of a6--a7 and e0--ex); 
			\coordinate (y) at ($(a7)!0.466!(a6)$); 
			\coordinate (p12n) at ($(a7)!0.11!(a6)$); 
			
			\coordinate (p) at (5.4,-3.778); 
			\coordinate (p1) at ($(a9)!(p)!(a10)$); 
			
			\coordinate (c) at (intersection of a12--a13 and e0--ex); 
			
			\node[nodec, color=violet, label={[below, color=violet, xshift=2mm]$T^{-2j-1}$}] at (a1) {};
			\node[nodec, color=black, label={[below left, color=black, xshift=0.5mm, yshift=0.5mm]$C^{-1}$}] at (a2) {};
			\node[nodec, color=violet, label={[above, color=violet, xshift=2.5mm]$T^{-2j+1}$}] at (a3) {};
			\node[nodec, color=violet, label={[above, color=violet]$T^{-2j+3}$}] at (a4) {};
			\node[nodec, color=violet, label={[above, color=violet]$T^{-3}$}] at (a5) {};
			\node[nodec, color=violet, label={[left, color=violet]$T^{-1}$}] at (a6) {};
			\node[nodec, color=violet, label={[below left, color=violet]$T$}] at (a7) {};
			\node[nodec, color=violet, label={[below right, color=violet, xshift=-1mm, yshift=1mm]$T^{-2}$}] at (a8) {};
			\node[nodec, color=violet, label={[right, color=violet]$T^{-2j+4}$}] at (a11) {};
			\node[nodec, color=violet, label={[right, color=violet]$T^{-2j+2}$}] at (a12) {};
			\node[nodec, color=violet, label={[above left, color=violet, xshift=3mm]$T^{-2j}$}] at (a13) {};
			\node[nodec, color=black, label={[below, color=black, xshift=2.5mm]$C^{-2}$}] at (a14) {};
			\node[nodec, color=violet, label={[left, color=violet]$T^{-2j-2}$}] at (a15) {};
			
			\node[nodec, color=black, label={[below left, color=black, xshift=1mm]$C$}] at (c) {};
			\node[nodec, color=black, label={[left, color=black]$Y$}] at (y) {};
			\node[nodec, color=violet, label={[below left, color=violet]$T^{1}$}] at (t1) {};
			
			\node[nodec, color=black, label={[right, color=black]$P$}] at (p) {};
			\node[nodec, color=black, label={[right, color=black, yshift=0.5mm]$P_1$}] at (p1) {};
			\node[nodec, color=black, label={[left, color=black]$P_1^{2n}$}] at (p12n) {};
			
			\draw[black, thick] (p)--(p1);
			\draw[black, thick, dashed] (p)--(p12n);
			
			\node[dot, label=$\mathcal{K}$] (kk) at (-1.8,0.5) {};

			\begin{pgfonlayer}{bg}
				\fill[orange!10!white] (zm1.center)--(a2.center)--(a3.center)--(a4.center)--(a5.center)--(a6.center)--(a7.center)--(a8.center)--(a9.center)--(a10.center)--(a11.center)--(a12.center)--(c.center)--(z.center)--cycle;
			\end{pgfonlayer}
			
			\end{tikzpicture}
	\end{center}
	\caption{Lemma \ref{lem:WsY_positive_x_02}: $C$, the intersection of $W_Y^{s+}$ with $Ox^{+}$, lies to the right of $Z$.}
	\label{figure:prop_C_right}
	\end{figure}
	
\noindent\textbf{Claim A.} $\alpha \subset \mathcal{K}$.\\
\textbf{Proof.} We claim that $\alpha \cap \partial\mathcal{K} = \{T^{-1},Z^{-2}\}$. Indeed, $\alpha \cap \overline{ZZ^{-1}}^u = \{Z^{-2}\}$ and $\alpha \cap \overline{TT^{-1}}^{s,Y} = \{T^{-1}\}$. 

	We now show that $\alpha \cap [T,C]^{s,Y} = \emptyset$. By construction, points $T,T^{-2},\ldots,T^{-2j + 2},C$ all lie in $L_{a,b}^{-1}(\mathcal{Q}_2)$, which is a convex set. Therefore, $[T,C]^{s,Y} \subset L_{a,b}^{-1}(\mathcal{Q}_2)$, so $[T,C]^{s,Y} \cap L_{a,b}^{-1}(Oy^{-}) = \emptyset$. Moreover, since $[T^2,T^{-2j + 2}]^{s,Y} \cap Ox^{+} = \emptyset$, it follows that
	\begin{equation*}
	L_{a,b}^{-2}([T^2,T^{-2j + 2}]^{s,Y}) \cap L_{a,b}^{-2}(Ox^{+}) = [T,T^{-2j}]^{s,Y} \cap L_{a,b}^{-2}(Ox^{+}) = \emptyset.
	\end{equation*}
Since $[T,C]^{s,Y} \subset [T,T^{-2j}]^{s,Y}$ and $L_{a,b}^{-1}(Oy^{+}) \subset L_{a,b}^{-2}(Ox^{+})$, we see that $[T,C]^{s,Y} \cap L_{a,b}^{-1}(Oy^{+}) = \emptyset$. Therefore, $[T,C]^{s,Y}$ does not intersect $L_{a,b}^{-1}(Oy)$, so $[T,C]^{s,Y}$ does not intersect $\alpha$ either.
	
	We similarly obtain $\alpha \cap [T^{-1},C^{-1}]^{s,Y} = \{T^{-1}\}$ because $[T^{-1},C^{-1}]^{s,Y}$ is a polygonal segment contained in $\mathcal{Q}_2 \cap L_{a,b}^{-1}(\mathcal{Q}_4)$.
	
	Since $\overline{CZ} \subset Ox^{+} \cap L_{a,b}^{-1}(\mathcal{Q}_2)$ and $\overline{C^{-1}Z^{-1}} \subset Oy^{+}$, we also have that $\overline{CZ} \cap L_{a,b}^{-1}(Oy) = \overline{C^{-1}Z^{-1}} \cap L_{a,b}^{-1}(Oy) = \emptyset$. As a consequence, $\overline{CZ} \cap \alpha = \overline{C^{-1}Z^{-1}} \cap \alpha = \emptyset$, which proves that $\alpha \cap \partial\mathcal{K} = \{T^{-1},Z^{-2}\}$.
	
	Finally, $O^{-1} \in \mathcal{K}$ because $O^{-1} \in Oy$ and $O^{-1}_y > T_y$. Since $\alpha$ does not intersect $\partial\mathcal{K}$ apart from its endpoints, it follows that $\alpha \subset \mathcal{K}$. This completes the proof of Claim A. \hfill $\blacksquare$ 
	
\noindent\textbf{Claim B.} $L_{a,b}^{-1}(\mathcal{K}) \subseteq \mathcal{K}$.\\
\textbf{Proof.} It suffices to show that $L_{a,b}^{-1}(\partial\mathcal{K}) \subset \mathcal{K}$. We have that $L_{a,b}^{-1}(\overline{ZZ^{-1}}^{u})=\overline{Z^{-2}Z^{-1}}^{u} \subset \overline{ZZ^{-1}}^{u}$ and $L_{a,b}^{-1}(\overline{CZ}) = \overline{C^{-1}Z^{-1}}$, so $L_{a,b}^{-1}(\overline{ZZ^{-1}}^{u})$ and $L_{a,b}^{-1}(\overline{CZ})$ are both contained in $\mathcal{K}$.

	We now show that $L_{a,b}^{-1}(\overline{C^{-1}Z^{-1}}) \subset \mathcal{K}$. The segment $\overline{C^{-1}Z^{-1}}$ lies on $\overline{Z^{-1}T} \subset Oy$, so the preimage $L_{a,b}^{-1}(\overline{C^{-1}Z^{-1}})$ lies on $\alpha$. Therefore, Claim A implies that $L_{a,b}^{-1}(\overline{C^{-1}Z^{-1}}) \subset \mathcal{K}$.

	It remains to show that $L_{a,b}^{-1}([C^{-1},C]^{s,Y}) \subset \mathcal{K}$. Since
	\begin{equation*}
	L_{a,b}^{-1}([C^{-1},C]^{s,Y}) = [C^{-1},C^{-2}]^{s,Y} = [C^{-1},C]^{s,Y} \cup [C,C^{-2}]^{s,Y},
	\end{equation*}
and $[C^{-1},C]^{s,Y} \subset \partial\mathcal{K}$, we need to show that $[C,C^{-2}]^{s,Y} \subset \mathcal{K}$.
	
	The segment $[C,C^{-2}]^{s,Y}$ does not intersect $[C^{-1},C]^{s,Y}$ apart from $C$ since $W_Y^s$ does not self-intersect, and it does not intersect $\overline{ZZ^{-1}}^u \subset W_X^u$ due to Corollary \ref{cor:WuXWsY_empty}. Moreover, $[C,C^{-2}]^{s,Y} \subset L_{a,b}^{-1}(\mathcal{Q}_2)$, so that segment does not intersect $\overline{C^{-1}Z^{-1}} \subset Oy^{+}$ either. It remains to consider the intersection of $[C,C^{-2}]^{s,Y}$ and $\overline{CZ}$.
	
	We know that $[C,C^{-2}]^{s,Y} = \overline{CT^{-2j}}^{s,Y} \cup \overline{T^{-2j}C^{-2}}^{s,Y}$. Since $T^{-2j}$ lies in $\mathcal{Q}_1$ inside the triangle $OZZ^{-1} \subset \mathcal{K}$, we have that $\overline{CT^{-2j}}^{s,Y} \subset \mathcal{K}$. Furthermore, if $\overline{T^{-2j}C^{-2}}^{s,Y}$ intersected $\overline{CZ}$, then, starting from $T^{-2j} \in \Int\mathcal{K}$, the segment would cross the boundary $\partial\mathcal{K}$ through $\overline{CZ}$. Since it cannot cross the other boundary component $[T,C]^{s,Y}$, its endpoint $C^{-2}$ would have to lie outside $\mathcal{K}$. But $C^{-2} \in \alpha \subset \mathcal{K}$, a contradiction. Therefore, $\overline{T^{-2j}C^{-2}}^{s,Y}$ does not intersect $\overline{CZ}$, so $\overline{T^{-2j}C^{-2}}^{s,Y} \subset \mathcal{K}$. 
	
	We conclude that $[C,C^{-2}]^{s,Y} \subset \mathcal{K}$, from which it follows that $L_{a,b}^{-1}([C^{-1},C]^{s,Y}) \subset \mathcal{K}$. This completes the proof of Claim B. \hfill $\blacksquare$
	
	Since $\overline{YT}^{s,Y} \subset \mathcal{K}$, Claim B implies that $W_Y^s \subset \mathcal{K}$.
	
\noindent\textbf{Claim C.} $\{P,P'\} \subset \mathcal{K}$.\\
\textbf{Proof.} Assume by contradiction that $P$ does not lie in $\mathcal{K}$. Let $\mathcal{M}$ be the polygon with boundary $\partial\mathcal{M} := [T,C]^{s,Y} \cup \overline{CO} \cup \overline{OT}$; in other words, $\mathcal{M} = \mathcal{K} \cap \mathcal{Q}_4$. Then $P$ lies in $\mathcal{Q}_4$ outside $\mathcal{M}$.

	Since $[T,C]^{s,Y}$ is compact, there exists a point $P_1 \in [T,C]^{s,Y}$ minimizing the distance to $P$, that is,
	\begin{equation*}
	\dist(P,P_1) = \min\{\dist(P,A) \colon A \in [T,C]^{s,Y}\}.
	\end{equation*}
By construction, $\overline{PP_1} \cap [T,C]^{s,Y} = \{P_1\}$. Since $P \notin \mathcal{M}$, the segment $\overline{PP_1}$ does not intersect $\mathcal{M}$ apart from $P_1$. Moreover, since $W_Y^s \cap \mathcal{Q}_4$ is contained in $\mathcal{M}$, it follows that $\overline{PP_1} \cap W_Y^s = \{P_1\}$.

	There exists $n \in \mathbb{N}$ such that $P_1^{2n} \in \overline{TT^2}^{s,Y} \setminus \{T\}$. In that case, $\overline{PP_1}$, $L_{a,b}^{2}(\overline{PP_1}) = \overline{PP_1^2}$, $\ldots$, $L_{a,b}^{2n-2}(\overline{PP_1}) = \overline{PP_1^{2n-2}}$ are all straight line segments in $\mathcal{Q}_4$ since $L_{a,b}^2$ acts on them as an affine map. Therefore, $L_{a,b}^{2n}(\overline{PP_1}) = \overline{PP_1^{2n}}$ is also a straight line segment. 
	
	We claim that $\overline{PP_1^{2n}}$ intersects $[T,C]^{s,Y}$. Indeed, $P_1^{2n} \in \mathcal{Q}_3$ lies on the segment $\overline{T^1T}^{s,Y}$, which has a negative slope, and therefore, $P_{1,y}^{2n} > T_y$. Moreover, $P \in \mathcal{Q}_4$ and $P_{y} > T_{y}$ (Lemma \ref{lem:P_above_T}), so $\overline{PP_1^{2n}}$ intersects $Oy^{-}$ above $T$. In particular, the intersection point of $\overline{PP_1^{2n}}$ and $Oy^{-}$ is an interior point of $\mathcal{K}$, so $\overline{PP_1^{2n}} \cap \Int\mathcal{K} \neq \emptyset$. Since $P$ is an endpoint of $\overline{PP_1^{2n}}$ that lies in $\mathcal{Q}_4$ outside $\mathcal{K}$, that segment intersects $\partial\mathcal{K} \cap \mathcal{Q}_4$, which is exactly $[T,C]^{s,Y}$.
	
	 On the other hand, the only intersection of $\overline{PP_1^{2n}}$ and $W_Y^s$ is $P_1^{2n}$, which contradicts $\overline{PP_1^{2n}} \cap [T,C]^{s,Y} = \emptyset$. Therefore, $P \in \mathcal{K}$. Since $P'=L_{a,b}^{-1}(P)$ and $\mathcal{K}$ is $L_{a,b}^{-1}$-invariant due to Claim B, we have that $P' \in \mathcal{K}$ as well. This proves Claim C. \hfill $\blacksquare$
	
	Claim C implies that $P \in \mathcal{K}$. Since $P$ does not lie on either of $Ox$, $Oy$, $W_X^u$, $W_Y^s$, we have that $P \in \Int\mathcal{K}$. Therefore, there exists $\varepsilon > 0$ such that $B_{\varepsilon}(P) \subset \Int\mathcal{K}$. Since $P \in \ell$, we have that $W_X^u \cap B_{\varepsilon}(P) \neq \emptyset$, hence $W_X^u \cap \Int\mathcal{K} \neq \emptyset$. Let $A \in W_X^u \cap \Int\mathcal{K}$. The polygon $\mathcal{K}$ is $L_{a,b}^{-1}$-invariant due to Claim B, and hence so is $\Int\mathcal{K}$, thus $A^{-n} \in \Int\mathcal{K}$ for every $n \in \mathbb{N}$. On the other hand, $A \in W_X^u$, so there exists $n_1 \in \mathbb{N}$ such that $A^{-n_1} \in \overline{ZZ^{-1}}^u$. This yields a contradiction since $\overline{ZZ^{-1}}^u \subset \partial\mathcal{K}$, hence $\overline{ZZ^{-1}}^{u} \cap \Int\mathcal{K} = \emptyset$. Therefore, the original assumption that $C$ lies to the left of $Z$ is false, which completes the proof.
	\end{proof} 
	
	We have proved that $W_Y^s$ intersects $Ox^{+}$ at a point that lies to the right of $Z$. We denote that intersection point by $C$.
	
	\begin{rem} \label{rem:first_tangency_point}
	Point $C$ corresponds to the point of first tangency introduced in \cite{ishii1997towards1}. This point was used in \cite{baptista2009basin} for the construction of the basin of the strange attractor for parameters in the Misiurewicz set. We remark that for parameters considered in \cite{baptista2009basin} and \cite{ishii1997towards1}, $C$ is the intersection of $Ox^{+}$ and $\overline{TT^{-2}}^{s,Y}$. For parameters in our set $\mathfrak{R}$, however, $C$ is generally given by $C=Ox^{+} \cap \overline{T^{-2j+2}T^{-2j}}^{s,Y}$ for some $j \in \mathbb{N}_0$, so the proof of the existence of $C$ and its position relative to $Z$ required a detailed geometric analysis given in Lemmas \ref{lem:WsY_positive_x_01} and \ref{lem:WsY_positive_x_02}. 
	\end{rem}
	
	 We now use the position of $C$ to identify a segment of $W_Y^s$ whose entire backward orbit remains in $\mathcal{Q}_1$. The goal is to show that these iterates approach the stable half-line $W_X^{s+}$. This will allow us to understand how $W_Y^s$ separates the plane.
 
 There exists $m_0\in\mathbb{N}$ such that $C$ lies on $[T^{-m_0+2},T^{-m_0}]^{s,Y}=\overline{T^{-m_0+2}T^{-m_0}}^{s,Y}$ (moreover, $m_0$ is even) and $\overline{CT^{-m_0}}^{s,Y}$ is contained in the first quadrant $\mathcal{Q}_1$; see Figure \ref{figure:cor_cvg_Ws+X}. As a consequence, $C^{-1}$ lies on $Oy^{+}$, and $L_{a,b}^{-1}(\overline{CT^{-m_0}}^{s,Y})=\overline{C^{-1}T^{-m_0-1}}^{s,Y}$ is a line segment lying in $L_{a,b}^{-1}(\mathcal{Q}_1)$ above the line segment $\overline{ZZ^{-1}}^{u}$. Hence, $\overline{C^{-1}T^{-m_0-1}}^{s,Y}$ is again fully contained in $\mathcal{Q}_1$. Inductively, we see that the same holds for all further preimages of that line segment and since $L_{a,b}^{-1}$ acts as an affine map in the upper half-plane, $L_{a,b}^{-n}(\overline{CT^{-m_0}}^{s,Y})=\overline{C^{-n}T^{-m_0-n}}^{s,Y}$ is a straight line segment for every $n\in\mathbb{N}$. The remaining part of the argument is to understand the backward iterates of the arc $\overline{CT^{-m_0}}^{s,Y}$. These iterates form the upper part of the separator and they approach $W_X^{s+}$.
	
	Recall that $W_X^{s+}$ denotes the half of $W_X^s$ which is a half-line in $\mathcal{Q}_1$ starting at $X$ and going to infinity. Note that $W_X^{s+}$ does not intersect $W_Y^s$. Using the same techniques as in Lemma \ref{lem:lozi_recurr_coef}, one can prove the following result from which it will follow that $W_Y^s$ approaches $W_X^{s+}$ in $\mathcal{Q}_1$. 
	
	\begin{lem}\label{lem:slope_cvg_upper}
	\begin{enumerate}
	\item Let $\alpha$ be a line segment in the upper half-plane which lies on a straight line whose slope equals $s_1$. Then $L_{a,b}^{-1}(\alpha)$ lies on a straight line whose slope equals
		$$s_2=b\cdot\frac{1}{s_1}+a.$$
	\item The sequence $(s_n)_{n\in\mathbb{N}_0}$ given by the recurrence
	\begin{equation}\label{eg:recurr2}
	s_{n+1}=b\cdot\frac{1}{s_n}+a,\quad n\geqslant0; \quad s_0\neq\tfrac{1}{2}\left(a-\sqrt{a^2+4b}\right),
	\end{equation}
	converges and
	$$\lim_{n\rightarrow\infty}s_n=\tfrac{1}{2}\left(a+\sqrt{a^2+4b}\right).$$
	\end{enumerate}
	\end{lem}
	
	The limit of the sequence $(s_n)_{n\in\mathbb{N}_0}$ from the previous lemma is precisely the slope of $W_X^{s+}$. Hence, if all preimages of a straight line segment lie in $\mathcal{Q}_1$ and its slope does not correspond to the unstable direction at $X$, the slopes of its preimages will converge to the stable direction at $X$.

    For the following result, let $\theta_1$ denote the portion of $L_{a,b}^{-1}(Oy^{+})$ which is a half-line above the segment $\overline{ZZ^{-1}}^u$, starting at $Z^{-2}$. Furthermore, for every $n\geqslant2$, let $\theta_n=L_{a,b}^{-n+1}(\theta_1)$. Notice that all $\theta_n$ are half-lines in $\mathcal{Q}_1$; see Figure \ref{figure:cor_cvg_Ws+X}.
	
	\begin{lem} \label{lem:y_preim_cvg_WXs+}
	For every $\varepsilon>0$, there is $k'\in\mathbb{N}$ such that $\theta_k$ and $W_X^{s+}$ are $\varepsilon$-close with respect to the Hausdorff distance for all $k \geqslant k'$.
	\end{lem}
	
	\begin{proof}
	Since $Z^{-1}$ lies on $Oy^{+}$, $Z^{-n-1}$ lies on the intersection of $L_{a,b}^{-n}(Oy^{+})$ with $\mathcal{Q}_1$. Now the claim follows from the fact that $\displaystyle\lim_{n\rightarrow\infty}Z^{-n}=X$ and Lemma \ref{lem:slope_cvg_upper}(2). 
	\end{proof}
	
	\begin{cor} \label{cor:CR_0_preim_cvg_WXs+}
	For every $\varepsilon>0$, there exists $k\in\mathbb{N}$ and points $A_1,A_2\in W_X^{s+}$ such that $L_{a,b}^{-k}(\overline{CT^{-m_0}}^{s,Y})$ and $\overline{A_1A_2}^{s}$ are $\varepsilon$-close with respect to the Hausdorff distance.
	\end{cor}
	
	\begin{proof}
	Since $C^{-1}$ lies on $Oy^{+}$ above $Z^{-1}$, its further preimages lie on the preimages of that axis in $\mathcal{Q}_1$. Applying Lemmas \ref{lem:slope_cvg_upper}(2) and \ref{lem:y_preim_cvg_WXs+} completes the proof. 
	\end{proof}
	
	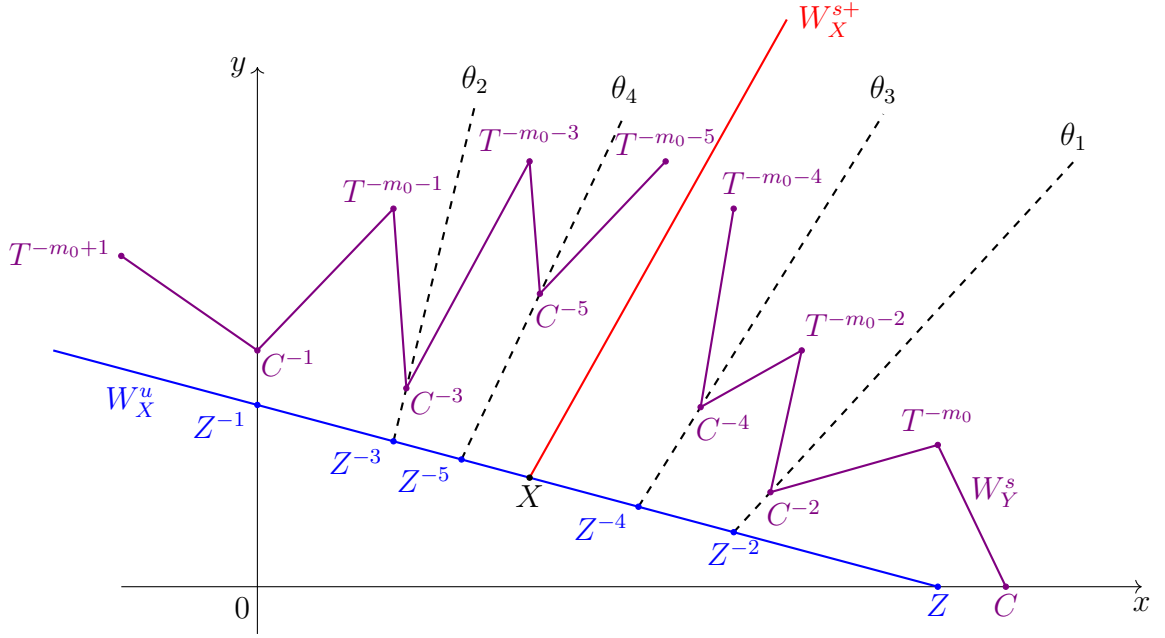
\begin{figure}[!ht]
	\begin{center}
	\begin{tikzpicture}[auto, xscale=1.8, yscale=1.25]
			\tikzstyle{nodec}=[draw,circle,fill=black,minimum size=2pt,
			inner sep=0pt, label distance=2mm]
			\tikzstyle{nodeh}=[draw,circle,fill=white,minimum size=4pt,
			inner sep=0pt]
			\tikzstyle{dot}=[circle,draw=none,fill=none,minimum size=0pt,inner sep=2pt, outer sep=-1pt]

			\draw[->] (-1,0)--(6.5,0) node [below]{$x$};
			\draw[->] (0,-0.5)--(0,5.5) node [left]{$y$};
			\node[label={[xshift=-0.2cm, yshift=-0.7cm]$0$}] at (0,0) {};
			
			\coordinate (e0) at (0,0);
			\coordinate (ex) at (1,0);
			\coordinate (ey) at (0,1);

			\coordinate (end) at (-1.5,2.5);			
			\coordinate (t0) at (5,0);

			\coordinate (t0m1) at (intersection of t0--end and e0--ey);

			\coordinate (t0m21) at (3.5,1);
			\coordinate (t0m22) at (3.5,2);
			\coordinate (t0m2) at (intersection of t0--end and t0m21--t0m22);

			\coordinate (t0m31) at (1,1);
			\coordinate (t0m32) at (1,2);
			\coordinate (t0m3) at (intersection of t0--end and t0m31--t0m32);

			\coordinate (t0m41) at (2.8,1);
			\coordinate (t0m42) at (2.8,2);
			\coordinate (t0m4) at (intersection of t0--end and t0m41--t0m42);

			\coordinate (t0m51) at (1.5,1);
			\coordinate (t0m52) at (1.5,2);
			\coordinate (t0m5) at (intersection of t0--end and t0m51--t0m52);

			\coordinate (x_1) at (2,1);
			\coordinate (x_2) at (2,2);
			\coordinate (x) at (intersection of t0--end and x_1--x_2);

			\coordinate (x_end) at (3.5,5);
			\coordinate (x1) at (1,6);
			\coordinate (x2) at (2,6);
			\coordinate (xend1) at (intersection of x--x_end and x1--x2);			

			\coordinate (m1end) at (6,4.5);
			\coordinate (m2end) at (1.6,5.1);
			\coordinate (m3end) at (4.6,5);
			\coordinate (m4end) at (2.7,5);		

			\coordinate (cm1) at (0,2.5);
			\coordinate (c) at (5.5,0);
			
			\coordinate (cm21) at (1,1);
			\coordinate (cm22) at (2,1);
			\coordinate (cm2) at (intersection of t0m2--m1end and cm21--cm22);

			\coordinate (cm31) at (1,2.1);
			\coordinate (cm32) at (2,2.1);
			\coordinate (cm3) at (intersection of t0m3--m2end and cm31--cm32);

			\coordinate (cm41) at (1,1.9);
			\coordinate (cm42) at (2,1.9);
			\coordinate (cm4) at (intersection of t0m4--m3end and cm41--cm42);
	
			\coordinate (cm51) at (1,3.1);	
			\coordinate (cm52) at (2,3.1);
			\coordinate (cm5) at (intersection of t0m5--m4end and cm51--cm52);			

			\coordinate (r) at (5,1.5);			
			\coordinate (rp1) at (-1,3.5);
			\coordinate (rm1) at (1,4);
			\coordinate (rm2) at (4,2.5);
			\coordinate (rm3) at (2,4.5);
			\coordinate (rm4) at (3.5,4);
			\coordinate (rm5) at (3,4.5);
			
			\draw[blue, thick] (t0)--(t0m1);
			\draw[blue, thick] (t0m1) to node[dot]{$W_X^u$} (end);
			\draw[red, thick] (x)--(xend1) node [right]{$W_X^{s+}$};
			\draw[thick, dashed] (t0m2)--(m1end) node[above]{$\theta_1$};
			\draw[thick, dashed] (t0m3)--(m2end) node[above]{$\theta_2$};
			\draw[thick, dashed] (t0m4)--(m3end) node[above]{$\theta_3$};
			\draw[thick, dashed] (t0m5)--(m4end) node[above]{$\theta_4$};
			
			\draw[violet, thick] (rp1)--(cm1)--(rm1)--(cm3)--(rm3)--(cm5)--(rm5);
			\draw[violet, thick] (c) to node[dot, swap]{$W_Y^s$} (r);
			\draw[violet, thick] (r)--(cm2)--(rm2)--(cm4)--(rm4);						
			\node[nodec, color=black, label={[below, color=black]$X$}] at (x) {};

			\node[nodec, color=blue, label={[below, color=blue]$Z$}] at (t0) {};	
			\node[nodec, color=blue, label={[below left, color=blue]$Z^{-1}$}] at (t0m1) {};
			\node[nodec, color=blue, label={[below, color=blue]$Z^{-2}$}] at (t0m2) {};
			\node[nodec, color=blue, label={[below left, color=blue]$Z^{-3}$}] at (t0m3) {};
			\node[nodec, color=blue, label={[below left, color=blue]$Z^{-4}$}] at (t0m4) {};
			\node[nodec, color=blue, label={[below left, color=blue]$Z^{-5}$}] at (t0m5) {};
			
			\node[nodec, color=violet, label={[below right, xshift=-1mm, yshift=1mm, color=violet]$C^{-1}$}] at (cm1) {};
			\node[nodec, color=violet, label={[below, color=violet]$C$}] at (c) {};
			\node[nodec, color=violet, label={[below right, xshift=-2mm, color=violet]$C^{-2}$}] at (intersection of t0m2--m1end and cm21--cm22) {};
			\node[nodec, color=violet, label={[below right, xshift=-1mm, yshift=1mm, color=violet]$C^{-3}$}] at (intersection of t0m3--m2end and cm31--cm32) {};
			\node[nodec, color=violet, label={[below, xshift=3mm, color=violet]$C^{-4}$}] at (intersection of t0m4--m3end and cm41--cm42) {};
			\node[nodec, color=violet, label={[below, xshift=3mm, color=violet]$C^{-5}$}] at (intersection of t0m5--m4end and cm51--cm52) {};
			
			\node[nodec, color=violet, label={[left, color=violet]$T^{-m_0+1}$}] at (rp1) {};
			\node[nodec, color=violet, label={[above, color=violet]$T^{-m_0}$}] at (r) {};
			\node[nodec, color=violet, label={[above, yshift=-1mm, color=violet]$T^{-m_0-1}$}] at (rm1) {};
			\node[nodec, color=violet, label={[above right, xshift=-1mm, color=violet]$T^{-m_0-2}$}] at (rm2) {};
			\node[nodec, color=violet, label={[above, yshift=-0.5mm, color=violet]$T^{-m_0-3}$}] at (rm3) {};
			\node[nodec, color=violet, label={[above right, xshift=-3mm, yshift=0mm, color=violet]$T^{-m_0-4}$}] at (rm4) {};
			\node[nodec, color=violet, label={[above, yshift=-1mm, color=violet]$T^{-m_0-5}$}] at (rm5) {};
			
			\end{tikzpicture}
	\end{center}
	\caption{For every $n\in\mathbb{N}$, points $Z^{-n-1}$ and $C^{-n-1}$ lie on the $n$-th preimage of $Oy^{+}$ (dashed lines) which converge to $W_X^{s+}$ above the line segment $\overline{ZZ^{-1}}^{u}$ (blue) due to Lemma \ref{lem:y_preim_cvg_WXs+}. As a consequence, preimages of $\overline{CT^{-m_0}}^{s,Y}$ approach $W_X^{s+}$ as well (Corollary \ref{cor:CR_0_preim_cvg_WXs+}).}
	\label{figure:cor_cvg_Ws+X}
	\end{figure}
	
	Therefore, backward iterates of segments of $W_Y^{s}$ in $\mathcal{Q}_1$ above $\overline{ZZ^{-1}}^{u}$ become arbitrarily close in direction and position to $W_X^{s+}$. Moreover, in $\mathcal{Q}_1$, $L_{a,b}^{-1}$ acts affinely and expands vectors in the stable direction by factor $\frac{1}{|\lambda_X^s|} > 1$. Hence, the lengths of these iterated segments tend to infinity, and the corresponding points escape to infinity in $\mathcal{Q}_1$.
	 
	Therefore, $W_X^{s+}\cup\overline{XV^1}^{s}$ separates $\mathcal{Q}_1$ into two connected components, and the two branches of $W_Y^s$ approaching $W_X^{s+}$ form a cross-cut of each component. Since the intersection of $W_Y^s$ with the union of $\mathcal{Q}_2$, $\mathcal{Q}_3$ and $\mathcal{Q}_4$ is the polygonal line $[C,C^{-1}]^{s,Y}$, and that polygonal line separates each one of these quadrants, we obtain the following result.
	
	\begin{cor} \label{cor:WsY_separates}
	$W_Y^s$ separates the plane.
	\end{cor}

	\subsection{Approaching infinity} \label{subsec:approaching_infinity}
	
	Corollary \ref{cor:WsY_separates} implies that the complement $\mathbb{R}^2\setminus W_Y^s$ consists of two connected components. We denote by $\mathcal{A}_1$ the one containing the fixed point $X$ and by $\mathcal{A}_2$ the other one. Note that $\mathcal{A}_1$ also contains $W_X^s$ since stable manifolds of distinct saddle points do not intersect. In addition, Corollary \ref{cor:WuXWsY_empty} implies that $\mathcal{A}_1$ contains $W_X^u$ and hence also the accumulation set $\ell$. In particular, the periodic orbit $\{P,P'\}$ is contained in $\mathcal{A}_1$.
	
\begin{lem}\label{lem:A1_invariant}
	The sets $\mathcal{A}_1$ and $\mathcal{A}_2$ are $L_{a,b}$-invariant and $L_{a,b}^{-1}$-invariant.
	\end{lem}
	
	\begin{proof}
	We first prove that $\mathcal{A}_1$ is $L_{a,b}$-invariant. We do so by partitioning $\mathcal{A}_1$ into five regions $\mathcal{S}_1,\mathcal{S}_2,\ldots,\mathcal{S}_5$ such that $L_{a,b}$ acts as an affine map on every one of them (see Figure \ref{figure:lem_A1_invariant}).
	
	\begin{figure}[!ht]
	\begin{center}
	\begin{subfigure}{.4\textwidth}
	\begin{flushleft}
	\begin{tikzpicture}[auto, scale=.6]
			\tikzstyle{nodec}=[draw,circle,fill=black,minimum size=2pt,
			inner sep=0pt, label distance=2mm]
			\tikzstyle{nodeh}=[draw,circle,fill=white,minimum size=4pt,
			inner sep=0pt]
			\tikzstyle{dot}=[circle,draw=none,fill=none,minimum size=0pt,inner sep=2pt, outer sep=-1pt]

			\draw[->] (-4.5,0)--(7,0) node [below]{$x$};
			\draw[->] (0,-6.5)--(0,8.5) node [left]{$y$};
			\node[label={[xshift=-0.2cm, yshift=-0.7cm]$0$}] at (0,0) {};
			
			\coordinate (e0) at (0,0);
			\coordinate (ex) at (1,0);
			\coordinate (ey) at (0,1);
			
			\coordinate (yend1) at (7,5.26);
			\coordinate (yend2) at (-7,5.26);
			\coordinate (lm1O) at (0,-4);
			
			\coordinate (yend11) at ($(lm1O)!0.8!(yend1)$);
			\coordinate (yend21) at ($(lm1O)!0.8!(yend2)$);
			\draw[thick, dashed] (yend21)--(lm1O)--(yend11) node[above, xshift=3mm]{$y=a|x|-1$};
			
			\node[nodec, color=black, label={[right, color=black]$O^{-1}$}] at (lm1O) {};
			
			\coordinate (t0) at (5,0);
			\coordinate (t0m1) at (0,2.16);
			\coordinate (t0m2) at (intersection of t0--t0m1 and lm1O--yend1);
			\coordinate (x1) at (2,1);
			\coordinate (x2) at (2,2);
			\coordinate (x) at (intersection of t0--t0m1 and x1--x2);
			\coordinate (xend) at (4.82,7.64);
			
			\draw[blue, thick] (t0)--(x);
			\draw[blue, thick] (x)--(t0m1);
			
			\draw[red, thick] (x)--(xend) node[above]{$W_X^{s+}$};

			\node[nodec, color=blue, label={[above, color=blue]$Z$}] at (t0) {};
			\node[nodec, color=blue, label={[above left, xshift=4mm, yshift=1mm, color=blue]$Z^{-2}$}] at (t0m2) {};
			\node[nodec, color=blue, label={[right, yshift=1mm, color=blue]$Z^{-1}$}] at (t0m1) {};
			\node[nodec, color=black, label={[above, color=black, xshift=-0.5mm]$X$}] at (x) {};
			
			\coordinate (r0) at (0,-6);
			\coordinate (r0m11) at (-4.28,1);
			\coordinate (r0m12) at (-4.28,2);
			\coordinate (r0m1) at (intersection of r0m11--r0m12 and lm1O--yend2);
			
			\coordinate (f1) at (3.98,-4.44);
			\coordinate (f2) at (6.12,-1.66);
			\coordinate (f3) at (5.58,1.76);
			\coordinate (f41) at (4,1);
			\coordinate (f42) at (4,2);
			\coordinate (f4) at (intersection of f41--f42 and lm1O--yend1);
			\coordinate (f5) at (4.54,3.42);
			\coordinate (f6) at (3.66,2.4);
			\coordinate (f7) at (5,6);
			\coordinate (f8) at (3.8,4.02);
			\coordinate (f9) at (5.28,7.6);
			
			\coordinate (g1) at (-4,3);
			\coordinate (g2) at (-0.66,3.82);
			\coordinate (g3) at (0,2.76);
			\coordinate (g4) at (1.7,5.3);
			\coordinate (g5) at (1.44,3.08);
			\coordinate (g6) at (3,6);
			\coordinate (g7) at (2.82,4.12);
			\coordinate (g8) at (4.4,7.7);
			\coordinate (g12) at (-3,4);

			\coordinate (c) at (intersection of e0--ex and f2--f3);
			
			\draw [violet, thick] (r0) to node[dot]{$W_Y^s$} (r0m1);
			\draw [violet, thick] (r0)--(f1);
			\draw [violet, thick] (r0m1)--(g1);
			\draw [violet, thick, dashed] (f1)--(f2);
			\draw [violet, thick, dashed] (g1)--(g12);
			\draw [violet, thick] (f2)--(f3)--(f4)--(f5)--(f6)--(f7)--(f8)--(f9);
			\draw [violet, thick] (g12)--(g2)--(g3)--(g4)--(g5)--(g6)--(g7)--(g8);
			
			\node[nodec, color=violet, label={[below left, color=violet]$T$}] at (r0) {};
			\node[nodec, color=violet, label={[left, yshift=-1mm, color=violet]$T^{-1}$}] at (r0m1) {};
			\node[nodec, color=violet, label={[below right, color=violet]$C$}] at (c) {};
			\node[nodec, color=violet, label={[above, yshift=3mm, xshift=0.5mm, color=violet]$C^{-1}$}] at (g3) {};
			
			\node[nodec, color=violet] at (f1) {};
			\node[nodec, color=violet] at (f2) {};
			\node[nodec, color=violet] at (f3) {};
			\node[nodec, color=violet] at (f4) {};
			\node[nodec, color=violet] at (f5) {};
			\node[nodec, color=violet] at (f6) {};
			\node[nodec, color=violet] at (f7) {};
			\node[nodec, color=violet] at (f8) {};
			
			\node[nodec, color=violet] at (g1) {};
			\node[nodec, color=violet] at (g12) {};
			\node[nodec, color=violet] at (g2) {};
			\node[nodec, color=violet] at (g4) {};
			\node[nodec, color=violet] at (g5) {};
			\node[nodec, color=violet] at (g6) {};
			\node[nodec, color=violet] at (g7) {};
			
			\node[dot] (reg1) at (2.2,3.22) {$\mathcal{S}_1$};
			\node[fill=green!10!white] (reg2) at (1,-1) {$\mathcal{S}_2$};
			\node[fill=violet!10!white] (reg3) at (3,-2.8) {$\mathcal{S}_3$};
			\node[fill=orange!10!white] (reg4) at (-2,1) {$\mathcal{S}_4$};
			\node[dot] (reg5) at (-0.62,-3.98) {$\mathcal{S}_5$};
			
			\begin{pgfonlayer}{bg}
				\fill[teal!10!white] (g8.center)--(g7.center)--(g6.center)--(g5.center)--(g4.center)--(g3.center)--(t0m1.center)--(t0.center)--(c.center)--(f3.center)--(f4.center)--(f5.center)--(f6.center)--(f7.center)--(f8.center)--(f9.center)--(xend.center)--cycle;
				\fill[green!10!white] (t0m1.center)--(lm1O.center)--(t0m2.center)--cycle;
				\fill[violet!10!white] (lm1O.center)--(r0.center)--(f1.center)--(f2.center)--(c.center)--(t0.center)--(t0m2.center)--cycle;
				\fill[orange!10!white] (t0m1.center)--(g3.center)--(g2.center)--(g12.center)--(g1.center)--(r0m1.center)--(lm1O.center)--cycle;
				\fill[cyan!10!white] (r0m1.center)--(r0.center)--(lm1O.center)--cycle;
				
				\fill[pattern=dots] (t0m1.center)--(lm1O.center)--(t0m2.center)--cycle;
				\fill[pattern=vertical lines] (lm1O.center)--(r0.center)--(f1.center)--(f2.center)--(c.center)--(t0.center)--(t0m2.center)--cycle;
				\fill[pattern=horizontal lines] (t0m1.center)--(g3.center)--(g2.center)--(g12.center)--(g1.center)--(r0m1.center)--(lm1O.center)--cycle;
			\end{pgfonlayer}	
			
			\end{tikzpicture}
	\end{flushleft}
	\end{subfigure}\hspace{4em}%
	\begin{subfigure}{.4\textwidth}
	\begin{tikzpicture}[auto, scale=.6]
			\tikzstyle{nodec}=[draw,circle,fill=black,minimum size=2pt,
			inner sep=0pt, label distance=2mm]
			\tikzstyle{nodeh}=[draw,circle,fill=white,minimum size=4pt,
			inner sep=0pt]
			\tikzstyle{dot}=[circle,draw=none,fill=none,minimum size=0pt,inner sep=2pt, outer sep=-1pt]

			\draw[->] (-4.5,0)--(7,0) node [below]{$x$};
			\draw[->] (0,-6.5)--(0,8.5) node [left]{$y$};
			\node[label={[xshift=-0.2cm, yshift=-0.7cm]$0$}] at (0,0) {};
			
			\coordinate (e0) at (0,0);
			\coordinate (ex) at (1,0);
			\coordinate (ey) at (0,1);

			\coordinate (xim) at (1.54,0);

			\coordinate (t0) at (5,0);
			\coordinate (t0m1) at (0,2.16);
			\coordinate (t0m2) at (intersection of t0--t0m1 and lm1O--yend1);
			\coordinate (x1) at (2,1);
			\coordinate (x2) at (2,2);
			\coordinate (x) at (intersection of t0--t0m1 and x1--x2);
			\coordinate (xend) at (4.82,7.64);
			
			\draw[blue, thick] (t0)--(x);
			\draw[blue, thick] (x)--(t0m1);
			
			\draw[red, thick] (x)--(xend) node[above]{$W_X^{s+}$};

			\node[nodec, color=blue, label={[above, color=blue]$Z$}] at (t0) {};
			\node[nodec, color=blue, label={[above left, xshift=4mm, yshift=1mm, color=blue]$Z^{-2}$}] at (t0m2) {};
			\node[nodec, color=blue, label={[right, yshift=1mm, color=blue]$Z^{-1}$}] at (t0m1) {};
			\node[nodec, color=black, label={[above, color=black, xshift=-0.5mm]$X$}] at (x) {};
			
			\coordinate (r0) at (0,-6);
			\coordinate (r0m11) at (-4.28,1);
			\coordinate (r0m12) at (-4.28,2);
			\coordinate (r0m1) at (intersection of r0m11--r0m12 and lm1O--yend2);
			\coordinate (r0p1) at (intersection of r0--r0m1 and e0--ex);
			
			\coordinate (f1) at (3.98,-4.44);
			\coordinate (f2) at (6.12,-1.66);
			\coordinate (f3) at (5.58,1.76);
			\coordinate (f41) at (4,1);
			\coordinate (f42) at (4,2);
			\coordinate (f4) at (intersection of f41--f42 and lm1O--yend1);
			\coordinate (f5) at (4.54,3.42);
			\coordinate (f6) at (3.66,2.4);
			\coordinate (f7) at (5,6);
			\coordinate (f8) at (3.8,4.02);
			\coordinate (f9) at (5.28,7.6);
			
			\coordinate (g1) at (-4,3);
			\coordinate (g2) at (-0.66,3.82);
			\coordinate (g3) at (0,2.76);
			\coordinate (g4) at (1.7,5.3);
			\coordinate (g5) at (1.44,3.08);
			\coordinate (g6) at (3,6);
			\coordinate (g7) at (2.82,4.12);
			\coordinate (g8) at (4.4,7.7);
			\coordinate (g12) at (-3,4);
			
			\coordinate (c) at (intersection of e0--ex and f2--f3);
			\coordinate (c1) at ($(g12)!0.25!(g2)$);
			
			\coordinate (t0p1) at (intersection of t0--t0m1 and xim--c1);
			
			\node[nodec, color=blue, label={[above, color=blue]$Z^{1}$}] at (t0p1) {};
			
			\draw [blue, thick] (t0m1)--(t0p1);
			\draw [black, thick, dashed] (t0p1)--(c1); 
			
			\draw [violet, thick] (r0) to node[dot]{$W_Y^s$} (r0m1);
			\draw [violet, thick] (r0)--(f1);
			\draw [violet, thick] (r0m1)--(g1);
			\draw [violet, thick, dashed] (f1)--(f2);
			\draw [violet, thick, dashed] (g1)--(g12);
			\draw [violet, thick] (f2)--(f3)--(f4)--(f5)--(f6)--(f7)--(f8)--(f9);
			\draw [violet, thick] (g12)--(g2)--(g3)--(g4)--(g5)--(g6)--(g7)--(g8);
			
			\node[nodec, color=violet, label={[below left, color=violet]$T$}] at (r0) {};
			\node[nodec, color=violet, label={[left, yshift=-1mm, color=violet]$T^{-1}$}] at (r0m1) {};
			\node[nodec, color=violet, label={[below right, color=violet]$C$}] at (c) {};
			\node[nodec, color=violet, label={[above, color=violet]$C^1$}] at (c1) {};
			\node[nodec, color=violet, label={[below left, color=violet]$T^1$}] at (r0p1) {};
			
			\node[nodec, color=violet] at (f1) {};
			\node[nodec, color=violet] at (f2) {};
			\node[nodec, color=violet] at (f3) {};
			\node[nodec, color=violet] at (f4) {};
			\node[nodec, color=violet] at (f5) {};
			\node[nodec, color=violet] at (f6) {};
			\node[nodec, color=violet] at (f7) {};
			\node[nodec, color=violet] at (f8) {};
			
			\node[nodec, color=violet] at (g1) {};
			\node[nodec, color=violet] at (g12) {};
			\node[nodec, color=violet] at (g2) {};
			\node[nodec, color=violet] at (g3) {};
			\node[nodec, color=violet] at (g4) {};
			\node[nodec, color=violet] at (g5) {};
			\node[nodec, color=violet] at (g6) {};
			\node[nodec, color=violet] at (g7) {};

			\node[dot] (reg1) at (2.2,3.22) {$\mathcal{S}'_1$};
			\node[fill=green!10!white] (reg2) at (0.9,0.8) {$\mathcal{S}'_2$};
			\node[fill=violet!10!white] (reg3) at (-2,1.8) {$\mathcal{S}'_3$};
			\node[fill=orange!10!white] (reg4) at (3,-2) {$\mathcal{S}'_4$};
			\node[dot] (reg5) at (-1.5,-1.5) {$\mathcal{S}'_5$};
			
			\begin{pgfonlayer}{bg}
				\fill[teal!10!white] (g8.center)--(g7.center)--(g6.center)--(g5.center)--(g4.center)--(g3.center)--(g2.center)--(c1.center)--(t0p1.center)--(t0.center)--(c.center)--(f3.center)--(f4.center)--(f5.center)--(f6.center)--(f7.center)--(f8.center)--(f9.center)--(xend.center)--cycle;
				\fill[green!10!white] (t0.center)--(e0.center)--(t0m1.center)--cycle;
				\fill[violet!10!white] (e0.center)--(t0m1.center)--(t0p1.center)--(c1.center)--(g12.center)--(g1.center)--(r0m1.center)--(r0p1.center)--(e0.center)--cycle;
				\fill[orange!10!white] (r0.center)--(e0.center)--(c.center)--(f2.center)--(f1.center)--cycle;
				\fill[cyan!10!white] (r0.center)--(r0p1.center)--(e0.center)--cycle;
				
				\fill[pattern=dots] (t0.center)--(e0.center)--(t0m1.center)--cycle;
				\fill[pattern=vertical lines] (e0.center)--(t0m1.center)--(t0p1.center)--(c1.center)--(g12.center)--(g1.center)--(r0m1.center)--(r0p1.center)--(e0.center)--cycle;
				\fill[pattern=horizontal lines] (r0.center)--(e0.center)--(c.center)--(f2.center)--(f1.center)--cycle;
			\end{pgfonlayer}	
			\end{tikzpicture}
	\end{subfigure}
	\end{center}
	\caption{Regions $\mathcal{S}_i$ and their images $\mathcal{S}'_i=L_{a,b}(\mathcal{S}_i)$ from Lemma \ref{lem:A1_invariant}.}
	\label{figure:lem_A1_invariant}
	\end{figure}
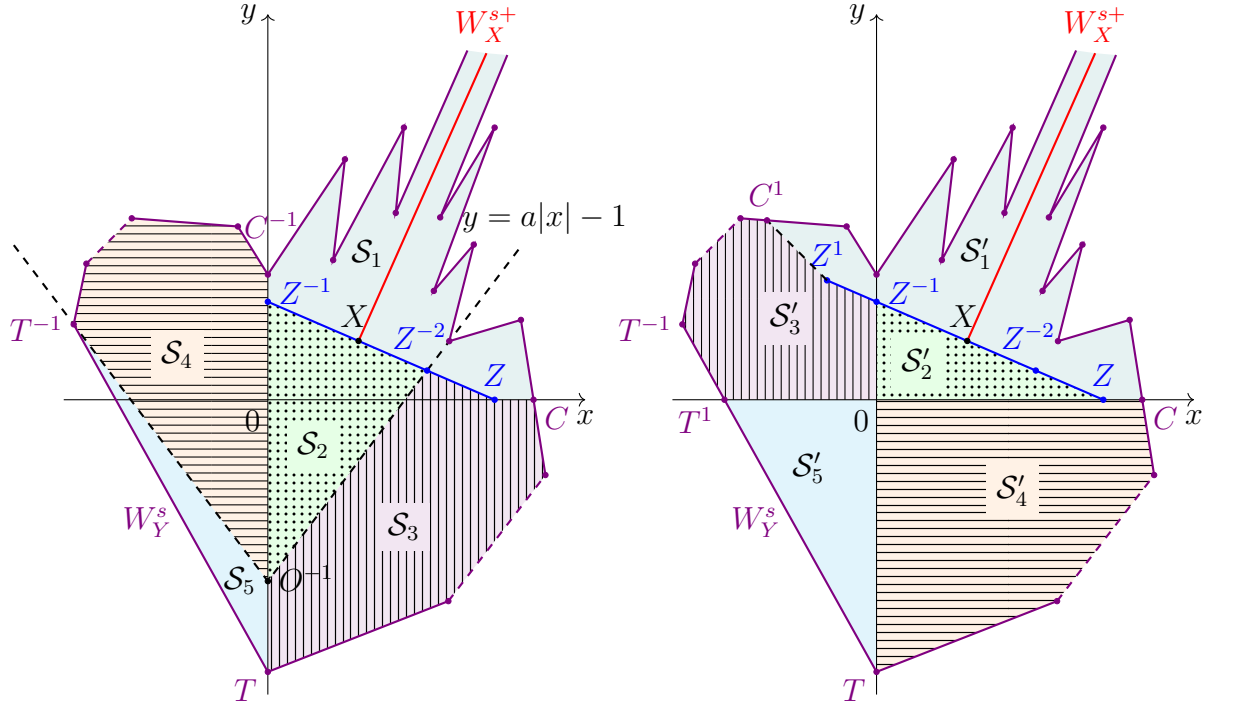

	\begin{enumerate}
	\item Let $\mathcal{S}_1$ be the intersection of $\mathcal{A}_1$ with the portion of $\mathcal{Q}_1$ above the line $\overline{ZZ^{-1}}^{u}$. Then $\mathcal{S}_1\setminus W_X^{s+}$ can be written as a union of polygons $\Delta_n$, $n\in\mathbb{N}_0$, such that
	\begin{equation} \label{eq:polygons_delta_n}
	\partial\Delta_n=\overline{Z^{-n}Z^{-n-2}}^{u}\cup\overline{Z^{-n-2}C^{-n-2}}\cup\overline{Z^{-n}C^{-n}}\cup[C^{-n},C^{-n-2}]^{s,Y}.
	\end{equation}
	Note that the boundary of each $\Delta_n$ consists of portions of $W_X^u$, $W_Y^s$, and the preimages of $Oy^{+}$. In particular, $\partial\Delta_0$ also contains a portion of $Ox^{+}$, that is, $\overline{ZC}$, while $\partial\Delta_1$ contains a portion of $Oy^{+}$, namely $\overline{Z^{-1}C^{-1}}$; see Figure \ref{figure:cor_cvg_Ws+X}.
	
	Since, by construction, $L_{a,b}$ maps $\partial\Delta_{n+1}$ onto $\partial\Delta_n$ and acts affinely on each $\Delta_n$, we have that $L_{a,b}(\Delta_{n+1})=\Delta_n$ for every $n\in\mathbb{N}_0$. In particular, $L_{a,b}(\Delta_0)$ is a polygon with boundary $\overline{C^{-1}Z^{-1}}\cup\overline{Z^{-1}Z^1}^{u}\cup\overline{Z^1C^1}$ which is contained in $\mathcal{A}_1$ ($W_Y^s \cap \overline{Z^1C^1} = \{C^1\}$). This fact, together with $L_{a,b}(W_X^{s+})=W_X^{s+}$, proves that $L_{a,b}(\mathcal{S}_1)\subset\mathcal{A}_1$.
	
	\item Let $\mathcal{S}_2$ be the triangle $O^{-1}Z^{-1}Z^{-2}$ in $\mathcal{Q}_1$ and $\mathcal{Q}_4$, where $O$ is the origin. Then $\mathcal{S}_2$ maps to the triangle $OZ^{-1}Z$, which is contained in $\mathcal{A}_1$ due to Lemma \ref{lem:WsY_positive_x_02}.
	
	\item Let $\mathcal{S}_3$ be the polygon in $\mathcal{Q}_1$ and $\mathcal{Q}_4$ with boundary
	\begin{equation*}
	\partial\mathcal{S}_3=\overline{TO^{-1}}\cup\overline{O^{-1}Z^{-2}}\cup\overline{Z^{-2}Z}^{u}\cup\overline{ZC}\cup[C,T]^{s,Y}.
	\end{equation*}
	Then $\mathcal{S}_3$ is mapped to the polygon in $\mathcal{Q}_2$ with boundary
	\begin{equation*}
	\partial L_{a,b}(\mathcal{S}_3)=\overline{T^1O}\cup\overline{OZ^{-1}}\cup\overline{Z^{-1}Z^1}^{u}\cup\overline{Z^1C^1}\cup[C^1,T^1]^{s,Y}.
	\end{equation*}
	Since $W_Y^s\cap\overline{Z^1C^1}=\{C^1\}$, it follows that $L_{a,b}(\mathcal{S}_3)$ is contained in $\mathcal{A}_1$.
	
	\item We denote by $\mathcal{S}_4$ the polygon in $\mathcal{Q}_2$ and $\mathcal{Q}_3$ with boundary
	\begin{equation*}
	\partial\mathcal{S}_4=\overline{C^{-1}O^{-1}}\cup\overline{O^{-1}T^{-1}}\cup[T^{-1},C^{-1}]^{s,Y}.
	\end{equation*}
	In this case, $\mathcal{S}_4$ maps to the polygon with boundary $\overline{TO}\cup\overline{OC}\cup[C,T]^{s,Y}$, which is the intersection of $\mathcal{A}_1$ with $\mathcal{Q}_4$.
	
	\item Finally, $\mathcal{S}_5$ denotes the triangle in $\mathcal{Q}_2$ and $\mathcal{Q}_3$ with vertices $O^{-1}$, $T$, and $T^{-1}$. Its image $L_{a,b}(\mathcal{S}_5)$ is the triangle with vertices $O$, $T^1$ and $T$, that is, the intersection of $\mathcal{A}_1$ with $\mathcal{Q}_3$.  
	\end{enumerate}
	
	For every $i=1,2,\ldots,5$, let $\mathcal{S}'_i=L_{a,b}(\mathcal{S}_i)$. The boundaries of the regions $\mathcal{S}_i$ match exactly along common edges, and their union covers $\mathcal{A}_1$, thus the sets $\mathcal{S}'_i$ also form a partition of $\mathcal{A}_1$; see Figure \ref{figure:lem_A1_invariant}. Therefore, $L_{a,b}(\mathcal{A}_1)=\mathcal{A}_1$. Since $L_{a,b}$ is bijective, we also have $L_{a,b}^{-1}(\mathcal{A}_1)=L_{a,b}^{-1}(L_{a,b}(\mathcal{A}_1))=\mathcal{A}_1$, that is, $\mathcal{A}_1$ is also $L_{a,b}^{-1}$-invariant. Since $W_Y^s$ is $L_{a,b}$-invariant and $L_{a,b}^{-1}$-invariant as well, the same holds for the complement of the union of these two sets, $\mathbb{R}^2\setminus(\mathcal{A}_1\cup W_Y^s)=\mathcal{A}_2$, which completes the proof.
	\end{proof}
	
	Observe the regions $\mathcal{S}_1$, $\mathcal{S}_2$, $\mathcal{S}_3$, $\mathcal{S}_4$, $\mathcal{S}_5$ introduced in the proof of Lemma \ref{lem:A1_invariant} (and reproduced in Figure \ref{figure:lem_A1_invariant}). Let $\mathcal{U}$ denote the union
	\begin{equation} \label{eq:region_U}
	\mathcal{U} = \mathcal{S}_2 \cup \mathcal{S}_3 \cup \mathcal{S}_4 \cup \mathcal{S}_5.
	\end{equation}
	
	\begin{lem} \label{lem:U_invariant}
	\begin{enumerate}
		\item The region $\mathcal{U}$ given by \eqref{eq:region_U} is $L_{a,b}$-invariant. 
		\item Every point in $\mathcal{S}_1 \setminus W_X^{s+}$ is mapped to $\mathcal{U}$ under finitely many iterations of $L_{a,b}$.
	\end{enumerate}
	\end{lem}
	
	\begin{proof}
	\begin{enumerate}
		\item In the proof of Lemma \ref{lem:A1_invariant}, we have shown that $L_{a,b}(\mathcal{S}_2) \subseteq \mathcal{S}_2 \cup \mathcal{S}_3$, $L_{a,b}(\mathcal{S}_3) \subseteq \mathcal{S}_4 \cup \mathcal{S}_5$, $L_{a,b}(\mathcal{S}_4) \subseteq \mathcal{S}_2 \cup \mathcal{S}_3$, $L_{a,b}(\mathcal{S}_5) \subseteq \mathcal{S}_4 \cup \mathcal{S}_5$. Therefore, for every $i \in \{2,3,4,5\}$, the region $\mathcal{S}_i$ is mapped to $\mathcal{U}$, so the claim follows.
		
		\item For $n \in \mathbb{N}_0$, let $\Delta_n$ be the polygon whose boundary is given by \eqref{eq:polygons_delta_n}. From the proof of Lemma \ref{lem:A1_invariant}, we know that $\mathcal{S}_1 \setminus W_X^{s+} = \bigcup_{n \in \mathbb{N}_0}\Delta_n$, and that $L_{a,b}(\Delta_{n+1}) = \Delta_n$ for every $n \in \mathbb{N}_0$. In particular, $\Delta_0$ is mapped to $\mathcal{S}_4 \subset \mathcal{U}$. Therefore, for every $n \in \mathbb{N}_0$, we have that $L_{a,b}^{n}(\Delta_n) = \Delta_0$, and hence $L_{a,b}^{n+1}(\Delta_n) \subset \mathcal{U}$, which completes the proof.
	\end{enumerate}
	\end{proof}
	
	Therefore, $\mathcal{U}$ is an $L_{a,b}$-invariant region, and the orbit of every point in $\mathcal{A}_1\setminus W_X^s$ eventually lands in $\mathcal{U}$. The following result shows that all the "interesting dynamics" for $L_{a,b}$ is contained in $\mathcal{A}_1$.
	
	\begin{prop}\label{prop:A2_infinity}
	All points in $\mathcal{A}_2$ tend to infinity under forward iterations of $L_{a,b}$.
	\end{prop}
	
	\begin{proof}
	It suffices to show that every orbit in $\mathcal{A}_2$ intersects every bounded set at most finitely many times. Assume for contradiction that there exists a bounded set $\mathcal{B}$ and $A \in \mathcal{A}_2$ such that infinitely many elements of the sequence $(A^n)_{n \in \mathbb{N}}$ lie in $\mathcal{B}$. Then these forward iterates lie in $\Cl\mathcal{B}$, which is compact, so there exists a subsequence $(A^{n_k})_{k \in \mathbb{N}}$ and $G \in \Cl\mathcal{B}$ such that $A^{n_k} \rightarrow G$ as $k \rightarrow \infty$. We have that $G \in \omega(A,L_{a,b})$.
	
	If $(n_k)_{k \in \mathbb{N}}$ contains infinitely many even terms, then $G \in \omega(A,L_{a,b}^2)$. Otherwise, $(n_k + 1)_{k}$ contains infinitely many even numbers and $A^{n_k + 1} \rightarrow G^1$, so $G^1 \in \omega(A,L_{a,b}^2)$. Therefore, by possibly replacing $G$ with its forward image, we have $G \in \omega(A,L_{a,b}^2)$.
	
	Since $\omega(A,L_{a,b}^2) \subseteq \Omega(L_{a,b}^2) \subseteq \ell \cup \{X,Y\}$ (where the last inclusion holds by Theorem \ref{thm:L2_non_wandering}), it follows that $G \in \ell \cup \{X,Y\}$. On the other hand, $\mathcal{A}_2$ is $L_{a,b}$-invariant due to Lemma \ref{lem:A1_invariant}, so the forward orbit of $A$ is contained in $\mathcal{A}_2$ and $G \in \Cl\mathcal{A}_2$. Since $\mathcal{A}_1$ is open and contains $X$ and $\ell$, neither $X$ nor any point of $\ell$ belongs to $\Cl\mathcal{A}_2$. Hence, $G = Y$ and $A^{n_k} \rightarrow Y$ as $k \rightarrow \infty$.
	
	In particular, for a sufficiently large $k$, $A^{n_k}$ lies in $\mathcal{Q}_3$ in the exterior of the triangle $OTT^1$, and all forward iterates of $A^{n_k}$ remain in the same region (all points in $\mathcal{Q}_3$ that get mapped to $\mathcal{Q}_4$ belong to $\mathcal{A}_1$). Therefore, $L_{a,b}$ acts on the forward orbit of $A^{n_k}$ as an affine map, so for every $i \in \mathbb{N}_0$,
	\begin{gather*}
	\dist(A^{n_k + i},TT^1) = |\lambda_Y^u|^{i} \dist(A^{n_k},TT^1),\\
	\dist(A^{n_k + i},YU^{-1}) = |\lambda_Y^s|^{i} \dist(A^{n_k},YU^{-1}).
	\end{gather*}
Therefore, the distance of $A^{n_k + i}$	to the line $TT^1$ grows exponentially by factor $|\lambda_Y^u| > 1$, and the distance to $YU^{-1}$ decreases by factor $|\lambda_Y^s|^{i}$, so the forward iterates of $A^{n_k}$ tend to infinity in $\mathcal{Q}_3$, which contradicts $A^{n_k} \rightarrow Y$. Therefore, our initial assumption was false, that is, every orbit in $\mathcal{A}_2$ intersects every bounded set at most finitely many times, which completes the proof.  
	\end{proof}

	\begin{rem} \label{rem:A2_infinity_other_proof}
	The preceding proof establishes escape to infinity using the non-wandering set. However, the escape mehanism can also be understood by analyzing the qualitative behavior of orbits in $\mathcal{A}_2$, as illustrated in Figure \ref{figure:thm_A2_infinity}. In order to provide geometric insight into this dynamical behavior, we sketch the corresponding argument below.
	
	Suppose that $A \in \mathcal{A}_2 \cap \mathcal{Q}_1$. Then $L_{a,b}$ acts on $A$ as an affine map and the distances of $A^j$ to $W_X^{s+}$ increase by the factor $|\lambda_X^u|$. Therefore, there exists $n_1 \in \mathbb{N}$ such that $A^{n_1} \in \mathcal{Q}_2$.
	
	Points from $\mathcal{Q}_2$ are mapped to $\mathcal{Q}_3$ or $\mathcal{Q}_4$, and points from $\mathcal{Q}_4\cap\mathcal{A}_2$ are all mapped to $\mathcal{Q}_2$. If the forward orbit of $A^{n_1}$ remained in $\mathcal{Q}_2 \cup \mathcal{Q}_4$, then successive applications of $L_{a,b}$ would alternate between these two quadrants. Consequently, $L_{a,b}^2$ would act on the orbit as an affine contraction whose attracting fixed points are $P \in \mathcal{Q}_4$ and $P' \in \mathcal{Q}_2$. Hence, the forward orbit of $A^{n_1}$ would converge to the periodic orbit $\{P,P'\}$. This is not possible because the orbit is contained in the invariant set $\mathcal{A}_2$, whereas $P,P' \in \mathcal{A}_1$. Therefore, there exists $n_2 \in \mathbb{N}$ such that $A^{n_2} \in \mathcal{Q}_3$.
	
	Finally, in $\mathcal{Q}_3$, the distances of $A^{n_2+j}$ to the line $TT^1$ grow exponentially by factor $|\lambda_Y^u|$ and the distances to $W_Y^{u+}$ decrease by $|\lambda_Y^s|$, so we see that forward iterates of $A$ in $\mathcal{Q}_3$ tend to infinity. 
	\end{rem}
	
	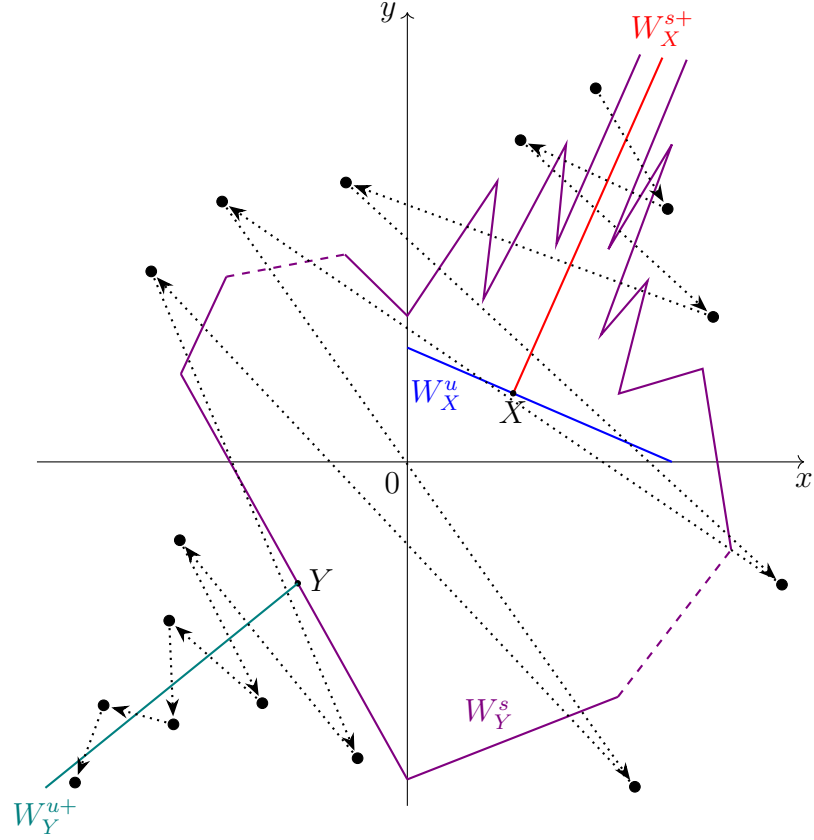
\begin{figure}[!ht]
	\begin{center}
	\begin{tikzpicture}[auto, scale=.7]
			\tikzstyle{nodec}=[draw,circle,fill=black,minimum size=2pt,
			inner sep=0pt, label distance=2mm]
			\tikzstyle{nodeq}=[draw,circle,fill=black,minimum size=4pt,
			inner sep=0pt, label distance=2mm]
			\tikzstyle{nodeh}=[draw,circle,fill=white,minimum size=4pt,
			inner sep=0pt]
			\tikzstyle{dot}=[circle,draw=none,fill=none,minimum size=0pt,inner sep=2pt, outer sep=-1pt]

			\draw[->] (-7,0)--(7.5,0) node [below]{$x$};
			\draw[->] (0,-6.5)--(0,8.5) node [left]{$y$};
			\node[label={[xshift=-0.2cm, yshift=-0.7cm]$0$}] at (0,0) {};
			
			\coordinate (e0) at (0,0);
			\coordinate (ex) at (1,0);
			\coordinate (ey) at (0,1);
			
			\coordinate (yend1) at (7,5.26);
			\coordinate (yend2) at (-7,5.26);
			\coordinate (lm1O) at (0,-4);

			\coordinate (t0) at (5,0);
			\coordinate (t0m1) at (0,2.16);
			\coordinate (t0m2) at (intersection of t0--t0m1 and lm1O--yend1);
			\coordinate (x1) at (2,1);
			\coordinate (x2) at (2,2);
			\coordinate (x) at (intersection of t0--t0m1 and x1--x2);
			\coordinate (xend) at (4.82,7.64);
			
			\draw[blue, thick] (t0)--(x);
			\draw[blue, thick] (x) to node[dot]{$W_X^u$} (t0m1);
			
			\draw[red, thick] (x)--(xend) node[above]{$W_X^{s+}$};

			\node[nodec, color=black, label={[below, color=black]$X$}] at (x) {};
			
			\coordinate (r0) at (0,-6);
			\coordinate (r0m11) at (-4.28,1);
			\coordinate (r0m12) at (-4.28,2);
			\coordinate (r0m1) at (intersection of r0m11--r0m12 and lm1O--yend2);
			
			\coordinate (f1) at (3.98,-4.44);
			\coordinate (f2) at (6.12,-1.66);
			\coordinate (f3) at (5.58,1.76);
			\coordinate (f41) at (4,1);
			\coordinate (f42) at (4,2);
			\coordinate (f4) at (intersection of f41--f42 and lm1O--yend1);
			\coordinate (f5) at (4.54,3.42);
			\coordinate (f6) at (3.66,2.4);
			\coordinate (f7) at (5,6);
			\coordinate (f8) at (3.8,4.02);
			\coordinate (f9) at (5.28,7.6);
			
			\coordinate (g1) at (-3.42,3.5);
			\coordinate (g2) at (-1.18,3.92);
			\coordinate (g3) at (0,2.76);
			\coordinate (g4) at (1.7,5.3);
			\coordinate (g5) at (1.44,3.08);
			\coordinate (g6) at (3,6);
			\coordinate (g7) at (2.82,4.12);
			\coordinate (g8) at (4.4,7.7);
			
			\coordinate (c) at (intersection of e0--ex and f2--f3);
			
			\draw [violet, thick] (r0)--(r0m1);
			\draw [violet, thick] (r0) to node[dot]{$W_Y^s$} (f1);
			\draw [violet, thick] (r0m1)--(g1);
			\draw [violet, thick, dashed] (f1)--(f2);
			\draw [violet, thick, dashed] (g1)--(g2);
			\draw [violet, thick] (f2)--(f3)--(f4)--(f5)--(f6)--(f7)--(f8)--(f9);
			\draw [violet, thick] (g2)--(g3)--(g4)--(g5)--(g6)--(g7)--(g8);

			\coordinate (y1) at (-2.07,1);
			\coordinate (y2) at (-2.07,2);
			\coordinate (y) at (intersection of y1--y2 and r0--r0m1);
			\node[nodec, color=black, label={[right, color=black]$Y$}] at (y) {};
			\coordinate (uyend) at (-6.84,-6.16);
			\draw [teal, thick] (y)--(uyend) node[below]{$W_Y^{u+}$};
			
			\coordinate (q1) at (3.56,7.06); \node[nodeq, color=black] at (q1) {};
			\coordinate (q2) at (4.92,4.78); \node[nodeq, color=black] at (q2) {};
			\coordinate (q3) at (2.14,6.08); \node[nodeq, color=black] at (q3) {};
			\coordinate (q4) at (5.78,2.74); \node[nodeq, color=black] at (q4) {};
			\coordinate (q5) at (-1.16,5.28); \node[nodeq, color=black] at (q5) {};
			\coordinate (q6) at (7.08,-2.32); \node[nodeq, color=black] at (q6) {};
			\coordinate (q7) at (-3.5,4.92); \node[nodeq, color=black] at (q7) {};
			\coordinate (q8) at (4.3,-6.14); \node[nodeq, color=black] at (q8) {};
			\coordinate (q9) at (-4.84,3.6); \node[nodeq, color=black] at (q9) {};
			\coordinate (q10) at (-0.94,-5.6); \node[nodeq, color=black] at (q10) {};
			\coordinate (q11) at (-4.3,-1.48); \node[nodeq, color=black] at (q11) {};
			\coordinate (q12) at (-2.74,-4.56); \node[nodeq, color=black] at (q12) {};
			\coordinate (q13) at (-4.5,-3); \node[nodeq, color=black] at (q13) {};
			\coordinate (q14) at (-4.42,-4.96); \node[nodeq, color=black] at (q14) {};
			\coordinate (q15) at (-5.74,-4.6); \node[nodeq, color=black] at (q15) {};
			\coordinate (q16) at (-6.28,-6.06); \node[nodeq, color=black] at (q16) {};
			
			\draw[-{Stealth[scale=1]}, shorten >=3pt, thick, dotted, bend right] (q1)--(q2);
			\draw[-{Stealth[scale=1]}, shorten >=3pt, thick, dotted, bend right] (q2)--(q3);
			\draw[-{Stealth[scale=1]}, shorten >=3pt, thick, dotted, bend right] (q3)--(q4);
			\draw[-{Stealth[scale=1]}, shorten >=3pt, thick, dotted, bend right] (q4)--(q5);
			\draw[-{Stealth[scale=1]}, shorten >=3pt, thick, dotted, bend right] (q5)--(q6);
			\draw[-{Stealth[scale=1]}, shorten >=3pt, thick, dotted, bend right] (q6)--(q7);
			\draw[-{Stealth[scale=1]}, shorten >=3pt, thick, dotted, bend right] (q7)--(q8);
			\draw[-{Stealth[scale=1]}, shorten >=3pt, thick, dotted, bend right] (q8)--(q9);
			\draw[-{Stealth[scale=1]}, shorten >=3pt, thick, dotted, bend right] (q9)--(q10);
			\draw[-{Stealth[scale=1]}, shorten >=3pt, thick, dotted, bend right] (q10)--(q11);
			\draw[-{Stealth[scale=1]}, shorten >=3pt, thick, dotted, bend right] (q11)--(q12);
			\draw[-{Stealth[scale=1]}, shorten >=3pt, thick, dotted, bend right] (q12)--(q13);
			\draw[-{Stealth[scale=1]}, shorten >=3pt, thick, dotted, bend right] (q13)--(q14);
			\draw[-{Stealth[scale=1]}, shorten >=3pt, thick, dotted, bend right] (q14)--(q15);
			\draw[-{Stealth[scale=1]}, shorten >=3pt, thick, dotted, bend right] (q15)--(q16);

			\end{tikzpicture}
	\end{center}
	\caption{Schematic forward dynamics of a point in $\mathcal{A}_2$, illustrating escape to infinity (Remark \ref{rem:A2_infinity_other_proof}).}
	\label{figure:thm_A2_infinity}
	\end{figure}	
 
	As a consequence, we see that the basin of attraction of $\ell$ is contained in $\mathcal{A}_1$. In fact, our final goal is to show that all points in $\mathcal{A}_1$ except $W_X^s$ tend to $\ell$ under forward iteration of $L_{a,b}$.
	
	To achieve this, we first analyze how forward orbits in $\mathcal{A}_1$ can accumulate on the remaining part of the non-wandering set of $L_{a,b}^2$, namely, the fixed points $X$ and $Y$. The next two lemmas establish, in two steps, that accumulation at a fixed point can only occur along its stable manifold. The first lemma treats the case in which the fixed point is the entire $\omega$-limit set. The second removes this restriction and shows that it suffices for the fixed point merely to belong to the $\omega$-limit set.

	\begin{lem} \label{lem:omega_limit_X_singleton}
	\begin{enumerate}
	\item Let $A \in \mathcal{A}_1$ be a point such that $\omega(A,L_{a,b}) = \{X\}$. Then $A \in W_X^s$.
	\item Let $A \in \Cl\mathcal{A}_1$ be a point such that $\omega(A,L_{a,b}) = \{Y\}$. Then $A \in W_Y^s$.
	\end{enumerate}
	\end{lem}
	
	\begin{proof}
	We will first prove claim (1). Assume by contradiction that $A \notin W_X^s$, that is, $A \in \mathcal{A}_1 \setminus W_X^s$. Let $\mathcal{U} \subset \mathcal{A}_1$ be the region given by \eqref{eq:region_U}. Lemma \ref{lem:U_invariant} implies that there exists $i \in \mathbb{N}_0$ such that $A^{i} \in \mathcal{U}$ and $A^{i_1} \in \mathcal{U}$ for every $i_1 \geqslant i$. Since $\mathcal{U}$ is bounded, and only finitely many points of $\mathcal{O}^{+}(A,L_{a,b})$ lie outside $\mathcal{U}$, the forward orbit $\mathcal{O}^{+}(A,L_{a,b})$ is bounded.
	
	Since $A \notin W_X^s$, there exists $\varepsilon > 0$ and a sequence of positive integers $(n_k)_{k \in \mathbb{N}}$ such that $n_k \rightarrow \infty$ as $k \rightarrow \infty$ and $\dist(A^{n_k},X) \geqslant \varepsilon$ for every $k \in \mathbb{N}$.
	
	Since $\mathcal{O}^{+}(A,L_{a,b})$ is bounded, $\Cl\mathcal{O}^{+}(A,L_{a,b})$ is compact. Therefore, there exists a convergent subsequence $(A^{n_{k_j}})_{j \in \mathbb{N}}$ such that $A^{n_{k_j}} \rightarrow G$ as $j \rightarrow \infty$. Then $G \in \omega(A,L_{a,b})$ but $\dist(G,X) \geqslant \varepsilon$, which contradicts $\omega(A,L_{a,b}) = \{X\}$. Hence, $A \in W_X^s$, which completes the proof of claim (1).
	
	For claim (2), let $A \in \Cl\mathcal{A}_1 = \mathcal{A}_1 \cup W_Y^s$, and assume by contradiction that $A \notin W_Y^s$. Since $X \notin \omega(A,L_{a,b})$, we have that $A \notin W_X^s$, that is, $A \in \mathcal{A}_1 \setminus W_X^s$, and we obtain that $\mathcal{O}^{+}(A,L_{a,b})$ is bounded as in the proof of claim (1). The rest of the proof follows in the same way as the proof of claim (1).  
	\end{proof}

	\begin{lem} \label{lem:omega_limit_X}
	\begin{enumerate}
	\item Let $A \in \mathcal{A}_1$ be a point such that $X \in \omega(A,L_{a,b})$. Then $A \in W_X^s$.
	\item Let $A \in \Cl\mathcal{A}_1$ be a point such that $Y \in \omega(A,L_{a,b})$. Then $A \in W_Y^s$.
	\end{enumerate}
	\end{lem}
	
	\begin{proof}
	We will first prove claim (1). Assume by contradiction that $A \notin W_X^s$, that is, $A \in \mathcal{A}_1 \setminus W_X^s$. As in the proof of Lemma \ref{lem:omega_limit_X_singleton}(1), we obtain that the forward orbit $\mathcal{O}^{+}(A,L_{a,b})$ is bounded, and the same conclusion holds for $\mathcal{O}^{+}(A,L_{a,b}^2)$.
	
	By Lemma \ref{lem:X_Y_isolated_from_ell}, $X$ is isolated from $\ell$, so there exists $\varepsilon > 0$ such that $B_{\varepsilon}(X) \cap \ell = \emptyset$. By possibly taking a smaller $\varepsilon$, we obtain
	\begin{equation}\label{eq:X_isolated_from_ell_Y}
	B_{\varepsilon}(X) \cap (\ell \cup \{Y\}) = \emptyset.
	\end{equation}
	
	By the assumption of the lemma, $X \in \omega(A,L_{a,b})$. Therefore, there exists a sequence of positive integers $(n_k)_{k \in \mathbb{N}}$ such that $A^{n_k} \rightarrow X$ as $k \rightarrow \infty$. We claim that $X \in \omega(A,L_{a,b}^2)$. Indeed, if the sequence $(n_k)_k$ contains infinitely many even terms, the claim holds. If not, then the sequence $(n_k + 1)_k$ contains infinitely many even terms and $A^{n_k + 1} \rightarrow L_{a,b}(X) = X$ as $k \rightarrow \infty$. In either case, $X \in \omega(A, L_{a,b}^2)$.
	
	We claim that $\omega(A,L_{a,b}^2) = \{X\}$. Assume by contradiction that there exists a point $G \in \omega(A,L_{a,b}^2)$ different from $X$. Since $\mathcal{O}^{+}(A,L_{a,b}^2)$ is bounded, it is precompact, so Lemma \ref{lem:omega_ICT} implies that $\mathcal{O}^{+}(A, L_{a,b}^2)$ is internally chain transitive. Therefore, for $\delta := \tfrac{1}{2}\varepsilon$, there exists a $\delta$-chain from $X$ to $G$, that is, a finite sequence of points $I_1,I_2,\ldots,I_n \in \omega(A,L_{a,b}^2)$ such that $I_1 = X$, $I_n = G$, and for every $j \in \{1,2,\ldots,n - 1\}$, we have that
	\begin{equation*}
	\dist(L_{a,b}^2(I_j),I_{j + 1}) < \tfrac{1}{2}\varepsilon.
	\end{equation*}
	In particular, for $j=1$, we have that
	\begin{equation*}
	\dist(L_{a,b}^2(X),I_2) = \dist(X,I_2) < \tfrac{1}{2}\varepsilon.
	\end{equation*}
Since $I_2$ is a point in $\omega(A,L_{a,b}^2) \subseteq \Omega(L_{a,b}^2) \subseteq \ell \cup \{X,Y\}$ (the last inclusion holds by Theorem \ref{thm:L2_non_wandering}), \eqref{eq:X_isolated_from_ell_Y} implies that $I_2 = X$.

	Since $I_2 = X$, the same argument gives $I_3 = X$. Inductively, $I_j = X$ for every $j \in \{1,2,\ldots,n\}$. In particular, $G = I_n = X$, which contradicts $G \neq X$. Therefore, $\omega(A,L_{a,b}^2) = \{X\}$.
	
	We now claim that $\omega(A,L_{a,b}) = \{X\}$. By contradiction, assume that there exists a point $H \in \omega(A,L_{a,b})$, $H \neq X$. Then, there exists a sequence of positive integers $(m_k)_{k \in \mathbb{N}}$ such that $A^{m_k} \rightarrow H$ as $k \rightarrow \infty$. If $(m_k)_{k}$ contains infinitely many even terms, then $H \in \omega(A,L_{a,b}^2)$. Otherwise, $(m_k + 1)_{k}$ contains infinitely many even terms and $A^{m_k + 1} \rightarrow H^1$ as $k \rightarrow \infty$, so $H^1 \in \omega(A,L_{a,b}^2)$. In either case, we obtain a contradiction with $\omega(A,L_{a,b}^2) = \{X\}$.
	
	Therefore, $\omega(A,L_{a,b}) = \{X\}$, and Lemma \ref{lem:omega_limit_X_singleton} implies $A \in W_X^s$, which completes the proof of claim (1).
	
	For claim (2), let $A \in \Cl\mathcal{A}_1 = \mathcal{A}_1 \cup W_Y^s$, and assume by contradiction that $A \notin W_Y^s$. Since $Y \in \omega(A,L_{a,b})$ and $X \neq Y$, the forward orbit of $A$ cannot converge to $X$, so $A \notin W_X^s$, and we again obtain that $\mathcal{O}^{+}(A,L_{a,b})$ and $\mathcal{O}^{+}(A,L_{a,b}^2)$ are bounded. The rest of the proof now follows by exchanging the roles of $X$ and $Y$ in the proof of claim (1). 		 
	\end{proof}

	\begin{theorem} \label{thm:A_1_basin_ell}
	Let $(a,b) \in \mathfrak{R}$. The set $\mathcal{A}_1\setminus W_X^s$ is the basin of attraction of $\ell$, that is, for every point $A \in \mathcal{A}_1 \setminus W_X^s$, we have $\dist(A^n,\ell) \rightarrow 0$ as $n \rightarrow \infty$.
	\end{theorem}
	
	\begin{proof}
	Recall that $\mathcal{A}_1$ is the union of regions $\mathcal{S}_i$, $i \in \{1,2,3,4,5\}$, as introduced in Lemma \ref{lem:A1_invariant}; see Figure \ref{figure:lem_A1_invariant}. In addition, let $\mathcal{U}=\mathcal{S}_2 \cup \mathcal{S}_3 \cup \mathcal{S}_4 \cup \mathcal{S}_5$, as in \eqref{eq:region_U}.
	
	Let $A \in \mathcal{A}_1 \setminus W_X^s$ be an arbitrary point. We claim that $\mathcal{U}$ contains all except at most finitely many forward iterates of $A$ under $L_{a,b}$. Indeed, if $A \in \mathcal{U}$, then Lemma \ref{lem:U_invariant}(1) implies that $\mathcal{O}^{+}(A,L_{a,b}) \subseteq \mathcal{U}$. Otherwise, if $A \in \mathcal{S}_1$, then by Lemma \ref{lem:U_invariant}(2), there exists $j \in \mathbb{N}$ such that $A^{j} \in \mathcal{U}$, and $\mathcal{O}^{+}(A^j,L_{a,b}) \subseteq \mathcal{U}$. In either case, there exists $n_0 \in \mathbb{N}_0$ such that $A^n \in \mathcal{U}$ for all $n \geqslant n_0$. 
	
	Since $\mathcal{U}$ is bounded, $\Cl \mathcal{U}$ is compact, so there exists a subsequence $(A^{n_k})_{k \in \mathbb{N}}$ converging to a point $F \in \Cl \mathcal{U}$. We have that $F \in \omega(A,L_{a,b})$.
	
	If the sequence $(n_k)_{k \in \mathbb{N}}$ contains infinitely many even terms, then $F \in \omega (A,L_{a,b}^2)$ as well. Otherwise, the sequence $(n_k+1)_{k \in \mathbb{N}}$ contains infinitely many even terms, and the corresponding sequence of iterates $(A^{n_k+1})_{k \in \mathbb{N}}$ converges to $F^1$ due to the continuity of $L_{a,b}$. Therefore, by possibly replacing $F$ with its forward image, we can assume that $F \in \omega(A,L_{a,b}^2)$. 
	
	Since
	\begin{equation*}
	F \in \omega(A,L_{a,b}^2) \subseteq \Omega(L_{a,b}^2) \subseteq \ell \cup \{X,Y\},
	\end{equation*}
where the last inclusion holds due to Theorem \ref{thm:L2_non_wandering}, we have the following possibilities:
	\begin{enumerate}[label=(\roman*)]
	\item $F = X$. In this case, $X \in \omega(A,L_{a,b})$, so Lemma \ref{lem:omega_limit_X}(1) would imply that $A \in W_X^s$, which contradicts the assumption that $A \in \mathcal{A}_1 \setminus W_X^s$.
	\item $F = Y$. As in the previous case, $Y \in \omega(A,L_{a,b})$, so Lemma \ref{lem:omega_limit_X}(2) would imply that $A \in W_Y^s$, which is not possible since $W_Y^s \cap \mathcal{A}_1 = \emptyset$.
	\item $F \in \ell$. Here, we have that $F \in \omega(A,L_{a,b}) \cap \ell \neq \emptyset$, so Lemma \ref{lem:approaching_ell} implies that $\dist(A^{n},\ell) \rightarrow 0$ as $n \rightarrow \infty$, which completes the proof.
	\end{enumerate}	   
	\end{proof}
	
	\begin{rem}
	Therefore, for parameters in $\mathfrak{R}$, the regions $\mathcal{A}_1$ and $\mathcal{A}_2$ induce a decomposition of the phase space into a disjoint union
	\begin{equation*}
	\mathbb{R}^2 = \mathcal{A}_1 \cup W_Y^s \cup \mathcal{A}_2,
	\end{equation*}
	which allows a complete classification of the asymptotic behavior of points, summarized as follows. For a point $A \in \mathbb{R}^2$, the forward orbit of $A$ under $L_{a,b}$ falls into exactly one of the following three cases:
	\begin{itemize}
		\item it converges to the fixed point $X$ (resp.\ $Y$) if $A$ lies on $W_X^s$ (resp.\ $W_Y^s$),
		\item it converges to $\ell = \Cl W_X^u \setminus W_X^u$ if $A$ lies in $\mathcal{A}_1 \setminus W_X^{s}$, or
		\item it tends to infinity in the third quadrant if $A$ lies in $\mathcal{A}_2$.
	\end{itemize} 
	\end{rem}

	\section{Concluding remarks} \label{sec:concluding_rems}
	
	Let $\mathcal{X}$ be a metric with a metric $d$, and $f \colon \mathcal{X} \rightarrow \mathcal{X}$. Recall that a non-empty, compact set $\Lambda \subset \mathcal{X}$ is said to be an \emph{attractor} for $f$ if $f(\Lambda) = \Lambda$ and there exists an open neighborhood $\mathcal{N}$ of $\Lambda$ in $\mathcal{X}$ such that
	\begin{equation*}
	\lim_{n \rightarrow \infty}\sup_{x \in \mathcal{N}} d(f^{n}(x),\Lambda) = 0.
	\end{equation*}
	
	In our case, when $(a,b) \in \mathfrak{R}$, the accumulation set $\ell$ arises as the intersection of compact nested sets $(L_{a,b}^{k}(\mathcal{D} \cup L_{a,b}(\mathcal{D})))_{k \in \mathbb{N}}$, so $\ell$ is compact and $L_{a,b}(\ell) = \ell$. Moreover, we have shown that the open region $\mathcal{A}_1 \setminus W_X^s$ is the basin of attraction of $\ell$. Therefore, $\ell$ is an attractor for $L_{a,b}$ in the sense of the above definition.
	
	At the beginning of Section \ref{sec:basin_attraction}, we already mentioned that there exists a parameter region $\mathcal{R} \subset \mathfrak{R}$ in which $\ell = \{P,P'\}$, that is, $\ell$ reduces to an attracting periodic orbit. This equality extends even to some parameters in $\mathfrak{R} \setminus \mathcal{R}$, so it is natural to ask whether the equality holds on the whole $\mathfrak{R}$.
	
	As already discussed in \cite{kilassa2026accumulation}, in a private correspondence, Y.\ Ishii and D.\ Sands pointed the author to examples of parameter values suggesting that $\ell$ might, in general, have a complex, nontrivial structure. The examples include the pair $(a,b) = (1.6,0.61)$, where transverse homoclinic intersections for period-six saddles should occur, and a small neighborhood of $(a,b) = (1.5,0.5)$.
	
	On the other hand, there are already well-known examples of strange attractors for $L_{a,b}$: the existence of these attractors has been rigorously proved in \cite{misiurewicz1980strange} for an open set of parameters which is now known as the Misiurewicz parameter set. Even if our accumulation set $\ell$ possessed a more complex structure, it would still be different from these strange attractors since they are the closure of the unstable manifold $W_X^u$, while $\ell$ is the $\omega$-limit set of that manifold but does not contain it.
	
	Thus, although the present work, together with the previous one in \cite{kilassa2026accumulation}, establishes that $\ell$ is an attractor of the Lozi map and, at the same time, contains all of its entropy-inducing dynamics, the structure of this set remains an open question. A deeper understanding of its geometry and dynamics requires further research and may ultimately lead to the identification of a new class of attractors for Lozi maps, distinct from the examples known to date.


\begin{thebibliography}{99}
		\bibitem{baptista2009basin}
		D.\ Baptista, R. Severino and S. Vinagre, The basin of attraction of Lozi mappings, \emph{Int.\ J.\  Bifurcat.\ Chaos} {\bf 19} (2009), 1043-1049	
	
		\bibitem{boronski2023densely}
		J.\ Boro\'nski, S.\ \v{S}timac, Densely branching trees as models for H\'enon-like and Lozi-like attractors, \emph{Adv.\ Math.}\ {\bf 429} (2023), 109191
		
		\bibitem{cao2000nonwandering}
		Y.\ Cao, J.\ Min Mao, The non-wandering set of some H\'enon maps, \emph{Chaos, Solitons Fractals} {\bf 11} (2000), 2045--2053
		
		\bibitem{hirsch2001chain}
		M.\ W.\ Hirsch, H.\ L.\ Smith, X.-Q.\ Zhao, Transitivity, Attractivity, and Strong Repellors for Semidynamical Systems, \emph{J.\ Dyn.\ Diff.\ Equ.}\ {\bf 13} (2001), 107--131
		
		\bibitem{ishii1997towards1}
		Y.\ Ishii, Towards a kneading theory for Lozi mappings I: A solution of the prunning front conjecture and the first tangency problem, \emph{Nonlinearity} {\bf 10} (1997), 731--747
		
		\bibitem{kilassa2026accumulation}
		K.\ Kilassa Kvaternik, Accumulation sets and zero entropy dynamics in the Lozi map, \emph{Chaos}, \url{https://doi.org/10.1063/5.0344038}
		
		\bibitem{kilassa2026tangential}
		K.\ Kilassa Kvaternik, Tangential Homoclinic Points for Lozi Maps, \emph{J.\ Dyn.\ Diff.\ Equ.}, \url{https://doi.org/10.1007/s10884-026-10505-2}
		
		\bibitem{kilassa2022tangential}
		K.\ Kilassa Kvaternik, Tangential homoclinic points locus of the Lozi maps and applications, Doctoral Thesis, University of Zagreb, Croatia (2022), available online: \url{https://repozitorij.pmf.unizg.hr/islandora/object/pmf:11546}
		
		\bibitem{lozi1978attracteur}
		R.\ Lozi, Un attracteur \'etrange du type attracteur de H\'enon, \emph{J.\ Physique (Paris)} {\bf 39} (Coll. C5) (1978), 9--10
		
		\bibitem{misiurewicz1980strange}
		M.\ Misiurewicz, Strange attractor for the Lozi mappings, \emph{Ann.\ New York Acad.\ Sci.}\ {\bf 357} (1980) (Nonlinear Dynamics), 348--358
		
		\bibitem{misiurewicz2024zero}
		M.\ Misiurewicz, S. \v{S}timac, The zero entropy locus for the Lozi maps, \emph{Monatsh.\ Math.}\ {\bf 208} (2025), 471--484
		
		\bibitem{shub1987global}
		M.\ Shub, \emph{Global stability of dynamical systems}, Springer-Verlag New York, 1987
		
		
		
	\end{thebibliography}
\end{document}